%% file: Main.tex
\documentclass[11pt]{amsart}

\usepackage{amssymb}
\usepackage{epsfig}
\usepackage{comment}
\usepackage{amsmath}
\usepackage{subcaption}
\usepackage[section]{placeins}
\usepackage{mathrsfs}
\usepackage{graphicx}
\usepackage[notrig]{physics}
\usepackage{units}
\usepackage{bm}
\usepackage[hidelinks]{hyperref}
\usepackage{booktabs}

\numberwithin{equation}{section}

\newtheorem{theorem}{Theorem}[section]
\newtheorem{proposition}[theorem]{Proposition}
\newtheorem{example}[theorem]{Example}
\newtheorem{lemma}[theorem]{Lemma}
\newtheorem{corollary}[theorem]{Corollary}
\newtheorem{definition}[theorem]{Definition}
\newtheorem{remark}[theorem]{Remark}

\newcommand{\R}{\mathbb{R}}
\newcommand{\diver}{\operatorname{div}}
\newcommand{\loc}{\mathrm{loc}}
\newcommand{\sym}{\mathrm{sym}}

\newcommand{\cons}{\mathrm{cons}}
\newcommand{\geom}{\mathrm{geom}}
\newcommand{\K}{\mathsf{K}}
\newcommand{\Cth}{\mathcal C}
\newcommand{\A}{\mathcal A}
\newcommand{\D}{\mathcal D}
\newcommand{\E}{\mathcal E}
\newcommand{\Z}{\mathcal Z}

\newcommand{\I}{\mathcal I}
\newcommand{\J}{\mathcal J}
\newcommand{\M}{\mathcal M}
\newcommand{\Lag}{\mathcal L}
\newcommand{\Gmap}{g}

\def\grad{{\nabla}}
\def\teste{{\delta e}}
\def\testu{{\delta u}}
\def\tests{{\delta s}}
\def\testl{{\delta \lambda}}
\def\testm{{\delta \mu}}
\def\k{\mathsf{k}}

\usepackage{xcolor}
\newcommand{\Ext}{\mathscr F}

\allowdisplaybreaks

\begin{document}

\title[Data-driven porous media]{Existence of thermodynamically consistent solutions for data-driven porous media problems}

\author{R. Codina}
\address{R. Codina:
Department of Civil and Environmental Engineering,
Universitat Polit\`ecnica de Catalunya (UPC),
Jordi Girona 1--3,
08034 Barcelona, Spain;
and International Centre for Numerical Methods in Engineering (CIMNE),
Edifici C1, Campus Nord UPC,
C/ Gran Capit\`a, S/N,
08034 Barcelona, Spain}
\email{ramon.codina@upc.edu}

\author{C. G. Gebhardt}
\address{C. G. Gebhardt:
Geophysical Institute and Bergen Offshore Wind Centre,
University of Bergen,
All\'{e}gaten 70,
5007 Bergen, Norway}
\email[corresponding author]{cristian.gebhardt@uib.no}

\author{M. Ortiz}
\address{M. Ortiz:
Division of Engineering and Applied Science, 
California Institute of Technology,
1200 East California Boulevard,
Pasadena, CA 91125, USA}
\email{ortiz@caltech.edu}

\address{M. Ortiz:
International Centre for Numerical Methods in Engineering (CIMNE)
and Universitat Polit\`ecnica de Catalunya (UPC),
Edifici C1, Campus Nord UPC,
C/ Gran Capit\`a, S/N,
08034 Barcelona, Spain}

\begin{abstract}
Data-Driven Computational Mechanics (DDCM) replaces traditional phenomenological constitutive models by directly reformulating boundary-value problems in terms of local material state data obtained from experiments or fine-scale simulations. Standard DDCM formulations, only enforcing equilibrium and compatibility, do not inherently guarantee compliance with the second law of thermodynamics. This breakdown occurs particularly when input material data sets are subject to noise or local physical non-admissibility. In this work, we present a variational DDCM framework specifically tailored to diffusion--reaction problems. Taking advantage of the simplicity of the thermodynamic constraint in gradient-flux systems, we propose an augmented formulation that explicitly enforces the second law of thermodynamics as a hard constraint within the energy-minimization problem. Although the set of thermodynamically admissible states is non-convex and fails to be weakly closed in the ambient phase space, we establish existence of minimizers by proving that the intersection of the admissible set with the subspace of fields that are compatible and in equilibrium is weakly sequentially closed via a compensated compactness argument. To enable practical computations, we analyze both a Lagrange multiplier formulation and a penalization scheme. We prove the $\Gamma$-convergence of the penalized functionals to the exact constrained problem and establish a fully discrete convergence framework incorporating spatial finite-element discretization and empirical data-set approximations. Numerical experiments confirm that the proposed penalty scheme effectively restores thermodynamic consistency even in the presence of severely corrupted material data.
\end{abstract}

\date{\today}



\maketitle

\section{Introduction}

In recent years, Data-Driven Computational Mechanics (DDCM) has emerged as a novel computational paradigm that reformulates classical bound\-ary-value problems arising in computational physics by directly incorporating material data obtained through experiments or numerical simulations, thereby bypassing the traditional step of constructing empirical phenomenological constitutive models. The standard unconstrained data-driven formulation \cite{KirchdoerferOrtiz2016} seeks local material states---such as stress-strain pairs in solid mechanics or gradient-flux pairs in transport problems---that satisfy exact geometric compatibility and conservation laws while minimizing an energy-equivalent distance function to a given empirical data set in phase space. By disconnecting the fundamental conservation and compatibility conditions (which are exact physical constraints) from material behavior (which is purely empirical), DDCM eliminates modeling errors, parameter calibration uncertainties, and empirical bias. The paradigm has since been extended well beyond its original static, noise-free setting, including to computational dynamics~\cite{KirchdoerferOrtiz2018}, to material data sets corrupted by experimental noise~\cite{KirchdoerferOrtiz2017}, and to strongly nonlinear material responses~\cite{GaletzkaLoukrezisDeGersem2021}. The question of existence of solutions for such data-driven optimization problems has also been addressed, mainly in the context of data-driven elasticity~\cite{CMO2018,ContiMullerOrtiz2020,ContiHoffmannOrtiz2023Inference,ContiHoffmannOrtiz2023Rates,GebhardtSteinbach2026}. 

From a variational standpoint, the continuous data-driven problem corresponds to an optimization problem over a phase-space functional constrained by the material-independent feasible set containing equilibrium (conservation), compatibility, and boundary conditions. The constraint can be enforced by introducing an indicator function in the functional to be minimized or, in some cases, by using Lagrange multipliers. 

Equilibrium and compatibility are non-negotiable physical constraints grounded in the axioms of classical physics; \emph{the second law of thermodynamics is equally fundamental}. However, standard DDCM formulations do not automatically preserve thermodynamic admissibility. This breakdown occurs either because the input data set is corrupted by experimental or numerical noise, or because complex and noisy data distributions lead to distance-minimization solutions that locally violate the second law, even when all individual data points adhere to it.

The primary objective of this study is twofold. First, we establish a comprehensive, variational DDCM framework tailored to transport phenomena in porous media---a formulation that extends naturally to general diffusion--reaction systems, though we adopt the terminology of flow in porous media throughout for clarity. 
Second, taking advantage of the straightforward geometric interpretation of thermodynamic constraints in this setting, we propose an augmented formulation that explicitly enforces the second law of thermodynamics as a hard constraint within the minimization problem, thereby guaranteeing the physical admissibility of the computed field solutions. The closest antecedent is the recent PDE-constrained DDCM formulation for diffusion--reaction problems presented in~\cite{BGBA-2025,CABG-2026}, where the second-law inequality was formulated but subsequently omitted from the analytical treatment in order to obtain a linear saddle-point problem. The present work addresses precisely this outstanding issue. We incorporate the thermodynamic inequality into the continuous variational problem, establish weak sequential closedness of the physically admissible feasible set through compensated compactness, and develop both an exact multiplier formulation and a penalized approximation.

These objectives dictate the overall structure of the paper. Following the formal problem statement in Section 2, Section 3 details the adaptation of the DDCM formulation to flow in porous media without enforcing the second law of thermodynamics---a baseline framework we refer to as the unconstrained DDCM problem. This section delivers the complete theoretical foundation, including main theorems, auxiliary results, and the approximation scheme for the empirical material data set. Special attention is given to the degenerate case of a single-element data set and to the enforcement of equilibrium and compatibility conditions via Lagrange multipliers, which serves as the core formulation for our numerical experiments. Next, Section 4 addresses the spatial discretization of the optimization problem using the finite element method.

Section 5 introduces the explicit enforcement of the second law of thermodynamics as a constraint within the optimization problem. The primary technical challenge lies in the topology of the phase space: the set $\mathcal{C}$ of thermodynamically admissible states is non-convex and, more critically, fails to be weakly closed under the natural topology of the phase space. To overcome this limitation, we redefine the admissible set as the intersection of $\mathcal{C}$ with the space of states $\mathcal{E}$  that satisfy exact equilibrium and compatibility. By leveraging a compensated compactness argument, we demonstrate that this restricted set $\mathcal{E} \cap \mathcal{C}$  is indeed weakly closed, which, in turn, establishes the existence of minimizers that strictly abide by the second law of thermodynamics.

Having established these core results, Section 6 explores the enforcement of the second law of thermodynamics using a Lagrange multiplier formulation. We demonstrate that if a Lagrange multiplier exists---solving the standard saddle-point system and fulfilling the corresponding complementarity conditions---the resulting primal state is a thermodynamically consistent minimizer of the DDCM problem. However, ensuring the existence of such a multiplier poses significant mathematical challenges, as the multiplier space is neither bounded nor compact. To address this issue, we adopt a game-theoretic approach, showing that restricting the multipliers to a bounded domain renders the saddle-point problem equivalent to a penalized formulation with a constant penalty parameter. This insight underpins the penalty-based enforcement scheme introduced in Section 7, where we prove that the penalized problem is well-posed and $\Gamma$-converges to the original constrained problem as the penalty parameter tends to infinity. Finally, Section 8 concludes the theoretical framework by establishing a unified analysis that simultaneously incorporates material data set approximation, spatial discretization via finite elements, and penalty-based thermodynamic enforcement.

Following the theoretical development, Section 9 presents a numerical example to evaluate the performance of the proposed method. The results demonstrate that the penalty formulation successfully restores thermodynamic consistency as the penalty parameter increases, even when the input data set is noisy or contains local violations of the second law of thermodynamics. Finally, Section 10 summarizes our main findings and offers concluding remarks.

\section{Classical formulation of porous media problems} \label{xuwzhr}

We begin by reviewing the classical setting for porous media. 

Let \(\Omega\subset\R^N\), \(N\in\{2,3\}\), be a bounded, connected, open
Lipschitz domain with boundary \(\Gamma:=\partial\Omega\). Let
\(\Gamma_D,\Gamma_N\subset\Gamma\) be relatively open, disjoint parts of the
boundary such that, up to \(\mathcal H^{N-1}\)-null sets,
        $\Gamma={\Gamma}_D\cup{\Gamma}_N$.
We prescribe hydraulic head \(g\) on \(\Gamma_D\) and normal flux \(h\) on
\(\Gamma_N\). 

The primary field is a scalar potential $u:\Omega\to\R$, interpreted as hydraulic head, pressure, or chemical potential depending on the application. The geometric and conservation laws for porous media are
\begin{subequations} \label{fvZ47W}
\begin{align}
        e(x) &= \grad u(x), &&\text{in }\Omega; \label{eq:pm-geom}\\
        \diver s(x)+\zeta(x)u(x)-q(x) &=0, &&\text{in }\Omega; \label{eq:pm-cons}\\
        u(x)&=g(x), &&\text{on }\Gamma_D, \label{eq:pm-dir}\\
        s(x)\cdot\nu(x)&=h(x), &&\text{on }\Gamma_N, \label{eq:pm-neu}
\end{align}
\end{subequations}
where $e=\grad u$ is the hydraulic-gradient variable and \(s\) is the Darcy flux. In addition, \(q\) is a source term, \(\zeta \geq 0\) a.e. is a reaction or storage coefficient, and \(\nu\) is the outward unit normal. 

The field equations (\ref{fvZ47W}) are closed by Darcy's law
\begin{equation} \label{s88Hp6} 
        s(x)=-\K e(x),
\end{equation}
where, for simplicity, \(\K\) is a uniform hydraulic-conductivity independent of \(x\). We further assume that  $\K\in\R^{N\times N}_{\sym}$ is a fixed symmetric positive-definite matrix. The isotropic specialization is \(\K=\k\, I\) with \(\k > 0\). 

Central to the discussion that follows is the property that, by virtue of these assumptions, the hydraulic gradient $e$ and the Darcy flux $s$ satisfy the pointwise sign condition
\begin{equation} \label{5L2sTj}
        e(x)\cdot s(x)\le 0 ,
\end{equation}
which is the physical requirement of positive dissipation demanded by the second law of thermodynamics \cite{HassanizadehGray1980, Coussy2004}. We note that, at the present level of description, (\ref{5L2sTj}) is a direct consequence of the assumed properties of \(\K\) and need not be enforced as a separate constraint.

Conditions for the well-posedness of the classical boundary value problem (BVP) (\ref{s88Hp6}),  (\ref{fvZ47W}) are well-known and are summarized in Theorem~\ref{thm:classical-darcy-wellposedness} below. We write
\[
        H^1_D(\Omega):=\{v\in H^1(\Omega):v=0\text{ on }\Gamma_D\},
\]
and, for an admissible trace \(g \in H^{1/2}(\Gamma_D)\),
\[
        H^1_g(\Omega):=\{v\in H^1(\Omega):v=g\text{ on }\Gamma_D\}.
\]
$\Ext g \in  H^1_g(\Omega)$ will denote a fixed lifting of $g \in H^{1/2}(\Gamma_D)$.

For vector fields \(s\in H(\diver,\Omega;\R^N)\), the normal trace \(s\cdot\nu\in H^{-1/2}(\partial\Omega)\) is understood through Green's formula:
\[
        \int_\Omega s\cdot \grad v\,\dd x
        +\int_\Omega (\diver s)v\,\dd x
        =
        \langle s\cdot\nu,v\rangle_{\partial\Omega}
\]
for all \(v\in H^1(\Omega)\). Here and below, we will use the notation $\langle \cdot , \cdot\rangle_\omega$ to denote the duality between $H^{-1/2}(\omega)\times H^{1/2}(\omega)$ when $\omega\subseteq \partial\Omega$ and, when $\omega = \Omega$, the duality between $H^{-1}(\Omega)\times H^{1}_D(\Omega)$; see~\cite{KurulaZwart2012} for a detailed account of this duality between the gradient and divergence operators on bounded Lipschitz domains.

The variational form of problem \eqref{fvZ47W}--\eqref{s88Hp6} admits two functional settings: the primal and the dual formulations. Except in the case of minimum regularity---determined by the regularity of $q$ and $\Omega$---both formulations yield the same unique solution at the continuous level. However, the choice of one formulation over the other has a significant impact on the numerical approximation. Although the dual formulation is more commonly used in the porous media community, here we restrict our attention to the primal formulation. Thus, the variational form of the problem we consider reads as follows: given $q\in H^{-1}(\Omega)$, $g \in H^{1/2}(\Gamma_D)$ and $h\in H^{-1/2}(\Gamma_N)$, find $(u, e, s)\in X := H^1_g(\Omega) \times L^2(\Omega;\mathbb{R}^N) \times L^2(\Omega;\mathbb{R}^N )$ such that
\begin{subequations} \label{eq:threefield}
\begin{alignat}{3}
 (\K e , \teste )  + ( s, \teste)      & \, = 0 \quad &&\forall \teste \in L^2(\Omega;\mathbb{R}^N),\label{eq:threefield-const}\\
 (e , \tests)                              - (\grad u , \tests)  & \, = 0 \quad &&\forall \tests \in L^2(\Omega;\mathbb{R}^N),\label{eq:threefield-geom}\\
                         -(s , \grad \testu )  + (\zeta u , \testu) & \,=   \langle q , \testu\rangle_\Omega - \langle h , \testu\rangle_{\Gamma_N} \quad &&\forall \testu \in H^1_D(\Omega),
                         \label{eq:threefield-bal}\
\end{alignat}
\end{subequations}
where $(\cdot,\cdot)$ stands for the $L^2(\Omega)$ inner product (of scalars or vectors). In the functional setting considered, this formulation is exactly equivalent to find $u \in H^1_g(\Omega) $ such that 
\begin{align}\label{eq:irred}
(\K\grad u , \grad \testu ) + (\zeta u , \testu ) =  \langle q , \testu\rangle_\Omega - \langle h , \testu\rangle_{\Gamma_N} ~&&\forall \testu \in H^1_D(\Omega).
\end{align}

In the phase-space formulation below, where fluxes $s$ are only in \(L^2(\Omega;\R^N)\), the Neumann datum is encoded weakly in the balance equation.

\begin{theorem}[Classical well-posedness of the Darcy problem]
\label{thm:classical-darcy-wellposedness}
Let $\Omega\subset\mathbb{R}^N$, $N\in\{2,3\}$, be a bounded Lipschitz domain, and assume that $\Gamma_D$ has positive $(N-1)$-dimensional measure. Let
\[
    \K\in\mathbb{R}^{N\times N}_{\rm sym}, \quad
    \K\xi\cdot\xi\ge \alpha |\xi|^2
    \quad \forall \xi\in\mathbb{R}^N
\]
for some $\alpha>0$, and let $\zeta\in L^\infty(\Omega)$ with $\zeta\ge 0$ a.e. in $\Omega$. Assume that $q\in H^{-1}(\Omega)$, $h\in H^{-1/2}(\Gamma_N)$, and that 
$g \in H^{1/2}(\Gamma_D)$.
Then, there exists a unique $u\in H^1_g(\Omega)$ solution to problem \eqref{eq:irred} that satisfies the stability estimate:
\[
    \|u\|_{H^1(\Omega)}
    \le
    C\Big(
    \|q\|_{H^{-1}(\Omega)}
    +
    \|h\|_{H^{-1/2}(\Gamma_N)}
    +
    \|  g\|_{H^{1/2}(\Gamma_D)}
    \Big),
\]
where $C>0$ depends only on $\Omega$, $\Gamma_D$, $\alpha$, $\K$,
and $\|\zeta\|_{L^\infty(\Omega)}$.
\end{theorem}

Extension of the classical BVP to non-uniform conductivities is straightforward, provided that $\K\in L^\infty(\Omega; \R^{N\times N}_{\sym})$ and there exist constants \(0<\alpha\le \beta<\infty\) such that
\[
        \alpha |\xi|^2\le \K\xi\cdot\xi\le \beta |\xi|^2 ,
        \quad\text{for all }\xi\in\R^N .
\]
However, this extension of the theory will not be pursued here in the interest of simplicity.

\section{Data-Driven unconstrained porous media problem}

In anticipation of an empirical characterization of the material, where the local material behavior is characterized by a sample of observed local states, we recast the preceding classical formulation of Darcy flow in Data-Driven form. The key idea is to abandon the constitutive law \eqref{eq:threefield-const} and to replace it by closeness to observed local states of the material. 

For now, we neglect the thermodynamic constraint (\ref{5L2sTj}) and return to it in Section~\ref{sec:thermo-constraint}. 

\subsection{The general Data-Driven porous media problem}

We identify the local phase variable with
\[
        z(x)=(e(x),s(x))\in \Z_{\loc}:=\R^N\times\R^N,
\]
where $e=\grad u$ is the hydraulic-gradient variable, \(s\) is the Darcy flux and $\Z_{\loc}$ is the local phase space of all local states of the material. We define the global phase space as
\[
        \Z:=L^2(\Omega;\R^N)\times L^2(\Omega;\R^N),
        \quad z=(e,s) ,
\]
and metrize \(\Z\) by the energy-equivalent norm
\[
        \|z\|_\Z^2
        := \frac12
        \int_\Omega
        \left(
        \K e\cdot e
        +
        \K^{-1}s\cdot s
        \right)\dd x .
\]
When $\K$ is not available, it can be replaced by $\k\, I$, with $\k > 0$ any appropriate scaling coefficient that renders the above expression dimensionally consistent.

The conservation-law feasible set is imposed in weak form,
\[
\begin{split}
        \E_{\cons}
        :=
        \big\{
        (u,e,s)\in X:\;& \hbox{equation \eqref{eq:threefield-bal} holds}
        \big\}.
\end{split}
\]
The geometric feasible set is
\[
        \E_{\geom}
        :=
        \left\{
        (u,e,s)\in X:  \hbox{equation \eqref{eq:threefield-geom} holds}
        \right\}.
\]
The corresponding material-independent constraint set in phase space is
\begin{equation}
        \E
        :=
        \left\{
        (e,s)\in \Z:
        \exists u\in H^1_g(\Omega)
        \text{ s.t. }
        (u,e,s)\in \E_{\cons}\cap\E_{\geom}
        \right\}.
        \label{eq:E-unconstrained}
\end{equation}
Thus, \(\E\) contains only compatibility, equilibrium, and boundary conditions. The thermodynamic inequality is not part of \(\E\) in this section.

Suppose that the conductivity of the solid at all material points $x \in \Omega$ is characterized by a local material data set \(\D_{\loc} \subset \R^N\times\R^N\) of all the observed local states of the material. We note that \(\D_{\loc}\) is allowed to be a general set, e.g.~a point set, not necessarily a graph. The corresponding global material data set is
\[
        \D
        :=
        \{z=(e,s)\in \Z: z(x)\in \D_{\loc}\text{ for a.e. }x\in\Omega\}.
\]
For \(z=(e,s)\in \Z\), the distance to the material set is
\[
        d(z,\D):=\inf_{y\in \D}\|z-y\|_\Z.
\]
It bears emphasis that \(\D\) is not required to be weakly closed in \(\Z\). Indeed, a point set \(\D_{\loc}\) is generally nonconvex and fine-scale oscillations in $\Omega$ can converge weakly to values outside \(\D\), see \cite{CMO2018} for suitable notions of closure \(\bar{\D}\) of the data set.

With given phase space \(\Z\), constraint set \(\E\subset \Z\), and material data set \(\D\subset \Z\) as defined, the Data-Driven problem is
\begin{equation} \label{A6DVP2}
        \operatorname*{argmin}_{z\in \E} d^2(z,\D),
\end{equation}
or, equivalently,
\begin{equation} \label{8mP9Dp}
        \operatorname*{argmin}_{z\in \Z} \J(z) ,
\end{equation}
where
\begin{equation} \label{gBmLMc}
        \J(z) := \I_\E(z)+d^2(z,\D) ,
\end{equation}
and, here and subsequently,
\[
        \I_\A(z)
        :=
        \begin{cases}
        0, & z\in \A,\\
        +\infty, & z\notin \A.
        \end{cases}
\]
is the indicator function of a subset $\A \subset \Z$.

\subsection{Existence of Data-Driven solutions}

We begin by verifying that the feasible set \(\E\) is stable under weakly convergent perturbations. 

\begin{lemma}[Weak closedness of the feasible set]
Assume \(\zeta\in L^\infty(\Omega)\). Then, \(\E\subset \Z\) is weakly sequentially closed.
\end{lemma}

\begin{proof}
Let \(z_j=(e_j,s_j)\in \E\) and \(z_j\rightharpoonup z=(e,s)\) weakly in \(\Z\). Choose \(u_j\in H^1_g(\Omega)\) with \(e_j=\grad u_j\) and \((u_j,e_j,s_j)\in\E_{\cons}\cap\E_{\geom}\). Since \((e_j)\) is bounded in \(L^2(\Omega)\), Poincare's inequality applied to \(u_j-\Ext g\in H^1_D(\Omega)\) gives boundedness of \((u_j)\) in \(H^1(\Omega)\). Passing to a subsequence, \(u_j\rightharpoonup u\) in \(H^1(\Omega)\), with \(u\in H^1_g(\Omega)\). The continuity of the gradient gives \(e=\grad u\). For every \(\testu\in H^1_D(\Omega)\), the weak balance passes to the limit because \(s_j\rightharpoonup s\) in \(L^2(\Omega)\) and \(u_j\rightharpoonup u\) in \(L^2(\Omega)\). Hence \((u,e,s)\in\E_{\cons}\cap\E_{\geom}\), and therefore \(z\in\E\).
\end{proof}

We note the following work-energy identity. 

\begin{lemma}[Porous-media power identity]
Let \(z=(e,s)\in\E\), let \(u\in H^1_g(\Omega)\) be an associated potential, and assume that \(s\in H(\diver,\Omega;\R^N)\) so that the full normal trace is defined. Then
\begin{equation}
        \int_\Omega e\cdot s\,\dd x
        =
        \int_\Omega \zeta u^2\,\dd x
        -
        \int_\Omega q u\,\dd x
        +
        \langle s\cdot\nu,g\rangle_{\Gamma_D}
        +
        \langle h,u\rangle_{\Gamma_N}.
        \label{eq:power-id}
\end{equation}
In particular, if \(g=0\) on \(\Gamma_D\), the Dirichlet-reaction term drops out.
\end{lemma}

\begin{proof}
Since \(e=\grad u\), Green's identity gives
\[
        \int_\Omega e\cdot s\,\dd x
        =
        \int_\Omega \grad u\cdot s\,\dd x
        =
        -\int_\Omega u\,\diver s\,\dd x
        +
        \langle s\cdot\nu,u\rangle_{\Gamma}.
\]
Using \(\diver s=q-\zeta u\), \(u=g\) on \(\Gamma_D\), and \(s\cdot\nu=h\) on \(\Gamma_N\) gives \eqref{eq:power-id}.
\end{proof}
This result is only used in Proposition~\ref{prop:ellip} with $g = 0$, so that in fact we never require $s\cdot\nu$ to be defined on $\Gamma_D$ and therefore we do not need \(s\in H(\diver,\Omega;\R^N)\).

Existence of Data-Driven solutions then follows from the following adaptation of Tonelli's theorem, see Appendix~\ref{TepXTQ}.

\begin{theorem}[Existence of Data-Driven minimizers]
\label{thm:general-dd-existence}
Let \(Z\) be a reflexive separable Banach space and let \(\E,\D\subset Z\) be nonempty. Assume that \(\E\) and \(\D\) are weakly sequentially closed. Suppose, in addition, that there exist constants \(c>0\) and \(b\geq 0\) such that the following transversality condition holds:
\begin{equation}
\label{eq:transversality-static}
  \|y-z\|_Z \geq c\bigl(\|y\|_Z+\|z\|_Z\bigr)-b
  \quad
  \forall\, y\in \D,\ \forall\, z\in \E .
\end{equation}
Then the Data-Driven problem (\ref{A6DVP2}), or (\ref{8mP9Dp}), admits at least one solution.
\end{theorem}

\begin{proof}
Let \((z_j)\subset \E\) be a minimizing sequence for \(d^2(\cdot,\D)\).
For every \(j\), choose \(y_j\in\D\) such that
\[
  \|z_j-y_j\|_Z \leq d(z_j,\D)+\frac1j .
\]
Since \(d(z_j,\D)\) is bounded along the minimizing sequence, the
transversality estimate \eqref{eq:transversality-static} gives boundedness
of both \((z_j)\) and \((y_j)\) in \(Z\). By reflexivity, after passing to a
subsequence,
\[
  z_j \rightharpoonup z \quad\text{and}\quad y_j \rightharpoonup y
  \quad\text{weakly in } Z .
\]
The weak sequential closedness of \(\E\) and \(\D\) gives \(z\in\E\) and
\(y\in\D\). By weak lower semicontinuity of the norm,
\[
  d(z,\D)
  \leq \|z-y\|_Z
  \leq \liminf_{j\to\infty}\|z_j-y_j\|_Z
  \leq
  \liminf_{j\to\infty}\left(d(z_j,\D)+\frac1j\right).
\]
Since \((z_j)\) is minimizing, the last term is the minimum value of the distance problem. Thus \(z\) minimizes \(d^2(\cdot,\D)\) over \(\E\) or, equivalently,
\(z\) minimizes \(\J\) over \(Z\).
\end{proof}

We recall that material data sets built from local point sets are typically not weakly closed. In such cases Theorem~\ref{thm:general-dd-existence} does not apply directly and, in order to ensure existence, \(\D\) must be replaced by a suitable relaxation \cite{CMO2018}, restrict attention to approximate minimizers or pass to a convergent family of approximating data sets in the sense of Mosco convergence, see Section~\ref{fcfUFg}.

We also note that \(\E\) and \(\D\) may have zero distance without having a common point. Equivalently, \(\operatorname{dist}(\E,\D)=0\) does not, by itself, imply \(\E \cap \D \neq \emptyset\). Moreover, when \(\D\) is not weakly closed, minimizing sequences may converge weakly only to a relaxed limit, rather than to an element of the original material data set.

If the infimum of \(d^2(\cdot,\D)\) over \(\E\) is attained at \(z_\ast\in\E\), then \(d(z_\ast,\D)=\operatorname{dist}(\E,\D)\). Thus \(z_\ast\) is a best feasible approximation to the material data set in the sense of the distance functional. If, in addition, \(\E\cap\D=\emptyset\) and \(\operatorname{dist}(\E,\D)>0\), then every attained minimizer has strictly positive residual: \(d(z_\ast,\D)>0\). If the distance from \(z_\ast\) to \(\D\) is also attained, then there exists \(y_\ast\in\D\) such that \(\|z_\ast-y_\ast\|_\Z = \operatorname{dist}(\E,\D)\).

So far, we have considered a general data set $\D$, without any particular algebraic structure. If $\D$ is a closed linear subspace of $\Z$, there is a case in which there is a unique minimizer of $\J$ and the minimum attained by this functional is zero. This corresponds precisely to the case in which the relationship between $e$ and $s$ is linear, i.e., to Darcy's law.

\begin{corollary}[Existence and zero residual under a direct-sum condition]
\label{cor:direct-sum-existence}
Assume that the material-independent set is affine, \(\E=z_0+\E_0\), where \(\E_0\subset \Z\) is a closed linear subspace. Assume also that \(\D\subset \Z\) is a closed linear subspace and that \(\Z=\E_0\oplus \D\). Then the Data-Driven problem (\ref{A6DVP2}), or (\ref{8mP9Dp}) has a unique minimizer \(z_\ast\). Moreover,
\[
  z_\ast\in \E\cap\D,
  \quad
  \text{and}
  \quad
  \min_{\Z}\J=0 .
\]
\end{corollary}

\begin{proof}
Since \(\Z=\E_0\oplus\D\), any element \(z_0\in\Z\) has a unique decomposition \(z_0=e_0+d\), with \(e_0\in\E_0\) and \(d\in\D\). Then, \(d=z_0-e_0\in z_0+\E_0=\E\), and also \(d\in\D\). Hence \(\E\cap\D\neq\emptyset\), so \(\inf_{z\in\E} d^2(z,\D)=0\). If \(z_1,z_2\in\E\cap\D\), then \(z_1-z_2\in\E_0\cap\D\). The direct-sum condition gives \(\E_0\cap\D=\{0\}\), hence \(z_1=z_2\). Therefore, \(\E\cap\D=\{z_\ast\}\), and the minimizer is unique.
\end{proof}

\begin{remark}[Relation with transversality]
\label{rem:direct-sum-transversality}
If \(\E_0\) and \(\D\) are closed subspaces and \(\Z=\E_0\oplus\D\), then the canonical projections onto \(\E_0\) and \(\D\) are bounded. Hence there is \(C>0\) such that
\[
        \|e_0\|_\Z+\|d\|_\Z\le C\|e_0+d\|_\Z,
        \quad e_0\in\E_0,\ d\in\D.
\]
Equivalently, after replacing \(d\) by \(-d\), the direct-sum condition yields the transversality estimate used in Theorem~\ref{thm:general-dd-existence} with $b=0$.
\end{remark}

A detailed discussion of complementary properties of \(\D\) and \(\E\) in finite dimensions may be found in \cite{ContiHoffmannOrtiz2023Inference, ContiHoffmannOrtiz2023Rates}.

\subsection{Data-Driven reformulation of linear Darcy porous media}

We verify, by way of sanity test, that the classical solutions of the Darcy BVP are recovered from the Data-Driven reformulation. For linear Darcy flow, the local material data set is the fixed graph
\begin{equation} \label{6eDLHt}
        \D_{\loc}
        :=
        \{(e,s)\in\R^N\times\R^N: s=-\K e\},
\end{equation}
and $\D = \{ (e,s)\in \Z : (e(x) , s(x)) \in \D_{\loc} ~ a.e. \}$.
The unconstrained Data-Driven porous-media problem is then (\ref{8mP9Dp}) specialized to this particular choice of $\D_{\loc}$. 

Given $z = (e,s)\in \Z$, a direct pointwise minimization gives
\begin{align}
        d^2(z,\D)
        & = \inf_{ (\tilde e , \tilde s)\in \D} \Vert (e -\tilde e , s -\tilde s) \Vert_{\Z}\nonumber \\
        & = \frac14 \int_\Omega
        \K^{-1}(s+\K e)\cdot(s+\K e)\,\dd x,
        \label{eq:darcy-distance}
\end{align}
which is attained for $(\tilde e , \tilde s) = \frac12 ( e  - \K^{-1}s ,  s - \K e)$.

In the case of a homogeneous problem ($q=0$, $g = 0$, $h= 0$), we have an explicit representation of the distance of any function satisfying \eqref{eq:threefield-geom} and the homogeneous counterpart of \eqref{eq:threefield-bal} to pairs satisfying Darcy's law:

\begin{lemma}[Homogeneous porous-media estimate] \label{YqnSXF}
Let
\begin{align*}
        \E_0:= \{ & (e,s)\in \Z:\ 
        \exists v\in H^1_D(\Omega)\text{ s.t. } \\ &
        (e ,\tests) - (\grad v ,\tests) = 0 ~ \forall \tests\in L^2(\Omega;\mathbb{R}^N), \\ & 
        -(s, \grad \testu) + (\zeta v , \testu)  = 0 ~ \forall \testu\in H^1_D(\Omega)\}.
\end{align*}
Then, for every \(z_0=(e_0,s_0)\in \E_0\),
\begin{equation} \label{AjEz9x}
        d^2(z_0,\D)
        \ge
        \frac12\,\|z_0\|_\Z^2.
\end{equation}
\end{lemma}

\begin{proof} Let $(e_0,s_0) \in \E_0$ and $v_0\in H^1_D(\Omega)$ the element whose existence characterizes $\E_0$. Using the first and the second properties successively, with $\tests = s_0$ and $\testu = v_0$, we get:
\begin{align*}
(e_0 , s_0) & = (\grad v_0 , s_0) = (\zeta v_0 , v_0) \geq 0. 
\end{align*}
Using now \eqref{eq:darcy-distance},
\begin{align*}
        d^2(z_0,\D)
        &=
        \int_\Omega
        \frac14\,\K^{-1}(s_0+\K e_0)\cdot(s_0+\K e_0)\,\dd x \\
        &=
        \frac14
        \int_\Omega
        \left(
        \K^{-1}s_0\cdot s_0
        +
        \K e_0\cdot e_0
        +
        2s_0\cdot e_0
        \right)\dd x  \\
        &\ge
        \frac14
        \int_\Omega
        \left(
        \K^{-1}s_0\cdot s_0
        +
        \K e_0\cdot e_0
        \right)\dd x
        =
        \frac12\,\|z_0\|_\Z^2 ,
\end{align*}
which proves (\ref{AjEz9x}).
\end{proof}

In addition, sequential weak lower semicontinuity of \(\mathcal J\) follows from the weak sequential closedness of \(\E\), which renders \(\mathcal I_\E\) weakly sequentially lower semicontinuous, and from the weak lower semicontinuity of \(d^2(\cdot,\D)\), since \(\D\) is a closed linear subspace of the Hilbert space \(\Z\). Existence of solutions of the Darcy Data-Driven problem (\ref{8mP9Dp}) then follows from Theorem~\ref{thm:general-dd-existence}.

It remains to verify that the Data-Driven minimizers are the classical solutions. 

\begin{theorem}[Recovery of the constant-conductivity Darcy problem] \label{2Zy8sW}
Assume that \(\K\) is a fixed symmetric positive-definite matrix, \(\zeta\in L^\infty(\Omega)\) with \(\zeta\ge0\) a.e., and that the boundary conditions and data \(q,g,h\) are such that the weak Darcy problem \eqref{eq:threefield} is well posed. Then, the unconstrained Data-Driven problem \eqref{8mP9Dp} has a unique solution that coincides with that of \eqref{eq:threefield}.

\end{theorem}

\begin{proof} Let $(u,e,s)\in X$ be the solution of problem \eqref{eq:threefield}. Then 
\(u\in H^1_g(\Omega)\) and
\[
        e=\grad u,
        \quad
        s=-\K e \quad \hbox{in} ~ L^2(\Omega;\mathbb{R}^N),
\]
from where \(z=(e,s)\in \E\cap \D\). Hence \(d(z,\D)=0\), so \(z\) is a minimizer
of \eqref{8mP9Dp}.

Conversely, if \(\bar z = ( \bar e, \bar s)\in \E\) is any minimizer of \eqref{8mP9Dp}, then
\[
        0\le d(\bar z,\D)\le d(z,\D)=0,
\]
and therefore \(\bar z\in \D\). Hence \(\bar z=(\grad \bar u,-\K\grad\bar u)\)
for some \(\bar u\in H^1_g(\Omega)\), and \(\bar u\) solves the weak Darcy
problem \eqref{eq:irred}, and hence $(\bar u, \bar e, \bar s)$ solves problem \eqref{eq:threefield}. Uniqueness of the weak Darcy solution gives \(\bar u=u\), and hence
\(\bar z=z\).
\end{proof}

\subsection{Approximation of the material data set} \label{fcfUFg}

We now consider a sequence \((\D_h)\) of porous-media material data sets in \(\Z\), for example generated by laboratory measurements, pore-scale simulations, or homogenized microstructure calculations. The abstract convergence results in this subsection are unchanged from the Data-Driven elasticity framework \cite{CMO2018}; only the interpretation of the phase variables changes from strain--stress pairs to gradient--flux pairs.

We use tools of the calculus of variations for ascertaining the properties of problem (\ref{gBmLMc}), see Appendix~\ref{fCnmub}. We note that the weak topology of \(\Z\) is metrizable on bounded subsets.  Indeed, \(\Z = L^2(\Omega;\mathbb{R}^N) \times L^2(\Omega;\mathbb{R}^N)\) is a separable Hilbert space and $\Z' = \Z$, hence its dual is separable, and the weak topology is metrizable on bounded sets; see, for example, Brezis~\cite[Ch.~3]{Brezis2011}.  Therefore, we shall use throughout the sequential characterization of \(\Gamma\)-convergence on bounded sublevel sets, see Lemma~\ref{Sur2ng}.

For every \(h \in \mathbb{N}\), set
\[
        \J_h(z):=\I_\E(z)+d^2(z,\D_h),
        \quad
        \J(z):=\I_\E(z)+d^2(z,\D).
\]
Then, a \emph{mutatis mutandis} adaptation of \cite{CMO2018} gives the following.

\begin{theorem}[Mosco convergence and uniform transversality imply convergence of Data-Driven problems {\cite{CMO2018}}] \label{udw8S6}
Let \(\Z\) be a reflexive separable Banach space, let \(\D,\D_h\subset \Z\), and
let \(\E\subset \Z\) be weakly sequentially closed. Suppose:
\begin{enumerate}
\item \(\D=M\hbox{-}\lim_{h\to\infty}\D_h\) in \(\Z\);
\item there exist constants \(c>0\) and \(b\ge0\) such that, for every
\(y\in \D_h\) and \(z\in \E\),
\begin{equation}
        \|y-z\|_\Z\ge c\left(\|y\|_\Z+\|z\|_\Z\right)-b.
        \label{eq:equi-trans}
\end{equation}
\end{enumerate}
Then
\[
        \J
        =
        \Gamma\hbox{-}\lim_{h\to\infty} \J_h
\]
with respect to the weak topology of \(\Z\). Moreover, any sequence
\((z_h)\subset \Z\) such that
\[
        \sup_h \J_h(z_h) <+\infty
\]
has a weakly convergent subsequence whose limit belongs to \(\E\).
\end{theorem}

The following corollary shows how Theorem~\ref{thm:general-dd-existence} is recovered from Theorem~\ref{udw8S6}. 

\begin{corollary}[Existence for weakly closed data sets {\cite{CMO2018}}]
Let \(\Z\) be a reflexive separable Banach space, and let \(\D,\E\subset \Z\) be weakly sequentially closed. Suppose that the transversality estimate \eqref{eq:equi-trans} holds with \(\D_h=\D\). Then, the Data-Driven problem
\[
        \operatorname*{argmin}_{z\in \E} d^2(z,\D)
\]
has solutions.
\end{corollary}

Next we record the standard consequence of Theorem~\ref{udw8S6} for minimizing sequences. 

\begin{corollary}[Convergence of approximate minimizers] \label{auBXLU}
Assume the hypotheses of Theorem~\ref{udw8S6}. Let \(\varepsilon_h\downarrow 0\),
and let \((z_h)\subset \Z\) be an approximate minimizing sequence in the
sense that
\[
        \J_h(z_h)\le \inf_{z\in \Z}\J_h(z)+\varepsilon_h .
\]
Assume that the limiting problem has finite value, i.e.
\[
        \inf_{z\in \Z}\J(z)<+\infty .
\]
Then \((z_h)\) is relatively weakly sequentially compact in \(\Z\). Every
weak cluster point \(z_\ast\) of \((z_h)\) is a minimizer of the limiting
Data-Driven problem
\[
        \min_{z\in \Z} \J(z)
        =
        \min_{z\in \E} d^2(z,\D).
\]
Moreover,
\[
        \lim_{h\to\infty}\inf_{z\in \Z}\J_h(z)
        =
        \min_{z\in \Z}\J(z).
\]
In particular, if the limiting problem has a unique minimizer \(z_\ast\), then the entire sequence converges weakly, \(z_h \rightharpoonup z_\ast\) in \(\Z\).
\end{corollary}

\begin{proof}
By Theorem 3.8, \(\J_h\) \(\Gamma\)-converges to \(\J\) with respect to the weak topology of \(\Z\), and bounded-energy sequences are relatively weakly sequentially compact. The assertion is therefore the standard compactness and convergence-of-quasi-minimizers property of \(\Gamma\)-convergence, applied in the weak topology; see Theorem A.10.
\end{proof}

Finally, we record the porous-media version of the standard local-data approximation result. The result is the same Mosco-convergence statement as in \cite{CMO2018}, with local states \(\xi=(e,s)\in\R^N\times\R^N\).

\begin{lemma}[Fine and uniform approximation {\cite{CMO2018}}] \label{7YAxr2}
Let \(\Z=L^2(\Omega;\R^N)\times L^2(\Omega;\R^N)\), and let
\[
        \D_h
        =
        \{z\in \Z:z(x)\in \D_{\loc,h}(x)\text{ for a.e. }x\in\Omega\},
\]
where \(\D_{\loc,h}(x)\subset\R^N\times\R^N\). Let
\[
        \D
        =
        \{z\in \Z:z(x)\in \D_{\loc}\text{ for a.e. }x\in\Omega\},
\]
with the fixed constant-conductivity graph
\[
        \D_{\loc}
        =
        \{(e,s)\in\R^N\times\R^N:s=-\K e\}.
\]
Assume that:
\begin{enumerate}
\item there exists \(\rho_h\downarrow0\) such that
\[
        d_x(\xi,\D_{\loc,h}(x))\le \rho_h
        \quad
        \forall \xi\in \D_{\loc}
        \quad\text{for a.e. }x\in\Omega;
\]
\item there exists \(t_h\downarrow0\) such that
\[
        d_x(\xi,\D_{\loc})\le t_h
        \quad
        \forall \xi\in \D_{\loc,h}(x)
        \quad\text{for a.e. }x\in\Omega.
\]
\end{enumerate}
Here \(d_x\) is the pointwise distance induced by the local version of the constant \(\K\)-weighted norm; in this specialization the metric itself is independent of \(x\). Then,
\[
        \D=M\hbox{-}\lim_{h\to\infty}\D_h
        \quad\text{in }\Z.
\]
\end{lemma}

\subsection{A particular case: $\D$ consisting of a single element}\label{sec:single-el}

Let $(\tilde e , \tilde s) \in \Z$ and suppose that $\D = \{ (\tilde e , \tilde s) \}$, i.e., $\D$ is made of a single element that may be synthesized from a larger set of data. In this case,
\begin{align*}
d^2 ((e,s) , \D) = \frac12 \k \Vert e -\tilde e\Vert^2_{L^2(\Omega;\R^N)} +  \frac12 \k^{-1}\Vert s -\tilde s\Vert^2_{L^2(\Omega;\R^N)} , \quad (e,s) \in \Z,
\end{align*} 
where $\k$ is now a physical scaling factor. 

Let us consider for simplicity the case $g = 0$, $h = 0$, the general case being recovered by a standard shifting. Rather than prescribing the feasibility condition $(e,s)\in \E$ introducing the indicator function as in \eqref{gBmLMc}, we can prescribe it using Lagrange multipliers $(\lambda,\mu) \in X_\lambda \times X_\mu$, the former to impose the balance equation \eqref{eq:pm-cons} and the latter to prescribe the geometric condition \eqref{eq:pm-geom}. Spaces $X_\lambda$ and $X_\mu$ are specified below. The problem to be solved is then
\begin{align}
\arg \inf_{(u, e, s) \in X } \sup_{(\lambda , \mu) \in X_\lambda \times X_\mu} {\mathcal K} ( (u, e, s) , (\lambda , \mu) ),\label{eq:probl-min-max}
\end{align} 
where
\begin{align*}
{\mathcal K} ( (u, e, s) , (\lambda , \mu) ) = d^2 ((e,s) , \D) + \langle \lambda ,  \diver s + \zeta u - q\rangle_\lambda + \langle \mu , e - \grad u \rangle_\mu. 
\end{align*} 
At this point, spaces $X_\lambda$ and $X_\mu$ and the dualities $\langle \cdot ,  \cdot \rangle_\lambda$  and $\langle \cdot , \cdot \rangle_\mu$ need to be specified depending on the functional setting adopted for the problem. In reference~\cite{BGBA-2025} the primal form of the problem was analyzed and approximated using a Galerkin/least-squares method, whereas in reference~\cite{CABG-2026} both the primal and the dual approaches were considered and approximated using a Variational Multiscale (VMS) formulation. 

As for the general case, we will restrict our attention to the primal form of the problem. In this case, $X_\lambda = H^1_D(\Omega)$, $X_\mu = L^2(\Omega;\R^N)$ and
\begin{align}
{\mathcal K} ( (u, e, s) , (\lambda , \mu) ) = d^2 ((e,s) , \D) + \langle \lambda ,  \diver s + \zeta u - q\rangle_\Omega + ( \mu , e - \grad u ). \label{eq:5field-f}
\end{align} 

Let $\testl\in H^1_D(\Omega)$ and $\testm\in L^2(\Omega;\R^N)$ be the test functions for $\lambda$ and $\mu$, respectively. Problem \eqref{eq:probl-min-max} is then equivalent to solve the system of variational equations: find $(u,e,s)\in X$ and $(\lambda,\mu) \in H^1_D(\Omega) \times L^2(\Omega;\R^N)$ such that
\begin{subequations} \label{eq:five-field}
\begin{alignat}{3}
( \zeta \lambda , \testu) - (\mu , \grad\testu) & = 0 \qquad && \forall \testu \in H^1_D(\Omega), \label{eq:five-field-u}\\
\k ( e , \teste )  + (\mu , \teste ) & = \k ( \tilde e , \teste )\qquad && \forall \teste \in L^2(\Omega; \R^N), \label{eq:five-field-e}\\
\k^{-1} ( s, \tests ) - (\grad \lambda , \tests ) & = \k^{-1} ( \tilde s, \tests) \qquad && \forall \tests \in L^2(\Omega; \R^N), \label{eq:five-field-s}\\
( \zeta  u ,\testl ) - (s, \grad \testl ) & = \langle q ,\testl \rangle_\Omega   \qquad && \forall \testl \in H^1_D(\Omega), \label{eq:five-field-l}\\
- (\grad u , \testm) + (e , \testm) & = 0 \qquad && \forall \testm \in L^2(\Omega; \R^N). \label{eq:five-field-m}
\end{alignat} 
\end{subequations}

Instead of the abstract transversality condition required in Theorem~\ref{thm:general-dd-existence}, now one can show that this problem is inf-sup stable, and therefore that the solution exists and is unique~\cite{BGBA-2025}:

\begin{theorem}[Existence and uniqueness when $\D$ consists of a single element]  Let $M = H^1_D(\Omega)\times L^2(\Omega; \R^N)$ and $Y = X \times M$. For each $0 \not = (u, e, s, \lambda,\mu)\in Y $ there exists $0 \not = (\testu, \teste, \tests, \testl, \testm) \in Y$ such that
\begin{align}
 & ( \zeta \lambda , \testu) - (\mu , \grad\testu) 
+ \k ( e , \teste )  + (\mu , \teste )  \nonumber \\
& + \k^{-1} ( s, \tests ) - (\grad \lambda , \tests )  
+ ( \zeta  u ,\testl ) - (s, \grad \testl )  - (\grad u , \testm) + (e , \testm) \nonumber \\
&  \geq c \Vert (u, e, s, \lambda,\mu) \Vert_Y \Vert (\testu, \teste, \tests, \testl, \testm) \Vert_Y \label{eq:inf-sup-five}
\end{align}
where $c > 0$  depends on $\Omega$, $\k$ and the function $0 \leq \zeta \in L^\infty(\Omega)$. Furthermore, the bilinear form in the left-hand-side is continuous in $Y$. As a consequence, for any $q\in H^{-1}(\Omega)$, $\tilde e\in  L^2(\Omega; \R^N)$ and  $\tilde s\in  L^2(\Omega; \R^N)$, problem~\eqref{eq:five-field} admits a unique solution $(u, e, s, \lambda,\mu)\in Y $.
\end{theorem}

Interestingly, well-posedness of the dual form depends on an appropriate choice of the scaling factor $\k$ (see \cite{CABG-2026}).

\section{Approximation by spatial discretization}

The convergence statement in Corollary~\ref{auBXLU} is independent of how the approximate minimizers of the Data-Driven functionals are generated. In computations, they are typically obtained by a method of discretization of the spatial domain $\Omega$, such as the finite-element method (FEM), and by a discrete enforcement of the field equations. Evidently, an interplay between the two types of approximation, data and spatial, is to be expected. In this section, we investigate conditions ensuring convergence of the combined approximations.

\subsection{Spatial discretization}

Let \((\Z_k)_{k\in\mathbb N}\) be finite-dimensional subspaces of \(\Z\) satisfying the strong density property
\[
        \forall z\in \Z,\quad
        \exists z_k\in \Z_k
        \quad\text{such that}\quad
        z_k\to z
        \quad\text{strongly in } \Z ,
\]
corresponding, e.g., to a finite-element discretization of $\Omega$. Let \((\E_k)_{k\in\mathbb N}\) be the corresponding discrete material-independent feasible sets, with \(\E_k\subset \Z_k\). The set \(\E_k\) consists of finite-element phase fields satisfying the discrete geometric equations, discrete conservation laws, and discrete boundary conditions. Concrete constructions of discrete spaces meeting these requirements for Darcy-type saddle-point problems---either through compatible mixed interpolations or through stabilized equal-order formulations---can be found in~\cite{B-boffi-et-al,MasudHughes2002,BadiaCodina2010,BurmanHansboLarson2023}.

As we shall see, the correct consistency assumption is the variational convergence 
\[
        \E=M\hbox{-}\lim_{k\to\infty}\E_k
        \quad\text{in } \Z .
\]
Equivalently, the following two conditions hold, see Appendix~\ref{fCnmub}, Def.~\ref{XjU2TR}:
\[
\begin{aligned}
&\text{if } z_k\in \E_k,\quad z_k\rightharpoonup z
        \quad\text{weakly in } \Z,
        \quad\text{then } z\in \E,\\
&\text{for every } z\in \E,\quad
        \exists z_k\in \E_k
        \quad\text{such that}\quad
        z_k\to z
        \quad\text{strongly in } \Z .
\end{aligned}
\]
The first condition is the discrete compactness, or weak consistency, property. It says that weak limits of discrete solutions satisfy the continuous field equations. The second condition is the finite-element recovery property. It says that every continuously feasible state can be approximated strongly by discrete feasible states.

The spatially discretized problem is now 
\begin{equation} 
        \inf_{z\in \Z} \J_{h,k}(z) ,
\end{equation}
where
\begin{equation} 
        \J_{h,k}(z) := \I_{\E_k}(z)+d^2(z,\D_h) .
\end{equation}
The convergence properties of the combined approximation are ascertained next. 

\begin{theorem}[Fully discrete compactness and convergence] \label{LHUWen}
Assume that
\begin{equation} \label{CCK2LV}
        \D=M\hbox{-}\lim_{h\to\infty}\D_h
        \quad\text{and}\quad
        \E=M\hbox{-}\lim_{k\to\infty}\E_k
        \quad\text{in } \Z .
\end{equation}
Assume also the uniform transversality estimate
\[
        \|y-z\|_{\Z}
        \ge
        c\bigl(\|y\|_{\Z}+\|z\|_{\Z}\bigr)-b
        \quad
        \forall y\in \D_h,\ \forall z\in \E_k ,
\]
with constants \(c>0\) and \(b\ge 0\) independent of \(h\) and \(k\). Let \(k=k(h)\to\infty\), and define \(\widehat{\J}_h:=\J_{h,k(h)}\). Then,
\[
        \J=\Gamma\hbox{-}\lim_{h\to\infty}\widehat{\J}_h
\]
with respect to the weak topology of \(\Z\). Moreover, every sequence \((z_h)\subset\Z\) satisfying
\[
        \sup_h \widehat{\J}_h(z_h)<+\infty
\]
has a weakly convergent subsequence in \(\Z\), and every weak cluster point belongs to \(\E\).
\end{theorem}

\begin{proof}
The statement is the standard stability of \(\Gamma\)-convergence under Mosco convergence of the constraint sets and of the material data sets. The Mosco convergence of \(\E_k\) gives the liminf stability of the discrete field equations and the strong recovery sequences for feasible limits. The Mosco convergence of \(\D_h\) gives the corresponding liminf and recovery properties for the distance-to-data term. The uniform transversality estimate gives equicoercivity, hence weak sequential compactness of bounded-energy sequences. The conclusion then follows from the sequential characterization of \(\Gamma\)-convergence on bounded sublevel sets and from the compactness principle recalled in Appendix~\ref{fCnmub}.
\end{proof}

We see from (\ref{CCK2LV}), that there are two distinct approximation mechanisms at play: approximation of the material data set; and approximation of the field equations. The first convergence controls the distance-to-data term. The second controls the discretization of compatibility, equilibrium, and boundary conditions. The fully discrete problem combines both effects through \(\J_{h,k}\). The strong density of the spaces \(\Z_k\) and the strong recovery property of the sets \(\E_k\) are consistency assumptions. They ensure that the finite-dimensional problems do not lose admissible limiting states. This scheme covers the usual finite-element situation in which the discrete fields do not satisfy the continuous equations exactly, but their residuals vanish in a variationally consistent sense. Theorem~\ref{LHUWen} then sets forth conditions under which both types of approximations result in the overall convergence of the scheme.

A direct application of Corollary~\ref{auBXLU} now gives the following. 

\begin{corollary}[Convergence of finite-element approximate minimizers] \label{5Hx8sa}
Assume the hypotheses of the preceding theorem. Let \(k=k(h)\to\infty\) as \(h\to\infty\), and let \((z_h) \subset \Z\) be approximate minimizers of the fully discrete problems,
\[
        \J_{h,k(h)}(z_h)
        \le
        \inf_{z\in\Z}\J_{h,k(h)}(z)+\varepsilon_h,
        \quad
        \varepsilon_h\downarrow 0 .
\]
Then \((z_h)\) is relatively weakly sequentially compact in \(\Z\). Every weak cluster point \(z_\ast\) minimizes the limiting Data-Driven functional $\J$. Moreover,
\[
        \lim_{h\to\infty}
        \inf_{z\in\Z}\J_{h,k(h)}(z)
        =
        \min_{z\in\Z}\J(z),
\]
provided the right-hand side is finite. If the limiting problem has a unique minimizer \(z_\ast\), then the whole sequence converges weakly, \(z_h\rightharpoonup z_\ast\) \quad\text{in } \(\Z\).
\end{corollary}

\subsection{Optimal diagonal strategies}

Corollary~\ref{5Hx8sa} establishes the convergence of fully discrete diagonal sequences defined by the mapping $k(h)$. There is no intrinsic optimal diagonal in the purely qualitative Mosco/\(\Gamma\)-convergence framework, unless the data have additional structure. 

An example of structured data is furnished by Lemma~\ref{7YAxr2}, which sets forth an explicit scale for the material-data approximation error in the case in which the exact data set corresponds to Darcy's law. The two one-sided local errors lift to global distance estimates of order
\[
        R_h:=|\Omega|^{1/2}\rho_h,
        \quad
        T_h:=|\Omega|^{1/2}t_h.
\]
More precisely, for every \(z\in \Z\),
\[
        d(z,\D_h)\le d(z,\D)+R_h,
        \quad
        d(z,\D)\le d(z,\D_h)+T_h.
\]
Thus, on bounded sublevel sets,
\[
        \bigl|d(z,\D_h)-d(z,\D)\bigr|
        \lesssim
        |\Omega|^{1/2}(\rho_h+t_h),
\]
where $\lesssim$ stands for $\leq$ up to positive constants independent of the discretization. 
Consequently, the natural material-data error scale for the functional
\(\J_h\) is
\[
        a_h:=|\Omega|^{1/2}(\rho_h+t_h).
\]

Let \(b_k\downarrow 0\) denote a finite-element consistency scale for the
discrete feasible sets \(\E_k\to\E\). For instance, \(b_k\) may be chosen so
that, on bounded sets,
\[
        \inf_{z_k\in\E_k}\|z_k-z\|_{\Z}\lesssim b_k ,
        \quad
        z\in\E,
\]
A weaker, value-level consistency condition is
\[
        0\le
        \inf_{\Z}\J_{h,k}-\inf_{\Z}\J_h
        \lesssim b_k .
\]
The diagonal sequence may then be chosen by balancing the discretization and data error, i.e., \(k(h)\) is the smallest integer for which
\[
        b_{k(h)}
        \lesssim
        |\Omega|^{1/2}(\rho_h+t_h) .
\]
This choice ensures that the finite-element error is no larger than the material-data error. Thus, if \(k(h)\) is chosen too small, the computed error is dominated by the spatial discretization. Conversely, if \(k(h)\) is chosen much larger than required by this inequality, the discrete field equations are solved more accurately than warranted by material data fidelity. By this balancing criterion, the optimal diagonal sequence represents the coarsest growth for which the finite-element consistency error is asymptotically no larger than the material-data error.

\begin{example}[Algebraic finite-element rates] {\rm
Suppose that the finite-element consistency scale satisfies \(b_k\simeq \Delta_k^p\), where \(\Delta_k\) is the mesh size and \(p>0\) is the approximation order. Then the balancing rule becomes
\[
        \Delta_{k(h)}^p
        \lesssim
        |\Omega|^{1/2}(\rho_h+t_h).
\]
If the number of degrees of freedom satisfies \(N_k\simeq \Delta_k^{-N}\), then the degree-of-freedom scale is
\begin{equation} \label{9cJTrB}
        N_{k(h)}
        \gtrsim
        \bigl(|\Omega|^{1/2}(\rho_h+t_h)\bigr)^{-N/p}.
\end{equation}
In practice, one chooses the smallest \(N_{k(h)}\), or equivalently the coarsest mesh, satisfying this estimate.} \hfill$\square$
\end{example}

\section{Thermodynamic inequality as a strong constraint}\label{sec:thermo-constraint}

We now impose the second-law inequality after the unconstrained theory has been
set up. For \(z=(e,s)\in \Z\), define
\[
        \Gmap(z):=e\cdot s .
\]
Since \(e,s\in L^2(\Omega;\R^N)\), Hölder's inequality gives
\(\Gmap(z)\in L^1(\Omega)\). The thermodynamic constraint set is
\begin{equation}
        \Cth
        :=
        \left\{
        z=(e,s)\in \Z:
        \Gmap(z)=e\cdot s\in L^1(\Omega),\
        e\cdot s\le0\text{ a.e. in }\Omega
        \right\}.
        \label{eq:Cth}
\end{equation}
Equivalently, \(\Cth=\Gmap^{-1}(L^1_-(\Omega))\), where
\[
        L^1_-(\Omega):=\{w\in L^1(\Omega): w\le0\text{ a.e.}\}.
\]
The constrained feasible set is $\E\cap\Cth$. The constrained Data-Driven functional is
\begin{equation} 
        \J_{\Cth}(z)
        :=
        \I_\E(z)+d^2(z,\D)+\I_{\Cth}(z)
        =
        \I_{\E\cap\Cth }(z)+d^2(z,\D),
        \label{eq:J-C}
\end{equation}
and the constrained Data-Driven problem is
\begin{equation} \label{FJaQzp}
        \operatorname*{argmin}_{z\in \Z} \J_\Cth(z) ,
\end{equation}
which replaces (\ref{8mP9Dp}).

\begin{remark}[Relation with the integral power identity]
If \(z=(e,s)\in \E\cap\Cth\) has the trace regularity required for the power identity, and \(u\) is an associated potential, then \eqref{eq:power-id} gives
\[
        \int_\Omega \zeta u^2\,\dd x
        \le
        \int_\Omega q u\,\dd x
        -
        \langle h,u\rangle_{\Gamma_N}
        -
        \langle s\cdot\nu,g\rangle_{\Gamma_D}.
\]
For homogeneous Dirichlet data, the last term is absent. This bound is a consequence of the pointwise inequality and is not used in the unconstrained development of Section~\ref{xuwzhr}. Here we will use it only in Proposition~\ref{prop:ellip} with $g= 0$, $h=0$, so that  $s$ does not need to have a well-defined normal trace.
\end{remark}

\begin{theorem}[Recovery of the constrained Darcy solution]
Under the assumptions of the recovery Theorem~\ref{2Zy8sW}, the constrained problem (\ref{FJaQzp}) has the same unique solution as the unconstrained Darcy problem. 
\end{theorem}

\begin{proof}
The unconstrained Darcy solution satisfies \(s=-\K e\), hence
\[
        e\cdot s=-\K e\cdot e\le0
        \quad\text{a.e. in }\Omega.
\]
Thus, the unconstrained minimizer belongs to \(\E\cap\Cth\) and has zero distance
to \(\D\). Every constrained minimizer has distance at least zero and therefore
also has distance zero. The recovery argument in Section~\ref{xuwzhr} then implies that it
is the unique weak Darcy solution.
\end{proof}

\subsection{Is the constraint set weakly closed?}

The set \(\Cth\) is not convex. Indeed, for any nonzero \(\xi\in\R^N\), the
constant states \((3\xi,-\xi)\) and \((-\xi,3\xi)\) belong to \(\Cth\), but their
average \((\xi,\xi)\) does not. More importantly, \(\Cth\) is not weakly closed
in the ambient phase space \(\Z=L^2(\Omega;\R^N)\times L^2(\Omega;\R^N)\).

\begin{lemma}[The pointwise sign set is not weakly closed in \(\Z\)]
The set \(\Cth\) defined by \eqref{eq:Cth} is not weakly sequentially closed in
\(\Z\).
\end{lemma}

\begin{proof}
Choose a fixed non-zero vector $\xi \in \R^N$. Let $y \mapsto \chi(y)$ be the $1$-periodic extension of the characteristic function $\chi_{(0,1/2)}$ on $[0,1)$. For each $j \in \mathbb{N}$, define the rapidly oscillating sequence of scalar functions on $\Omega$:
\[
a_j(x) := 2 \chi(j x_1), \quad b_j(x) := 2 \left(1 - \chi(j x_1)\right), \quad \hbox{for} ~ x = (x_1,\dots,x_N) \in \Omega.
\]
Equivalently, $a_j(x) = 2 \chi_{(0, 1/2)}(j x_1 \bmod 1)$ and $b_j(x) = 2 \chi_{(1/2, 1)}(j x_1 \bmod 1)$.
Then
\[
        a_j\rightharpoonup 1,
        \quad
        b_j\rightharpoonup 1
        \quad\text{weakly in }L^2(\Omega).
\]
Set
\[
        e_j=a_j\xi,
        \quad
        s_j=b_j\xi .
\]
Since \(a_jb_j=0\) a.e.,
\[
        e_j\cdot s_j=0\quad\text{a.e.},
\]
and therefore \(z_j=(e_j,s_j)\in\Cth\) for every \(j\). However
\(z_j\rightharpoonup z=(\xi,\xi)\) weakly in \(\Z\), and
\[
        \xi\cdot\xi=|\xi|^2>0.
\]
Thus, \(z\notin\Cth\). Hence \(\Cth\) is not weakly sequentially closed in
\(\Z\).
\end{proof}

The preceding lemma shows that the abstract weak-closedness hypotheses in the
Data-Driven existence theorem cannot be applied to \(\Cth\) alone. However, the
physically relevant set is not \(\Cth\) by itself but \(\E\cap\Cth\). On the
material-independent set \(\E\), the fields have differential structure: \(e\) is
a gradient and \(s\) satisfies a divergence equation. This restores weak
closedness by compensated compactness \cite{Murat1978,Tartar1979}.

\begin{theorem}[Weak closedness of \(\E\cap\Cth\) by compensated compactness]
Assume \(\zeta\in L^\infty(\Omega)\). Then \(\E\cap\Cth\) is weakly sequentially closed in \(\Z\).
\end{theorem}

\begin{proof}
Let \(z_j=(e_j,s_j)\in \E\cap\Cth\) and suppose
\(z_j\rightharpoonup z=(e,s)\) weakly in \(\Z\). By weak closedness of \(\E\),
there exists \(u\in H^1_g(\Omega)\) such that \(e=\grad u\) and
\(\diver s+\zeta u-q=0\). Choose associated potentials
\(u_j\in H^1_g(\Omega)\) with \(e_j=\grad u_j\) and
\(\diver s_j+\zeta u_j-q=0\). As in the proof of weak closedness of \(\E\), the
sequence \((u_j)\) is bounded in \(H^1(\Omega)\); hence, up to a subsequence,
\[
        u_j\rightharpoonup u\text{ in }H^1(\Omega),
        \quad
        u_j\to u\text{ in }L^2(\Omega).
\]
Therefore
\[
        \diver s_j=q-\zeta u_j
        \to
        q-\zeta u=\diver s
        \quad\text{strongly in }H^{-1}(\Omega).
\]
Moreover \(\operatorname{curl} e_j=0\) in distributions, since \(e_j=\grad u_j\). We use the distributional div-curl lemma \cite{Murat1978, Tartar1979} in the form: if \(e_j\rightharpoonup e\) and \(s_j\rightharpoonup s\) in \(L^2(\Omega)\), while \(\operatorname{curl}e_j\) and \(\diver s_j\) are precompact in \(H^{-1}(\Omega)\), then \(e_j\cdot s_j\to e\cdot s\) in \(\mathcal D'(\Omega)\). Its hypotheses are satisfied here, and hence
\[
        e_j\cdot s_j \to e\cdot s
        \quad\text{in }\mathcal D'(\Omega).
\]
For every nonnegative \(\varphi\in C_c^\infty(\Omega)\), the pointwise
constraint gives
\[
        \int_\Omega \varphi\, e_j\cdot s_j\,\dd x\le0.
\]
Passing to the distributional limit gives
\[
        \int_\Omega \varphi\, e\cdot s\,\dd x\le0
        \quad\text{for every }\varphi\ge0.
\]
Thus, \(e\cdot s\le0\) in the sense of distributions and therefore a.e. in
\(\Omega\). Since \(e,s\in L^2(\Omega)\), \(e\cdot s\in L^1(\Omega)\). Hence \(z\in\Cth\), and
therefore \(z\in \E\cap\Cth\).
\end{proof}

\begin{proposition}[Ellipsoidal outer constraint] \label{prop:ellip}
Assume \(g=0\) on \(\Gamma_D\), \(h=0\) on \(\Gamma_N\), \(q\in L^2(\Omega)\), and \(\zeta\in L^\infty(\Omega)\) with \(\zeta\ge \zeta_0>0\) a.e. If \(z=(e,s)\in\E\cap\Cth\) and \(u\in H^1_D(\Omega)\) is its potential, then
\[
        \int_\Omega \zeta\left(u-\frac{q}{2\zeta}\right)^2\,\dd x
        \le
        \int_\Omega \frac{q^2}{4\zeta}\,\dd x.
\]
Thus the pointwise thermodynamic condition implies a closed convex, \(L^2(\Omega)\)-ellipsoidal outer constraint on \(u\). The converse need not hold.
\end{proposition}

\begin{proof}
With homogeneous boundary data, \eqref{eq:power-id} and \(e\cdot s\le0\) give
\[
        \int_\Omega \zeta u^2\,\dd x-
        \int_\Omega q u\,\dd x\le0.
\]
Completing the square gives the claim.
\end{proof}

\begin{corollary}[Existence and convergence with the thermodynamic indicator]
Let \(\D,\D_h\subset \Z\) and assume \(\D=M\hbox{-}\lim_{h\to\infty}\D_h\). Suppose the
transversality estimate \eqref{eq:equi-trans} holds for every \(y\in \D_h\) and
\(z\in \E\cap\Cth \). Then
\[
        \I_{\E\cap\Cth }(\cdot)+d^2(\cdot,\D)
        =
        \Gamma\hbox{-}\lim_{h\to\infty}
        \left(\I_{\E\cap\Cth }(\cdot)+d^2(\cdot,\D_h)\right)
\]
with respect to the weak topology of \(\Z\). In particular, if \(\D\) is weakly
sequentially closed and the estimate holds with \(\D_h=\D\), then the constrained
Data-Driven problem \eqref{FJaQzp} has minimizers.
\end{corollary}

\begin{proof}
The compensated-compactness theorem gives weak sequential closedness of
\(\E\cap\Cth \). The result is the abstract Mosco/Gamma convergence theorem of
Section~\ref{xuwzhr} applied with \(\E\) replaced by \(\E\cap\Cth \).
\end{proof}

\begin{remark}[Why the compensated-compactness argument is essential]
The indicator \(\I_{\Cth}\) is not weakly lower semicontinuous on \(\Z\), because
\(\Cth\) is not weakly closed. Thus, the constrained functional
\(\I_\E+d^2(\cdot,\D)+\I_{\Cth}\) is not covered by the usual direct method by
looking at \(\I_{\Cth}\) alone. The correct weak closedness statement is the
closedness of \(\E\cap\Cth\), not of \(\Cth\). The differential constraints in
\(\E\) remove the oscillatory counterexample by forcing convergence of the power
densities \(e_j\cdot s_j\) in distributions.
\end{remark}

\section{Enforcing thermodynamic consistency by Lagrange multipliers}

We continue to consider the natural phase space
\[
        \Z:=L^2(\Omega;\R^N)\times L^2(\Omega;\R^N),
        \quad z=(e,s) ,
\]
and define
\[
        g(z):=e\cdot s\in L^1(\Omega).
\]
Then the thermodynamic constraint set is
\[
        \Cth
        =
        g^{-1}\bigl(L^1_-(\Omega)\bigr),
        \quad
        L^1_-(\Omega):=
        \{w\in L^1(\Omega):w\le0\text{ a.e.}\}.
\]
The positive multiplier cone is
\[
        L^\infty_+(\Omega):=
        \{\lambda\in L^\infty(\Omega):\lambda\ge0\text{ a.e.}\}.
\]
Since \(L^1_-(\Omega)\) is a closed convex cone,
\begin{equation}
        \I_{\Cth}(z)
        =
        \I_{L^1_-(\Omega)}(g(z))
        =
        \sup_{\lambda\in L^\infty_+(\Omega)}
        \int_\Omega \lambda\,e\cdot s\,\dd x,
        \label{eq:indicator-support}
\end{equation}
where the supremum is \(+\infty\) if \(e\cdot s\not\le0\) a.e. 

\subsection{Formulation as saddle problem}

Define
\begin{equation} \label{eq:lagrangian}
        \Lag(z,\lambda)
        :=
        \I_\E(z)+d^2(z,\D)
        +
        \int_\Omega \lambda\,e\cdot s\,\dd x,
        \quad
        \lambda\in L^\infty_+(\Omega).
\end{equation}

The first question is what a saddle point of the multiplier Lagrangian actually implies for the original constrained Data-Driven problem. The next result records the expected consequence: saddle points enforce the sign constraint and yield the usual complementarity relation.

\begin{theorem}[Saddle point formulation] \label{thm:saddle-complementarity}
Let
\[
        z_\ast=(e_\ast,s_\ast)\in\E,
        \quad
        \lambda_\ast\in L^\infty_+(\Omega).
\]
If \((z_\ast,\lambda_\ast)\) is a saddle point of \(\Lag\), i.e.,
\begin{equation} \label{eq:saddle}
        \Lag(z_\ast,\lambda)
        \le
        \Lag(z_\ast,\lambda_\ast)
        \le
        \Lag(z,\lambda_\ast)
        \quad
        \forall z\in\Z,\quad
        \forall \lambda\in L^\infty_+(\Omega),
\end{equation}
then \(z_\ast\) solves the constrained Data-Driven problem on \(\E\cap\Cth\)
and
\begin{equation}
        \lambda_\ast\ge0,
        \quad
        e_\ast\cdot s_\ast\le0,
        \quad
        \lambda_\ast\,e_\ast\cdot s_\ast=0
        \quad\text{a.e. in }\Omega.
        \label{eq:complementarity}
\end{equation}
\end{theorem}

\begin{proof}
Taking \(\lambda=t\mu\), with \(t>0\) and arbitrary
\(\mu\in L^\infty_+(\Omega)\), in the left inequality of \eqref{eq:saddle}
gives
\[
        t\int_\Omega \mu\,e_\ast\cdot s_\ast\,\dd x
        \le
        \int_\Omega \lambda_\ast\,e_\ast\cdot s_\ast\,\dd x
        \quad\text{for all }t>0.
\]
Letting \(t\to\infty\) yields
\[
        \int_\Omega \mu\,e_\ast\cdot s_\ast\,\dd x\le0
        \quad
        \forall \mu\in L^\infty_+(\Omega),
\]
and hence \(e_\ast\cdot s_\ast\le0\) a.e. Taking \(\lambda=0\) and then
\(\lambda=2\lambda_\ast\) in the same inequality gives
\[
        \int_\Omega \lambda_\ast\,e_\ast\cdot s_\ast\,\dd x=0.
\]
Together with \(\lambda_\ast\ge0\) and \(e_\ast\cdot s_\ast\le0\), this implies
\eqref{eq:complementarity}. The right inequality in \eqref{eq:saddle},
restricted to \(z\in\E\cap\Cth\), gives
\[
        d^2(z_\ast,\D)
        =
        \Lag(z_\ast,\lambda_\ast)
        \le
        \Lag(z,\lambda_\ast)
        \le
        d^2(z,\D),
\]
because \(\int_\Omega \lambda_\ast e\cdot s\,\dd x\le0\) for feasible \(z\),
and the value at \(z_\ast\) has zero multiplier contribution by
\eqref{eq:complementarity}. Hence \(z_\ast\) is a constrained minimizer.
\end{proof}

Equivalently, a positive multiplier can only be supported where the constraint is active. On the inactive set \(e_\ast\cdot s_\ast<0\), maximizing over nonnegative multipliers gives \(\lambda_\ast=0\); if \(e_\ast\cdot s_\ast>0\) on a set of positive measure, the multiplier term is unbounded above and no saddle point can exist.

Thus, the Lagrangian formulation selects those primal minimizers that admit a supporting multiplier.  In convex programming, suitable constraint qualifications often imply that every primal minimizer has such a multiplier; see Ekeland and T{\'e}mam~\cite[Ch.~VI]{EkelandTemam1999}. The complementarity system~\eqref{eq:complementarity} is, in essence, a mathematical program with complementarity constraints (MPCC); an analogous MPCC structure arises, in a contact-mechanics rather than thermodynamic setting, in the hybrid data-driven formulation of~\cite{GebhardtLangeSteinbach2024}.

Here, the existence of Lagrange multipliers is not automatic and is in fact highly problematic. First, the pointwise thermodynamic constraint is nonlinear and nonconvex in the phase variables and the product \(e\cdot s\) is not weakly continuous in the ambient space \( \Z = L^2(\Omega;\R^N)\times L^2(\Omega;\R^N)\).  The compensated-compactness argument establishes the weak closedness needed for the primal problem on \(\E\cap\Cth\), but it does not, by itself, supply a supporting multiplier for every primal minimizer. Second, a basic difficulty with the Lagrangian formulation in Theorem~\ref{thm:saddle-complementarity} is that the multiplier variable ranges over the cone \(L^\infty_+(\Omega)\), which is neither bounded nor compact in the strong topology.  Thus minimizing sequences in the primal variable may be compact, while the corresponding multiplier sequences may escape to infinity, reflecting concentration of the enforcement of the constraint \(e\cdot s\le0\).  Even after resorting to weak-\(*\) boundedness, compactness is available only on bounded subsets of \(L^\infty(\Omega)\), and the bilinear coupling in \eqref{eq:lagrangian} is not jointly continuous under the natural weak topology in \(\Z\) and the weak-\(*\) topology in \(L^\infty(\Omega)\).

Bounding the multipliers restores weak-\(*\) compactness of the truncated multiplier set \(\Lambda_R\) by Banach--Alaoglu \cite{Brezis2011}, but it does not repair this continuity defect.

\begin{proposition}[Failure of joint continuity]
The trilinear form
\[
        B(e,s,\lambda):=\int_\Omega \lambda\,e\cdot s\,\dd x
\]
is not jointly sequentially continuous for
\[
        L^2_{\rm weak}(\Omega;\R^N)\times
        L^2_{\rm weak}(\Omega;\R^N)\times
        L^\infty_{\rm weak-*}(\Omega).
\]
\end{proposition}

\begin{proof}
On \( \Omega = (0,2\pi)\), take \(e_n(x)=s_n(x)=\sin(nx)\) and \(\lambda_n\equiv1\). Then \(e_n\rightharpoonup0\) and \(s_n\rightharpoonup0\) in \(L^2(\Omega)\), and \(\lambda_n\overset{*}{\rightharpoonup}1\) in \(L^\infty(\Omega)\), but
\[
        B(e_n,s_n,\lambda_n)=\int_0^{2\pi}\sin^2(nx)\,\dd x=\pi
        \not\to 0=B(0,0,1).
\]
\end{proof}

\subsection{Game-theoretical formulation}

The ill-posedness of the constrained saddle problem (\ref{eq:saddle}), specifically the compactness and continuity defects just noted, call for suitable relaxations of the constraint and compactifications of the Lagrange multipliers. Useful intuition and suggestions of appropriate choices may be obtained by reformulating the saddle problem (\ref{eq:saddle}) as a zero-sum game. To this end, let
\begin{equation}
    \M(z,\lambda) := -\Lag(z,\lambda). 
\end{equation}
Then, the game of finding $(z_\ast,\lambda_\ast) \in \Z \times L^\infty_+(\Omega)$ such that
\begin{subequations} \label{eq:game}
\begin{align}
    &
    \Lag(z_\ast,\lambda_\ast)
    \le
    \Lag(z,\lambda_\ast) ,
    \quad \forall z\in\Z,
    \\ &
    \M(z_\ast,\lambda_\ast)
    \le
    \M(z_\ast,\lambda) ,
    \quad \forall \lambda\in L^\infty_+(\Omega),
\end{align}
\end{subequations}
is equivalent to the saddle problem (\ref{eq:saddle}). In this setting, $(z_\ast,\lambda_\ast)$ is referred to as a Nash equilibrium. 

The game reformulation does not change the mathematical content of the saddle problem: a pure Nash equilibrium of the zero-sum game is precisely a saddle point of the Lagrangian.  Its value is instead organizational and selective.  It separates the two optimality requirements into two best-response conditions and, therefore, makes explicit that a Lagrangian solution is a primal minimizer whose optimality is supported by a multiplier.  Thus, the game formulation identifies the multiplier-supported subset of the primal solution set, rather than necessarily all primal minimizers.

This viewpoint is useful because it indicates what must be proved in order to obtain saddle solutions by fixed-point methods: compact strategy sets, nonempty best responses, convexity of the values, and closedness of the best-response graph; see, e.g., \cite{Glicksberg1952}.  It also suggests natural compactified and penalized games that overcome the compactness and continuity defects of the saddle problem (\ref{eq:saddle}). In this way the game formulation provides a systematic approximation and selection mechanism.  

We formalize these heuristics in the following theorem. 

\begin{theorem}[Bounded multipliers and exact penalization] \label{thm:exact-p}
Let \(R>0\) and define
\begin{equation} \label{Rqq7mF}
        \Lambda_R
        :=
        \{\lambda\in L^\infty(\Omega):0\le \lambda\le R \text{ a.e. in }\Omega\}.
\end{equation}
There holds:
\begin{equation}
        \J_R(z)
        :=
        \sup_{\lambda\in\Lambda_R}\Lag(z,\lambda) 
        =
        \I_\E(z)+d^2(z,\D)
        +
        R\int_\Omega (e\cdot s)_+\,\dd x,
        \label{eq:bounded-multiplier-penalty}
\end{equation}
where \(a_+:=\max\{a,0\}\). 
\end{theorem}

\begin{proof}
Fix \(z=(e,s)\in\Z\). If \(z\notin\E\), then \(\I_\E(z)=+\infty\), and both
sides of \eqref{eq:bounded-multiplier-penalty} are \(+\infty\). Assume,
therefore, that \(z\in\E\). Since \(e,s\in L^2(\Omega;\mathbb R^N)\), we have
\(
        g:=e\cdot s\in L^1(\Omega).
\)
Thus, for every \(\lambda\in\Lambda_R\),
\[
        \int_\Omega \lambda g\,\dd x
        \le
        \int_\Omega R g_+\,\dd x.
\]
This gives
\[
        \sup_{\lambda\in\Lambda_R}
        \int_\Omega \lambda g\,\dd x
        \le
        R\int_\Omega g_+\,\dd x.
\]
Conversely, the choice
\[
        \lambda_R(x)
        :=
        R\, \chi_{\{ x ~ s.t. ~ g(x)>0\}}
        \label{SbpYqt}
\]
belongs to \(\Lambda_R\) and gives
\[
        \int_\Omega \lambda_R g\,\dd x
        =
        R\int_\Omega g_+\,\dd x.
\]
Therefore
\[
        \sup_{\lambda\in\Lambda_R}
        \int_\Omega \lambda e\cdot s\,\dd x
        =
        R\int_\Omega (e\cdot s)_+\,\dd x.
\]
Adding the terms \(\I_\E(z)+d^2(z,\D)\), which are independent of \(\lambda\), proves \eqref{eq:bounded-multiplier-penalty}. 
\end{proof}

The maximizing bounded multiplier is therefore bang-bang: it marks the violation set \(\{e\cdot s>0\}\) rather than acting as an unconstrained reaction enforcing the hard constraint.

A consequence of the compactification (\ref{Rqq7mF}) is that 
the upper value of the truncated multiplier game is the modified primal problem
\begin{equation}
        \inf_{z\in\Z}\J_R(z)
        =
        \inf_{z\in\E}
        \left\{
        d^2(z,\D)
        +
        R\int_\Omega (e\cdot s)_+\,\dd x
        \right\} ,
        \label{eq:truncated-game-penalized-primal}
\end{equation}
which is a \emph{penalty reformulation} of the constrained primal minimum problem. Thus, the net outcome of the Lagrange multiplier excursion is to identify an appropriate penalty reformulation of the original primal minimum problem. The solution of problem \eqref{SbpYqt}, if it exists, satisfies $\lambda = 0$ when $e\cdot s < 0$, $\lambda = \lambda(x)$ (variable, in general) when $e\cdot s = 0$ and the set of points with $e\cdot s > 0$ has zero measure. Should we solve the problem
\begin{align*}
(z_R , \lambda_R)  = \arg \inf_{z\in \Z} \sup_{\lambda\in \Lambda_R} {\Lag}(z,\lambda)
\end{align*}
we would obtain $\lambda_R = 0$ when $e_R\cdot s_R < 0$, $\lambda_R = \lambda_R(x)$ when $e_R\cdot s_R = 0$ and $\lambda_R = R$  when $e_R\cdot s_R > 0$. According to Theorem~\ref{thm:exact-p}, we do not need to solve for $\lambda_R$, as the solution $z_R$ (and, consequently, $u_R$) is the same as that of the penalty formulation \eqref{eq:truncated-game-penalized-primal}.

\section{Penalty enforcement of the thermodynamic constraint}

The bounded-multiplier formula \eqref{eq:bounded-multiplier-penalty} suggests replacing the hard thermodynamic indicator by the finite penalty
\begin{equation}
        \mathcal P(z)
        :=
        \int_\Omega (e\cdot s)_+\,\dd x,
        \quad z=(e,s)\in\Z.
        \label{eq:positive-part-penalty}
\end{equation}
Then, for \(R>0\), the penalized Data-Driven problem is (\ref{eq:truncated-game-penalized-primal}), i.e., minimization of the functional obtained from the truncated multiplier set \(\Lambda_R\) in the bounded-multiplier formula \eqref{eq:bounded-multiplier-penalty}.

The penalized formulation replaces the nonconvex indicator \(\I_{\Cth}\) by the
weakly lower-semicontinuous positive-part penalty \(R\mathcal P\) on the
material-independent set \(\E\).  For finite \(R\), violations of the second-law
inequality are allowed but charged in proportion to their positive part.  In the
limit \(R\to+\infty\), finite energy forces
\[
        \int_\Omega (e\cdot s)_+\,\dd x=0,
\]
and hence recovers the constrained feasible set \(\E\cap\Cth\).

For every \(R>0\), the minimum value
\[
        m_R:=\inf_{z\in\Z}\J_R(z)
\]
satisfies
\[
        m_R\le m_{\Cth}:=
        \inf_{z\in\Z}\J_\Cth
        =
        \inf_{z\in\E\cap\Cth}d^2(z,\D),
\]
because \(\J_R=\J_\Cth\) on \(\E\cap\Cth\).  Also, \(R\mapsto \J_R(z)\) is
nondecreasing for each fixed \(z\), since \(\mathcal P(z)\ge0\).  The following
results make the limiting statement precise.

We proceed to show that the primal constrained functional $\J_\Cth$, \eqref{eq:J-C}, is indeed recovered from the penalized functionals $\J_R$ as $R\to+\infty$ in the sense of $\Gamma$-convergence. Heuristically, this limit is suggested by the fact that \(\mathcal P(z)=0\) if and only if \((e\cdot s)_+=0\) a.e., or, equivalently, \(e\cdot s\le0\) a.e. In addition, we show that the sequence $(\J_R)$ is equi-coercive, which, together with $\Gamma$-convergence ensures convergence of minimizers. 

\subsection{Properties of the penalized functional}

The following observation is the penalty analogue of the compensated-compactness closedness result used for the indicator formulation.  It uses the convexity of the scalar map \(a\mapsto a_+\).

\begin{theorem}[Weak lower semicontinuity of the positive-part penalty]
\label{thm:penalty-wlsc}
The functional \(\mathcal P\) defined by \eqref{eq:positive-part-penalty} is weakly sequentially lower semicontinuous on \(\E\).  Consequently, if \(\zeta\in L^\infty(\Omega)\)  and \(z\mapsto d^2(z,\D)\) is weakly sequentially lower semicontinuous on \(\Z\), then \(\J_R\) is weakly sequentially lower semicontinuous on \(\Z\) for every \(R>0\).
\end{theorem}

\begin{proof}
Let \(z_j=(e_j,s_j)\in\E\) and suppose that \(z_j\rightharpoonup z=(e,s)\) weakly in \(\Z\).  By weak closedness of \(\E\), we have \(z\in\E\).  Choose associated potentials \(u_j\in H^1_g(\Omega)\) with \(e_j=\grad u_j\) and \(\diver s_j+\zeta u_j-q=0\).  As before, \((u_j)\) is bounded in \(H^1(\Omega)\), and, up to subsequences,
\[
        u_j\rightharpoonup u\text{ in }H^1(\Omega),
        \quad
        u_j\to u\text{ in }L^2(\Omega),
\]
where \(e=\grad u\) and \(\diver s+\zeta u-q=0\).  Hence
\[
        \diver s_j=q-\zeta u_j\to q-\zeta u=\diver s
        \quad\text{strongly in }H^{-1}(\Omega),
\]
and \(\operatorname{curl}e_j=0\).  The div-curl lemma \cite{Murat1978,Tartar1979} gives
\begin{equation}
        e_j\cdot s_j\to e\cdot s
        \quad\text{in }\mathcal D'(\Omega).
        \label{eq:product-distribution-penalty}
\end{equation}
For every \(\varphi\in C_c^\infty(\Omega)\) with \(0\le\varphi\le1\),
\[
        \varphi\,e_j\cdot s_j\le (e_j\cdot s_j)_+
        \quad\text{a.e. in }\Omega.
\]
Indeed, this is immediate where \(e_j\cdot s_j\ge0\), and trivial where \(e_j\cdot s_j<0\). Integrating, passing to the distributional limit on the left using \eqref{eq:product-distribution-penalty}, and taking the liminf on the right gives
\[
        \int_\Omega \varphi\,e\cdot s\,\dd x
        \le
        \liminf_{j\to\infty}
        \int_\Omega (e_j\cdot s_j)_+\,\dd x.
\]
Taking the supremum over \(0\le\varphi\le1\) yields
\[
        \mathcal P(z)=\int_\Omega (e\cdot s)_+\,\dd x
        \le
        \liminf_{j\to\infty}\mathcal P(z_j).
\]
This proves weak sequential lower semicontinuity of \(\mathcal P\) on \(\E\).
The asserted lower semicontinuity of \(\J_R\) follows from the weak closedness of
\(\E\), the lower semicontinuity of \(d^2(\cdot,\D)\), and the nonnegative
coefficient \(R\).
\end{proof}

The $\Gamma$-convergence of the sequence $\J_R$ follows immediately from the monotonicity of the sequence. 

\begin{corollary}[Immediate \(\Gamma\)-convergence consequence] \label{TGPxN5}
The sequence $(\J_R)$ 
is monotone increasing in \(R\), and its pointwise supremum is
\[
        \sup_{R>0}\mathcal J_R(z)
        =
        \mathcal I_{\mathcal E\cap\mathcal C}(z)
        +d^2(z,\mathcal D)
        =:\mathcal J_\Cth(z).
\]
Hence,
\[
        \mathcal J_\Cth
        =
        \Gamma \hbox{-}\lim_{R\to+\infty} \J_R
\]
in the weak topology of \(\Z\). 
\end{corollary}

\begin{proof}
Since Theorem~\ref{thm:penalty-wlsc} gives the weak lower semicontinuity of the positive-part penalty on \(\mathcal E\), the limiting functional \(\mathcal J_\Cth\) is weakly lower semicontinuous.  Therefore, the standard monotone-convergence property of \(\Gamma\)-convergence \cite{DalMaso1993, Braides2006} gives
\[
        \mathcal J_\Cth
        =
        \Gamma\text{-}\lim_{R\to+\infty}\mathcal J_R
\]
with respect to the weak topology of \(\Z\).
\end{proof}

In order to ascertain the convergence of the penalized minimizers, we investigate the equicoercivity of the sequence $(\J_R)$.

\begin{theorem}[Equicoercivity of the penalty family]
\label{thm:JR-equicoercive}
Assume that there exist constants \(c>0\) and \(b\ge0\) such that
\begin{equation}
        \|y-z\|_\Z
        \ge
        c\left(\|y\|_\Z+\|z\|_\Z\right)-b
        \quad
        \forall y\in\D,
        \quad
        \forall z\in\E .
        \label{eq:fixed-D-transversality}
\end{equation}
Then, the family \((\J_R)_{R>0}\) is equicoercive with respect to the weak
topology of \(\Z\).  More precisely, for every \(C<+\infty\),
\begin{equation}
        \bigcup_{R>0}
        \{z\in\Z:\J_R(z)\le C\}
        \label{eq:JR-uniform-sublevel}
\end{equation}
is bounded in \(\Z\), and therefore relatively weakly sequentially compact.
\end{theorem}

\begin{proof}
Let \(z\in\Z\) and \(R>0\) be such that \(\J_R(z)\le C\).  Then, \(z\in\E\) and
\[
        d^2(z,\D)\le \J_R(z)\le C,
\]
because the penalty term is nonnegative.  Taking the infimum over
\(y\in\D\) in \eqref{eq:fixed-D-transversality} gives
\[
        d(z,\D)
        \ge
        c\|z\|_\Z-b.
\]
Consequently,
\[
        c\|z\|_\Z-b
        \le
        d(z,\D)
        \le
        C^{1/2},
\]
and hence
\[
        \|z\|_\Z
        \le
        \frac{C^{1/2}+b}{c}.
\]
This bound is independent of \(R\).  Since \(\Z\) is reflexive, bounded subsets
of \(\Z\) are relatively weakly sequentially compact.
\end{proof}

\begin{corollary}[Convergence of almost minimizers]
\label{cor:minimizers-penalty}
Assume the hypotheses of Theorems~\ref{thm:penalty-wlsc} and \ref{thm:JR-equicoercive}.  Let \(R_j\to+\infty\), and let \((z_j)\subset\Z\)
be an almost-minimizing sequence in the sense that
\[
        \J_{R_j}(z_j)
        \le
        \inf_{z\in\Z}\J_{R_j}(z)+o(1).
\]
Then, \((z_j)\) is relatively weakly sequentially compact in \(\Z\).  Every weak
cluster point belongs to \(\E\cap\Cth \) and minimizes the constrained functional
\(\J_\Cth\).  Moreover,
\[
        \lim_{j\to\infty}\inf_{z\in\Z}\J_{R_j}(z)
        =
        \min_{z\in\Z}\J_\Cth,
\]
provided the right-hand minimum is finite.
\end{corollary}

\begin{proof}
The equicoercivity theorem gives weak relative compactness of almost minimizers on bounded energy levels.  The identification of all cluster points as minimizers of \(\J_\Cth\), together with convergence of minimum values, is the fundamental theorem of \(\Gamma\)-convergence, see Appendix~\ref{fCnmub}, Theorem~\ref{K87zb7}.
\end{proof}

\section{All together now: Data, discretization and penalization}

Finally, we combine the three approximation mechanisms considered separately above: approximation of the material data set, spatial discretization of the field equations, and penalty enforcement of the thermodynamic constraint. For \(h,k\in\mathbb N\) and \(R>0\), define
\begin{equation}
        \J_{h,k,R}(z)
        :=
        \I_{\E_k}(z)
        +
        d^2(z,\D_h)
        +
        R\,\mathcal P(z),
        \quad z\in\Z,
        \label{eq:J-h-k-R}
\end{equation}
where
\[
        \mathcal P(z)
        =
        \int_\Omega (e\cdot s)_+\,\dd x,
        \quad z=(e,s).
\]
The intended limiting functional is the constrained Data-Driven functional $\J_{\Cth}$. Thus, the data sets \(\D_h\) approximate \(\D\), the discrete feasible sets \(\E_k\) approximate \(\E\), and the penalty parameter \(R\) enforces \(\mathcal P(z)=0\), equivalently \(z\in\Cth\), in the limit.

\subsection{Properties of the approximation scheme}

The following discrete compactness assumption is the natural analogue, for the discrete feasible sets, of the compensated-compactness property used in Theorem~\ref{thm:penalty-wlsc}:
\begin{equation}
\begin{gathered}
        k_j\to\infty,\quad z_j\in\E_{k_j},\quad
        z_j\rightharpoonup z\text{ weakly in }\Z
        \\[0.3em]
        \Longrightarrow
        \quad
        z\in\E,
        \quad
        \mathcal P(z)
        \le
        \liminf_{j\to\infty}\mathcal P(z_j).
\end{gathered}
        \label{eq:discrete-penalty-liminf}
\end{equation}
It is automatic for conforming discretizations for which the discrete fields satisfy the continuous differential constraints exactly, and it is the consistency condition that must be verified for nonconforming or weakly enforced schemes. In addition, the diagonal sequence must not drive \(R\) to infinity faster than the discretization can recover thermodynamically admissible states. We therefore call a diagonal scheme
\[
        j\mapsto (h_j,k_j,R_j),
        \quad
        h_j\to\infty,\quad k_j\to\infty,\quad R_j\to+\infty,
\]
\emph{admissible} if, for every \(z\in\E\cap\Cth\), there exists
\(z_j\in\E_{k_j}\) such that
\begin{equation}
        z_j\to z\;\, \text{(strongly)},
        \quad
        \limsup_{j\to\infty}d^2(z_j,\D_{h_j})
        \le
        d^2(z,\D),
        \quad
        R_j\mathcal P(z_j)\to0.
        \label{eq:admissible-diagonal-recovery}
\end{equation}
The last condition is the only new restriction introduced by the penalty:
strong recovery alone gives \(\mathcal P(z_j)\to0\), but the product
\(R_j\mathcal P(z_j)\) must also vanish. This requirement is better understood in the proof of the following Theorem.

\begin{theorem}[Diagonal \(\Gamma\)-convergence of the fully approximate problems]
\label{thm:all-together-gamma}
Assume that
\[
        \D=M\hbox{-}\lim_{h\to\infty}\D_h,
        \quad
        \E=M\hbox{-}\lim_{k\to\infty}\E_k
        \quad\text{in }\Z.
\]
Assume also the uniform transversality estimate
\begin{equation}
        \|y-z\|_\Z
        \ge
        c\bigl(\|y\|_\Z+\|z\|_\Z\bigr)-b
        \quad
        \forall y\in\D_h,\quad
        \forall z\in\E_k,
        \label{eq:h-k-R-transversality}
\end{equation}
with constants \(c>0\) and \(b\ge0\) independent of \(h\) and \(k\).
Finally, assume the discrete penalty lower-semicontinuity property
\eqref{eq:discrete-penalty-liminf}, and let
\((h_j,k_j,R_j)\) be an admissible diagonal in the sense of
\eqref{eq:admissible-diagonal-recovery}. Define
\[
        \widehat{\J}_j
        :=
        \J_{h_j,k_j,R_j}.
\]
Then
\begin{equation}
        \J_{\Cth}
        =
        \Gamma\hbox{-}\lim_{j\to\infty}\widehat{\J}_j
        \label{eq:all-together-gamma-limit}
\end{equation}
with respect to the weak topology of \(\Z\).
\end{theorem}

\begin{proof}
We prove the two sequential \(\Gamma\)-inequalities on bounded sublevel sets,
where the weak topology of \(\Z\) is metrizable.

First, let \(z_j\rightharpoonup z\) weakly in \(\Z\). If
\[
        \liminf_{j\to\infty}\widehat{\J}_j(z_j)=+\infty,
\]
there is nothing to prove. Passing to a subsequence, not relabeled, we may
therefore assume that
\[
        \sup_j \widehat{\J}_j(z_j)<+\infty.
\]
Then \(z_j\in\E_{k_j}\), and
\[
        d^2(z_j,\D_{h_j})\le \widehat{\J}_j(z_j).
\]
Choose \(y_j\in\D_{h_j}\) such that
\[
        \|z_j-y_j\|_\Z
        \le
        d(z_j,\D_{h_j})+\frac1j.
\]
The transversality estimate \eqref{eq:h-k-R-transversality} gives
boundedness of both \((z_j)\) and \((y_j)\) in \(\Z\). Up to a further
subsequence,
\[
        y_j\rightharpoonup y
        \quad\text{weakly in }\Z.
\]
Since \(\D=M\hbox{-}\lim\D_h\), the Mosco-liminf condition for sets gives
\(y\in\D\). Likewise, by the discrete compactness part of
\eqref{eq:discrete-penalty-liminf}, \(z\in\E\). Hence, by weak lower
semicontinuity of the norm,
\[
        d(z,\D)
        \le
        \|z-y\|_\Z
        \le
        \liminf_{j\to\infty}\|z_j-y_j\|_\Z
        \le
        \liminf_{j\to\infty}\left(d(z_j,\D_{h_j})+\frac1j\right).
\]
Moreover, since \(R_j\to+\infty\) and
\[
        R_j\mathcal P(z_j)
        \le
        \widehat{\J}_j(z_j),
\]
we have
\[
        \mathcal P(z_j)\to0.
\]
Using \eqref{eq:discrete-penalty-liminf},
\[
        0\le \mathcal P(z)
        \le
        \liminf_{j\to\infty}\mathcal P(z_j)
        =
        0.
\]
Thus \(\mathcal P(z)=0\), i.e. \(z\in\Cth\). Therefore
\(z\in\E\cap\Cth\), and
\[
        \J_{\Cth}(z)
        =
        d^2(z,\D)
        \le
        \liminf_{j\to\infty}d^2(z_j,\D_{h_j})
        \le
        \liminf_{j\to\infty}\widehat{\J}_j(z_j).
\]
This proves the \(\Gamma\)-liminf inequality.

For the \(\Gamma\)-limsup inequality, let \(z\in\Z\). If
\(\J_{\Cth}(z)=+\infty\), there is nothing to prove. Otherwise
\(z\in\E\cap\Cth\). By admissibility of the diagonal, there exists
\(z_j\in\E_{k_j}\) such that
\[
        z_j\to z\quad\text{strongly in }\Z,
        \quad
        \limsup_{j\to\infty}d^2(z_j,\D_{h_j})
        \le
        d^2(z,\D),
        \quad
        R_j\mathcal P(z_j)\to0.
\]
Hence
\[
        \limsup_{j\to\infty}\widehat{\J}_j(z_j)
        =
        \limsup_{j\to\infty}
        \left(
        d^2(z_j,\D_{h_j})
        +
        R_j\mathcal P(z_j)
        \right)
        \le
        d^2(z,\D)
        =
        \J_{\Cth}(z).
\]
This proves the recovery inequality and completes the proof.
\end{proof}

\begin{theorem}[Equicoercivity of the fully approximate family]
\label{thm:h-k-R-equicoercive}
Assume the uniform transversality estimate
\eqref{eq:h-k-R-transversality}. Then the family
\[
        \{\J_{h,k,R}: h,k\in\mathbb N,\ R>0\}
\]
is equicoercive with respect to the weak topology of \(\Z\). More precisely,
for every \(C<+\infty\),
\begin{equation}
        \bigcup_{h,k\in\mathbb N,\ R>0}
        \{z\in\Z:\J_{h,k,R}(z)\le C\}
        \label{eq:h-k-R-sublevels}
\end{equation}
is bounded in \(\Z\), and therefore relatively weakly sequentially compact.
\end{theorem}

\begin{proof}
Let \(z\in\Z\) satisfy \(\J_{h,k,R}(z)\le C\). Then \(z\in\E_k\), and, since
the penalty term is nonnegative,
\[
        d^2(z,\D_h)\le C.
\]
Taking the infimum over \(y\in\D_h\) in
\eqref{eq:h-k-R-transversality} gives
\[
        d(z,\D_h)
        \ge
        c\|z\|_\Z-b.
\]
Therefore
\[
        c\|z\|_\Z-b
        \le
        d(z,\D_h)
        \le
        C^{1/2},
\]
and hence
\[
        \|z\|_\Z
        \le
        \frac{C^{1/2}+b}{c}.
\]
This bound is independent of \(h\), \(k\), and \(R\). Since \(\Z\) is
reflexive, bounded subsets of \(\Z\) are relatively weakly sequentially
compact.
\end{proof}

\begin{corollary}[Convergence of fully approximate minimizers]
\label{cor:h-k-R-minimizers}
Assume the hypotheses of Theorem~\ref{thm:all-together-gamma} and
Theorem~\ref{thm:h-k-R-equicoercive}. Let
\[
        \widehat{\J}_j:=\J_{h_j,k_j,R_j}
\]
for an admissible diagonal \((h_j,k_j,R_j)\), and let
\((z_j)\subset\Z\) be approximate minimizers,
\begin{equation}
        \widehat{\J}_j(z_j)
        \le
        \inf_{z\in\Z}\widehat{\J}_j(z)+\varepsilon_j,
        \quad
        \varepsilon_j\downarrow0.
        \label{eq:h-k-R-approx-min}
\end{equation}
Assume that the constrained limiting problem has finite value,
\[
        \min_{z\in\Z}\J_{\Cth}(z)<+\infty.
\]
Then \((z_j)\) is relatively weakly sequentially compact in \(\Z\). Every weak
cluster point \(z_\ast\) belongs to \(\E\cap\Cth\) and minimizes
\(\J_{\Cth}\). Moreover,
\[
        \lim_{j\to\infty}
        \inf_{z\in\Z}\widehat{\J}_j(z)
        =
        \min_{z\in\Z}\J_{\Cth}(z).
\]
In particular, if the constrained problem has a unique minimizer \(z_\ast\),
then the full sequence converges weakly,
\[
        z_j\rightharpoonup z_\ast
        \quad\text{in }\Z.
\]
\end{corollary}

\begin{proof}
By Theorem~\ref{thm:all-together-gamma},
\[
        \widehat{\J}_j
        \quad
        \Gamma\hbox{-converges to}
        \quad
        \J_{\Cth}
\]
with respect to the weak topology of \(\Z\). By Theorem~\ref{thm:h-k-R-equicoercive}, the family \((\widehat{\J}_j)\) is equicoercive. The conclusion is therefore the standard compactness and convergence theorem for quasi-minimizers under \(\Gamma\)-convergence, applied in the weak topology of \(\Z\); see Theorem~\ref{K87zb7}. The inclusion \(z_\ast\in\E\cap\Cth\) also follows directly from the liminf argument in Theorem~\ref{thm:all-together-gamma}.
\end{proof}

\subsection{Optimal diagonal strategies}

The qualitative theorem above only requires
\[
        h_j\to\infty,
        \quad
        k_j\to\infty,
        \quad
        R_j\to+\infty,
\]
together with the admissibility condition \eqref{eq:admissible-diagonal-recovery}. In computation, however, these three limits need to be chosen in a balanced way.

Let \(a_h\downarrow0\) denote the material-data error scale. In the setting of Lemma~\ref{7YAxr2}, in which the exact data set corresponds to Darcy's law, we may take
\[
        a_h:=|\Omega|^{1/2}(\rho_h+t_h).
\]
Let \(b_k\downarrow0\) denote the finite-element consistency scale for \(\E_k\to\E\). Finally, let \(p_k\downarrow0\) denote the thermodynamic recovery leakage, meaning that, on bounded subsets of \(\E\cap\Cth\), one can choose recovery states \(z_k\in\E_k\) with
\[
        z_k\to z\quad\text{strongly in }\Z,
        \quad
        \mathcal P(z_k)\lesssim p_k.
\]
Then the recovery error for the fully approximate functional has the form
\[
        \J_{h,k,R}(z_k)-\J_{\Cth}(z)
        \lesssim
        a_h+b_k+R\,p_k.
\]
At the same time, the penalty parameter controls the residual violation of the thermodynamic inequality on the scale \(R^{-1}\): bounded energy implies
\[
        \mathcal P(z_{h,k,R})\lesssim R^{-1}
\]
for near minimizers, up to the value of a constrained competitor.

Thus a balanced diagonal should satisfy
\[
        a_h\to0,
        \quad
        b_k\to0,
        \quad
        R\to+\infty,
        \quad
        R\,p_k\to0.
\]
A practical rule is to choose \(R\) so that the penalty tolerance is no
larger than the material-data error,
\[
        R(h)\simeq a_h^{-1},
\]
and then choose the coarsest spatial discretization satisfying
\[
        b_{k(h)}\lesssim a_h,
        \quad
        p_{k(h)}\lesssim a_h^2.
\]
The first condition prevents the finite-element error from dominating the
data error. The second condition prevents the amplified penalty leakage
\(R(h)p_{k(h)}\) from dominating the total error.

If the discretization is thermodynamically structure-preserving, so that
one can recover constrained states with \(p_k=0\), the additional penalty
restriction disappears and the optimal rule reduces to the same balance as
in the unconstrained case:
\[
        b_{k(h)}\lesssim a_h,
        \quad
        R(h)\to+\infty,
        \quad
        R(h)^{-1}\lesssim a_h.
\]
In that case, \(R(h)\simeq a_h^{-1}\) and \(b_{k(h)}\simeq a_h\) are
natural choices.

\begin{example}[Algebraic spatial rates with penalty leakage] {\rm
Assume that
\[
        b_k\simeq \Delta_k^p,
        \quad
        p_k\simeq \Delta_k^r,
\]
where \(\Delta_k\) is the mesh size, \(p>0\) is the finite-element
consistency order, and \(r>0\) is the thermodynamic recovery order. With
\(a_h:=|\Omega|^{1/2}(\rho_h+t_h)\), the balancing choice
\[
        R(h)\simeq a_h^{-1}
\]
leads to
\[
        \Delta_{k(h)}^p\lesssim a_h,
        \quad
        \Delta_{k(h)}^r\lesssim a_h^2.
\]
Equivalently,
\[
        \Delta_{k(h)}
        \lesssim
        \min\left\{
        a_h^{1/p},
        a_h^{2/r}
        \right\}.
\]
If the number of degrees of freedom satisfies \(N_k\simeq \Delta_k^{-N}\), then
\[
        N_{k(h)}
        \gtrsim
        \max\left\{
        a_h^{-N/p},
        a_h^{-2N/r}
        \right\}.
\]
Thus the optimal computable diagonal is the coarsest mesh satisfying both the usual finite-element accuracy requirement and the stronger requirement created by amplification of thermodynamic leakage by the penalty parameter. If \(p_k=0\), this reduces to the unconstrained degree-of-freedom scaling
\[
        N_{k(h)}\gtrsim a_h^{-N/p} ,
\]
see eq.~(\ref{9cJTrB}).}\hfill\(\square\)
\end{example}

\input{Example}

\section{Conclusions}

In this work, we have developed a DDCM framework for porous media flow that explicitly accounts for the second law of thermodynamics. The primary technical contribution is the proof that, although the set of local thermodynamically admissible states $\mathcal{C}$ is non-convex and lacks weak closedness in $L^2(\Omega;\mathbb{R}^N) \times L^2(\Omega;\mathbb{R}^N)$, the combined feasible set $\mathcal{E} \cap \mathcal{C}$---which couples physical differential constraints (equilibrium and compatibility) with thermodynamic dissipation---is weakly sequentially closed. This fundamental property, proved via a compensated compactness argument (the div-curl lemma), guarantees the existence of thermodynamically consistent solutions for continuous DDCM problems.

When trying to design a standard Lagrange multiplier approach to impose the thermodynamic inequality constraint, the space of Lagrange multipliers is not compact. By reframing the problem as a zero-sum game with bounded multipliers, we have shown that exact constraint enforcement is equivalent to a positive-part penalty formulation. We proved that the penalized functionals $\mathcal{J}_R$ are weakly lower semicontinuous, equicoercive, and $\Gamma$-converge to the strictly constrained functional $\mathcal{J}_{\mathcal{C}}$ as $R \to +\infty$. Moreover, when combining data-set approximation, space approximation and penalization, we provided precise balancing rules for spatial finite-element mesh sizes, material data scales, and penalty parameters to achieve optimal asymptotic error convergence.

A numerical test case involving a manufactured solution and perturbed material data confirms the theoretical assertions. The penalty formulation effectively eliminates local thermodynamic admissibility violations ($e \cdot s > 0$) as $R$ grows, recovering smooth, physical gradient and flux fields even when the provided material data sets are severely corrupted by artificial noise or localized non-physical artifacts. The developed framework offers a robust blueprint for incorporating thermodynamic constraints into broader data-driven computational paradigms for complex multi-physics coupled continuum systems.

\section*{Acknowledgments}

RC gratefully acknowledges support from the ICREA Acad\`emia programme of the Generalitat de Catalunya.

CGG gratefully acknowledges the financial support from the European Research Council through the ERC Consolidator Grant “DATA-DRIVEN OFFSHORE” (Project ID 101083157).

MO gratefully acknowledges the financial support of the \emph{Centre Internacional de M\`etodes Num\`erics a l'Enginyeria} (CIMNE) of the \emph{Universitat Polit\`ecnica de Catalunya} (UPC), Spain, through the \emph{UNESCO Chair in Numerical Methods in Engineering}.

\begin{appendix} 

\section{Basic facts of the calculus of variations} \label{fCnmub} 

For completeness, we recall the basic direct-method tools used in the paper.  

We begin with existence of minimizers. Let \((X,d)\) be a metric space and let
\[
        F:X\to \mathbb R\cup\{+\infty\}.
\]

\begin{definition}[Lower semicontinuity]
The functional \(F\) is lower semicontinuous at \(x\in X\) if, for every
sequence \(x_j\to x\) in \(X\),
\[
        F(x)\le \liminf_{j\to\infty}F(x_j).
\]
It is lower semicontinuous on \(X\) if it is lower semicontinuous at every
\(x\in X\).
\end{definition}

\begin{definition}[Coercivity]
The functional \(F\) is coercive if its sublevel sets are bounded; equivalently,
for every \(C<+\infty\),
\[
        \{x\in X:F(x)\le C\}
\]
is bounded in \(X\).  In a normed space, this is equivalent to
\[
        \|x\|\to+\infty
        \quad\Longrightarrow\quad
        F(x)\to+\infty .
\]
\end{definition}

\begin{theorem}[Tonelli's existence theorem] \label{TepXTQ}
Let \(X\) be a reflexive Banach space, and let
\[
        F:X\to \mathbb R\cup\{+\infty\}
\]
be proper, bounded from below, coercive, and sequentially weakly lower semicontinuous.  Then \(F\) attains its minimum: there exists \(x_\ast\in X\) such that
\[
        F(x_\ast)=\inf_{x\in X}F(x).
\]
\end{theorem}

\begin{proof}
Let \((x_j)\subset X\) be a minimizing sequence, so that
\[
        F(x_j)\to \inf_X F .
\]
Since \(F\) is coercive, \((x_j)\) is bounded in \(X\).  By reflexivity, there
exist a subsequence, not relabeled, and \(x_\ast\in X\) such that
\[
        x_j\rightharpoonup x_\ast
        \quad\text{weakly in }X .
\]
By sequential weak lower semicontinuity,
\[
        F(x_\ast)
        \le
        \liminf_{j\to\infty}F(x_j)
        =
        \inf_X F .
\]
Hence \(F(x_\ast)=\inf_X F\), and \(x_\ast\) is a minimizer.
\end{proof}

$\Gamma$-convergence is a variational form of convergence of functionals which, together with equicoercivity, ensures convergence of minimizers. We record basic definitions and facts used in the text for completeness. Further details may be consulted in \cite{DalMaso1993, Braides2006}.

\begin{definition}[\(\Gamma\)-convergence \cite{DalMaso1993, Braides2006}]
Let \((X,\tau)\) be a topological space and let \(F_h:X\to\R\cup\{+\infty\}\).
We say that \(F=\Gamma(\tau)\!-\!\lim_{h\to\infty}F_h\) if
\[
        F(x)
        =
        \sup_{U\in\mathcal N(x)}
        \liminf_{h\to\infty}\inf_{y\in U}F_h(y)
        =
        \sup_{U\in\mathcal N(x)}
        \limsup_{h\to\infty}\inf_{y\in U}F_h(y),
\]
where \(\mathcal N(x)\) is the family of all open neighborhoods of \(x\).
\end{definition}

On metric spaces, we have the following sequential characterization.

\begin{lemma}[Sequential characterization on bounded sets {\cite{CMO2018}}] \label{Sur2ng}
Let \((X,\tau)\) be a topological space. Assume that there exists a metric \(d_X\) on \(X\) such that, on bounded sets, the topology induced by \(d_X\) agrees with \(\tau\). Assume also that
\[
        X_C:=\{x\in X:F_h(x)<C\text{ for some }h\}
\]
is bounded for every \(C\in\R\). Then, \(F=\Gamma(\tau)\!-\!\lim F_h\) is
equivalent to the two sequential conditions:
\[
        F(x)\le \liminf_{h\to\infty}F_h(x_h)
        \quad\text{whenever }x_h\to x,
\]
and, for every \(x\in X\), there exists \(x_h\to x\) such that
\[
        \limsup_{h\to\infty}F_h(x_h)\le F(x).
\]
\end{lemma}

Mosco convergence combines weak and strong topologies in a manner that is ideally suited to characterize convergence with respect to data when the limiting data set is weakly sequentially lower semicontinuous. We recall that, in particular, every strongly closed convex set is weakly sequentially closed. 

\begin{definition}[Mosco convergence of functionals {\cite{Mosco1969, CMO2018}}] \label{ufxehj}
Let \(X\) be a Banach space. A sequence \(F_h:X\to\R\cup\{+\infty\}\) converges to \(F:X\to\R\cup\{+\infty\}\) in the sense of Mosco, written
\[
        F=M\hbox{-}\lim_{h\to\infty}F_h,
\]
if:
\begin{enumerate}
\item for every sequence \(x_h\rightharpoonup x\) weakly in \(X\),
\[
        \liminf_{h\to\infty}F_h(x_h)\ge F(x);
\]
\item for every \(x\in X\), there exists \(x_h\to x\) strongly in \(X\) such
that
\[
        \lim_{h\to\infty}F_h(x_h)=F(x).
\]
\end{enumerate}
\end{definition}

\begin{definition}[Mosco convergence of sets {\cite{Mosco1969, CMO2018}}] \label{XjU2TR}
A sequence of subsets \(\E_h\subset X\) converges to \(\E\subset X\) in the sense of Mosco, written
\[
        \E=M\hbox{-}\lim_{h\to\infty}\E_h,
\]
if
\[
        \I_\E=M\hbox{-}\lim_{h\to\infty} \I_{\E_h}.
\]
\end{definition}

\begin{lemma}[Weak closedness of Mosco limits {\cite{CMO2018}}]
Every Mosco limit functional is weakly sequentially lower semicontinuous.
In particular, every Mosco limit set is weakly sequentially closed.
Moreover, \(\E=M\hbox{-}\lim \E\) if and only if \(\E\) is weakly sequentially
closed.
\end{lemma}

We recall the basic compactness consequence of equicoercivity and \(\Gamma\)-convergence in metric spaces.  

\begin{definition}[Equicoercivity in metric spaces]
Let \((X,d)\) be a metric space and let \(F_h:X\to\mathbb R\cup\{+\infty\}\). The family \((F_h)\) is said to be equicoercive if, for every \(C<+\infty\), the union of sublevel sets
\[
        \bigcup_h \{x\in X:F_h(x)\le C\}
\]
is relatively compact in \(X\).  Equivalently, every sequence
\((x_h)\subset X\) satisfying
\[
        \sup_h F_h(x_h)<+\infty
\]
has a subsequence converging in the metric \(d\).
\end{definition}

\begin{theorem}[Convergence of minimizers and quasi-minimizers] \label{K87zb7}
Assume that \(F_h\) \(\Gamma\)-converges to \(F\) with respect to the metric topology of \(X\), and that the family \((F_h)\) is equicoercive.  Let
\((x_h)\subset X\) satisfy
\[
        F_h(x_h)\le \inf_X F_h+\varepsilon_h,
        \quad
        \varepsilon_h\to0 .
\]
Then \((x_h)\) is relatively compact in \(X\).  Every cluster point of
\((x_h)\) is a minimizer of \(F\).  Moreover,
\[
        \lim_{h\to\infty}\inf_X F_h=\min_X F,
\]
provided the right-hand side is finite.  In particular, if \(F\) has a unique
minimizer \(x\), then
\[
        x_h\to x
        \quad\text{in }X .
\]
\end{theorem}

\begin{proof}
Equicoercivity gives relative compactness of \((x_h)\) whenever the energies
\(F_h(x_h)\) are uniformly bounded.  Let \(x_{h_j}\to x\) be a convergent
subsequence.  By the \(\Gamma\)-liminf inequality,
\[
        F(x)\le \liminf_{j\to\infty}F_{h_j}(x_{h_j}).
\]
On the other hand, for every \(y\in X\), the \(\Gamma\)-limsup inequality gives
a recovery sequence \(y_h\to y\) such that
\[
        \limsup_{h\to\infty}F_h(y_h)\le F(y).
\]
Therefore
\[
        \limsup_{h\to\infty}\inf_X F_h\le F(y).
\]
Taking the infimum over \(y\in X\) yields
\[
        \limsup_{h\to\infty}\inf_X F_h\le \inf_X F .
\]
Combining this with quasi-minimality and the \(\Gamma\)-liminf inequality gives
\[
        F(x)
        \le
        \liminf_{j\to\infty}F_{h_j}(x_{h_j})
        \le
        \limsup_{j\to\infty}\inf_X F_{h_j}
        \le
        \inf_X F .
\]
Hence \(x\) minimizes \(F\), and the minimum values converge.  If the minimizer
of \(F\) is unique, all convergent subsequences have the same limit, and
relative compactness implies convergence of the full sequence.
\end{proof}

\end{appendix}

\bibliographystyle{plain}

\bibliography{bibliography}

\end{document}

%% file: Example.tex
\section{Numerical example}

In this section, we assess the effectiveness of the proposed penalty-based data-driven formulation through a manufactured-solution test. The objective is to verify the implementation and investigate the ability of the penalty formulation to recover thermodynamically admissible solutions. We consider the setting of Section~\ref{sec:single-el}, in which the data set consists of a single element \((\tilde e, \tilde s) \in L^2(\Omega;\R^N) \times L^2(\Omega;\R^N)\)  which may be thermodynamically inconsistent.

Let \(\lbrace\mathcal T_h\rbrace_{h>0}\) be a family of conforming, quasi-uniform, shape-regular finite element meshes of the computational domain \(\Omega\), with characteristic mesh size \(h\). Let \(\mathbb{P}_k(K)\) denote the space of polynomials of degree \(k\) on an element \(K \in\mathcal T_h\). Note that in this section $h$ refers to the mesh size and $k$ to the polynomial order, as it is usual in the finite element literature, whereas in the rest of the paper $k\to \infty$ denoted the space approximation parameter and $h \to \infty$ the parameter controlling the approximation of the data set. 

We define the finite-element spaces associated to \(\mathcal T_h\) by
\[
\text{CG}_k(\mathcal T_h) = \left\lbrace v \in C^0(\Omega) : v|_K\in\mathbb{P}_k(K)~\forall~K\in\mathcal T_h\right\rbrace
\]
and
\[
\text{DG}_k(\mathcal T_h) = \left\lbrace v \in L^2(\Omega): v|_K\in\mathbb{P}_k(K)~\forall~K\in\mathcal T_h \right\rbrace.
\]
The continuous product space is
\[
Y := H^1_0(\Omega)\times  L^2(\Omega;\R^N )\times L^2(\Omega;\R^N)\times H^1_0(\Omega)\times L^2(\Omega;\R^N),
\]
where we seek 
\[
(u, e, s, \lambda, \mu) \in Y.
\]
We approximate this continuous product space by the finite-dimensional product space
\[Y_h :=
\left(\text{CG}_{k+1} \times [\text{DG}_k]^N \times [\text{DG}_k]^N \times \text{CG}_{k+1} \times [\text{DG}_k]^N\right) \cap Y,  \label{eq:fe-spaces}
\]
where we seek 
\[
(u_h, e_h, s_h, \lambda_h, \mu_h) \in Y_h. 
\]
The choice of the finite element spaces \eqref{eq:fe-spaces} trivially satisfies the inf-sup conditions between the spaces of $u_h$, $\lambda_h$ and those of $e_h$, $s_h$, $\mu_h$, thus resulting in the discrete counterpart of the stability estimate~\eqref{eq:inf-sup-five}. For the possibility of using other finite element interpolations, see~\cite{BGBA-2025,CABG-2026}.

\subsection{Manufactured solution}
The first step is to construct a manufactured solution that serves as a reference and that is thermodynamically consistent by design. To this end, we consider a two-dimensional unit square, \(\Omega = (0,L)^2\) with \(L = 1\) and coordinates $(x,y)$, a unit reaction coefficient, \(\zeta = 1\), and a unit consistency parameter, \(\k = 1\). Nevertheless, to retain generality and improve readability, the analytical expressions are presented in terms of the parameters \(\zeta\) and \(\k\) rather than their numerical values.
We prescribe a hydraulic potential that satisfies homogeneous Dirichlet boundary conditions exactly. The corresponding gradient field is computed directly and the flux is then obtained through the constitutive relation \(s = -\k e\). The remaining fields are determined from the optimality conditions associated with functional $\mathcal{K}$ in \eqref{eq:5field-f}, i.e.,
\[
\inf_{(u,e,s)}\sup_{(\lambda, \mu)} \mathcal{K}(u,e,s,\lambda,\mu),
\]
where
$$\mathcal{K} (u,e,s,\lambda,\mu)= \frac{\k}{2}\|e-\tilde e\|^2+\frac{1}{2\k}\|s-\tilde s\|^2- (\nabla \lambda, s)+ ( \lambda, \zeta u-q) + ( \mu, e-\nabla u),$$
and listed as problem~\eqref{eq:five-field}.

The hydraulic potential is prescribed as
\begin{align}
u(x,y) = U_0\sin\left(\frac{\pi x}{L}\right)\sin\left(\frac{\pi y}{L}\right).\label{eq:u-manu}
\end{align}
The corresponding gradient field is
\[
e(x,y)=\nabla u(x,y)
=
\frac{U_0 \pi}{L}
\begin{pmatrix}
\cos\left(\frac{\pi x}{L}\right)\sin\left(\frac{\pi y}{L}\right)\\[0.5em]
\sin\left(\frac{\pi x}{L}\right)\cos\left(\frac{\pi y}{L}\right)
\end{pmatrix}.
\]
The flux field is obtained from Darcy's  constitutive relation as
\[
s(x,y)
= -\k e(x,y) =
-\frac{\k U_0\pi}{L}
\begin{pmatrix}
\cos\left(\frac{\pi x}{L}\right)\sin\left(\frac{\pi y}{L}\right)\\[0.5em]
\sin\left(\frac{\pi x}{L}\right)\cos\left(\frac{\pi y}{L}\right)
\end{pmatrix}
\]
and consequently
\[
e\cdot s = -\k |e|^2 \le 0.
\]
The data-driven gradient is taken as
\[
\tilde e(x,y) = e(x,y),
\]
the data-driven flux as
\[
\tilde s(x,y) = s(x,y)
\]
and the source as
\[
q(x,y) = \left(2\k\frac{\pi^2}{L^2}+\zeta\right)U_0\sin\left(\frac{\pi x}{L}\right)\sin\left(\frac{\pi y}{L}\right).
\]
For completeness, we also present the multiplier associated with the conservation law,
\[
\lambda(x,y) = 0
\]
and the multiplier corresponding to the geometric law,
\[
\mu(x,y) = 
\begin{pmatrix}
0\\[0.5em]
0
\end{pmatrix}.
\]

\subsection{Perturbed data-driven gradient and flux}

To investigate the ability of the penalized formulation to approximately enforce the thermodynamic admissibility condition, we introduce controlled perturbations into the data-driven fields. Let \(\tau>0\) denote the perturbation amplitude. Motivated by the material-data approximation scale, we consider perturbations such that \[ \tau \sim |\Omega|^{1/2}(\rho_h+t_h). \]

The data-driven gradient is perturbed as
\[
\tilde e(x,y)=e(x,y)+\tau r_e(x,y),
\]
where
\[
r_e(x,y)=
\begin{pmatrix}
\sin\left(\frac{10\pi x}{L}\right)
\sin\left(\frac{10\pi y}{L}\right)
\\[1ex]
\cos\left(\frac{10\pi x}{L}\right)
\cos\left(\frac{10\pi y}{L}\right)
\end{pmatrix}.
\]
The data-driven flux is perturbed as
\[
\tilde s(x,y)
=
s(x,y)
+\tau \k r_s(x,y)
+\tau\k \gamma b(x,y;L,\sigma) e(x,y),
\]
where
\[
r_s(x,y) =
\begin{pmatrix}
\cos\!\left(\frac{3\pi x}{L}\right)
\sin\!\left(\frac{3\pi y}{L}\right)
\\[1ex]
\sin\!\left(\frac{3\pi x}{L}\right)
\cos\!\left(\frac{3\pi y}{L}\right)
\end{pmatrix},
\]
and
\[
b(x, y; L, \sigma)
=
\exp\left(
-\frac{(\tfrac{x}{L}-\tfrac12)^2+(\tfrac{y}{L}-\tfrac12)^2}{\sigma^2}
\right)
\]
is a localized Gaussian bump function centered at the midpoint of the domain.

The parameter $\gamma\ge0$ controls the magnitude of the localized perturbation, while $\sigma>0$ determines its spatial extent. The term
\[
\tau\k\gamma b(x, y; L, \sigma) e(x, y)
\]
introduces a localized perturbation aligned with the gradient field. For sufficiently large values of $\gamma$, regions may emerge in which
\[
\tilde e\cdot\tilde s > 0,
\]
thereby violating the thermodynamic admissibility condition and activating the penalty term.

\subsection{Dimensionless measure of admissibility violation}

To quantify the violation of the admissibility condition
$
e\cdot s \le 0,
$
we introduce the positive-part functional
\[
\mathcal{P}(e,s)
=
\int_{\Omega}
(e\cdot s)_+ \, \dd x,
\]
where
$
a_+ = \max(a,0).
$
For the constitutive relation
$
s=-\k e,
$
it follows that
\[
e\cdot s = -\k |e|^2.
\]
Consequently, the natural scale of the quantity \(e\cdot s\) is that of
\(
\k |e|^2
\).
For the manufactured solution in \eqref{eq:u-manu}, the gradient satisfies
\[
|e(x,y)|^2
=
\left(\frac{U_0\pi}{L}\right)^2
\left(
\cos^2\left(\frac{\pi x}{L}\right)\sin^2\left(\frac{\pi y}{L}\right)
+
\sin^2\left(\frac{\pi x}{L}\right)\cos^2\left(\frac{\pi y}{L}\right)
\right)
\]
and therefore
\[
0\le |e|^2 \le \left(\frac{U_0\pi}{L}\right)^2.
\]
The characteristic scale for the violation functional is then
\[
\k\left(\frac{U_0\pi}{L}\right)^2|\Omega| = \k U_0^2 \pi^2.
\]
Motivated by this observation, we define the dimensionless violation measure
\[
\hat{\mathcal{P}}(e,s)
=
\frac{1}{\k U_0^2\pi^2}
\int_{\Omega}
(e\cdot s)_+~\dd x.
\]
In the numerical implementation, the positive-part operator is replaced by
the smooth approximation
\[
(a)_+
\approx
\frac12
\left(
a+\sqrt{a^2+\varepsilon^2}
\right),
\]
leading to the regularized dimensionless measure
\[
\hat{\mathcal{P}}_{\varepsilon}(e,s)
=
\frac12
\int_{\Omega}
\left(
\frac{e\cdot s}{\k U_0^2\pi^2}
+
\sqrt{\left(\frac{e\cdot s}{\k U_0^2\pi^2}\right)^2+\varepsilon^2}
\right)
~\dd x
\]
with \(\varepsilon > 0\) a dimensionless regularization parameter. However, the optimization problem to be solved uses the dimensional version of this regularized violation functional, \({\mathcal{P}}_{\varepsilon}\). The problem we consider is thus
\[
\arg \inf_{(u,e,s)}\sup_{(\lambda, \mu)} \left[ \mathcal{K}(u,e,s,\lambda,\mu) + R\, {\mathcal{P}}_{\varepsilon}(e,s)\right].
\]

\subsection{Results}

Here, we investigate the ability of the proposed penalized data-driven formulation to recover thermodynamically admissible solutions from inconsistent material data. The computational domain is discretized with a uniform \(100\times100 \times 2\) triangular mesh and the polynomial degree adopted is \(k=0\), i.e.,
\[Y_h :=
\left(\text{CG}_1 \times [\text{DG}_0]^2 \times [\text{DG}_0]^2 \times \text{CG}_1 \times [\text{DG}_0]^2\right) \cap Y.  \label{eq:fe-spaces-example}
\]
As stated above, the model parameters are fixed to \(\zeta=1\) and \(\k=1\), and we take \(U_0 = 1\), while the perturbation parameters are chosen as \(\gamma=50\) and \(\sigma=0.15\), and the regularization parameter as \(\varepsilon = 0.02\).

The objective is to assess the influence of the dimensionless penalty parameter \(\hat R:= \k U_0^2\pi^2R\) on the admissibility of the computed solution. For each dimensionless perturbation amplitude \(\hat \tau : = \tfrac{L\tau}{U_0\pi}\), the prescribed data fields \((\tilde e,\tilde s)\) are constructed as indicated above and may contain regions where the admissibility condition \(\tilde e\cdot\tilde s \le 0\) is violated. The optimization problem is then solved for increasing values of the dimensionaless penalty parameter \(\hat R\), and the number of cells for which \(e\cdot s>0\) is recorded. This allows us to quantify the ability of the proposed formulation to correct inadmissible material data while remaining as close as possible to the supplied data set.

\begin{table}[h!]
    \centering
    \small
    \begin{tabular}{c c c c c c c c}
        \toprule
        & \multicolumn{7}{c}{\textbf{Dimensionless Penalty Parameter } $\hat R$} \\
        \cmidrule(lr){2-8}
        & $0$ & $1.5625$ & $3.125$ & $6.25$ & $12.5$ & $25$ & $50$ \\
        \cmidrule(lr){2-2} \cmidrule(lr){3-8}
        $\hat\tau$ & $\#(\tilde{e} \cdot \tilde{s})_+$ & \multicolumn{6}{c}{$\#(e \cdot s)_+$} \\
        \midrule
        0.0000 &    0 &  0 &  0 &  0 &  0 & 0 & 0 \\
        0.0318 & 2332 &  2 &  2 &  0 &  0 & 0 & 0 \\
        0.0637 & 3321 &  6 &  6 &  6 &  2 & 2 & 0 \\
        0.0955 & 3891 & 12 & 12 &  8 &  8 & 2 & 0 \\
        0.1273 & 4296 & 26 & 18 & 18 & 12 & 2 & 0 \\
        0.1592 & 4600 & 40 & 36 & 26 & 18 & 2 & 0 \\
        \bottomrule
    \end{tabular}
    \caption{Number of cells violating the admissibility condition across increasing values of $\hat\tau$ (rows) and penalty parameter $\hat R$ (columns). The $\hat R=0$ column reports violations in the prescribed data $(\tilde{e}, \tilde{s})$, while columns for $\hat R > 0$ report violations in the computed solution $(e, s)$.}
    \label{tab:admissibility_violations}
\end{table}

The perturbed data fields exhibit a substantial number of admissibility violations. Nevertheless, the penalized data-driven formulation successfully suppresses violations in the computed solution. As the dimensionless penalty parameter \(\hat R\) increases, the number of cells satisfying \(e\cdot s > 0\) decreases monotonically overall and eventually vanishes. This suggests that the formulation is capable of recovering thermodynamically admissible states even when the supplied material data are strongly inconsistent.

To complement the quantitative results of Table~\ref{tab:admissibility_violations}, Figures~\ref{fig:manufactured}--\ref{fig:extreme_case} illustrate the fields involved at representative parameter values. We first show the manufactured fields $e$ and $s$ that serve as the admissible reference (Figure~\ref{fig:manufactured}), followed by the perturbed data $\tilde{e}$, $\tilde{s}$ and the resulting thermodynamic corrections for increasing $\hat\tau$ and $\hat R$ (Figures~\ref{fig:perturbed_tau05}--\ref{fig:extreme_case}). Throughout, the diverging color scale runs from blue (lowest value) to red (highest value).

Figure~\ref{fig:manufactured} shows the manufactured gradient and flux fields $e$ and $s$, which are smooth and free of discretization artifacts, as expected at this frequency on the $100\times100\times2$ mesh. Consistent with Darcy's law $s=-\mathsf{k}e$, the $s$ panels are sign-reversed mirror images of the corresponding $e$ panels.

\newcommand{\admissfig}[1]{%
    \includegraphics[width=\linewidth, trim=751bp 176bp 751bp 176bp, clip]{#1}%
}

\begin{figure}[t]
    \centering
    \begin{subfigure}[b]{0.48\linewidth}
        \centering
        \admissfig{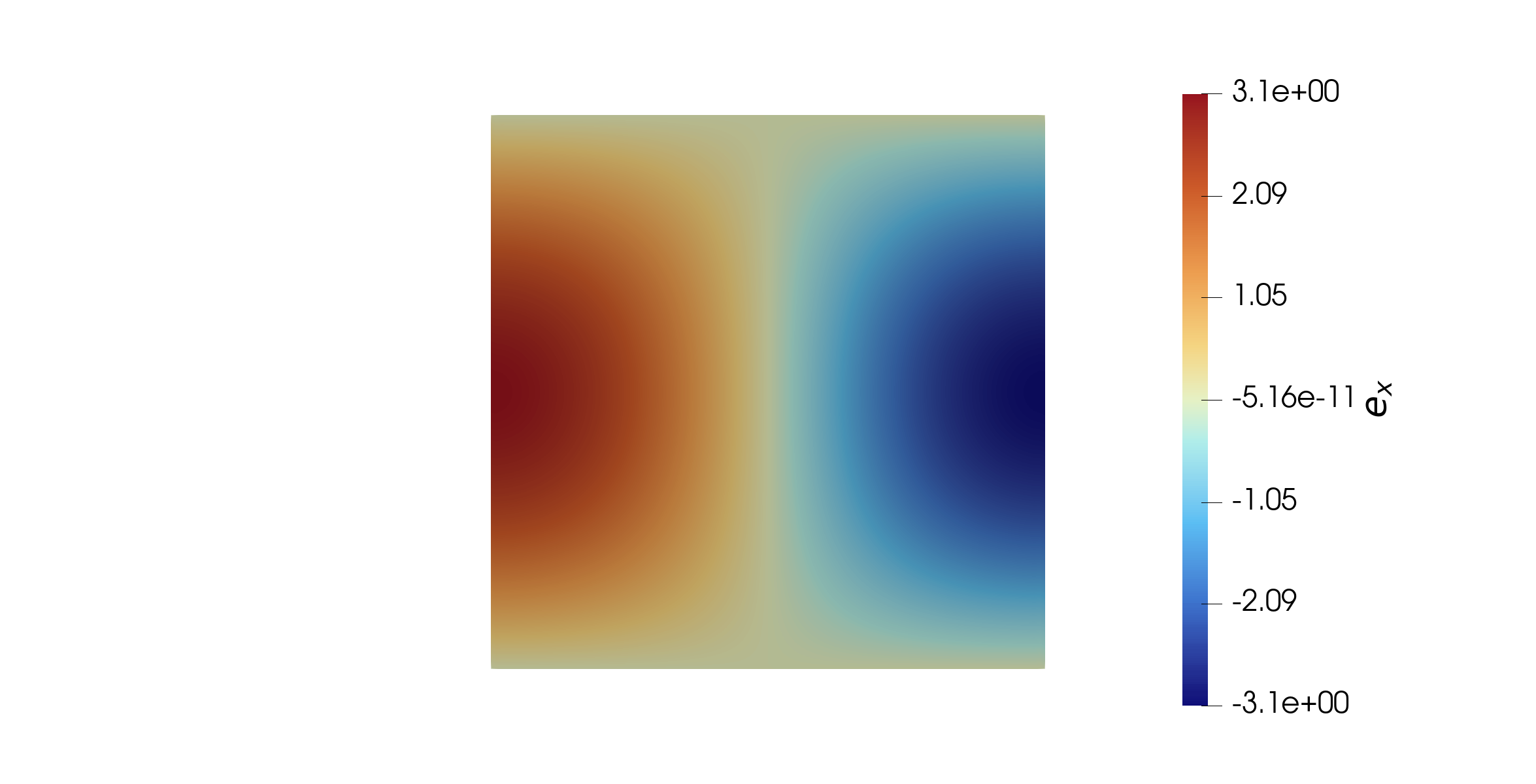}
        \caption{$e_x\in [-3.1,3.1]$.}
    \end{subfigure}
    \hfill
    \begin{subfigure}[b]{0.48\linewidth}
        \centering
        \admissfig{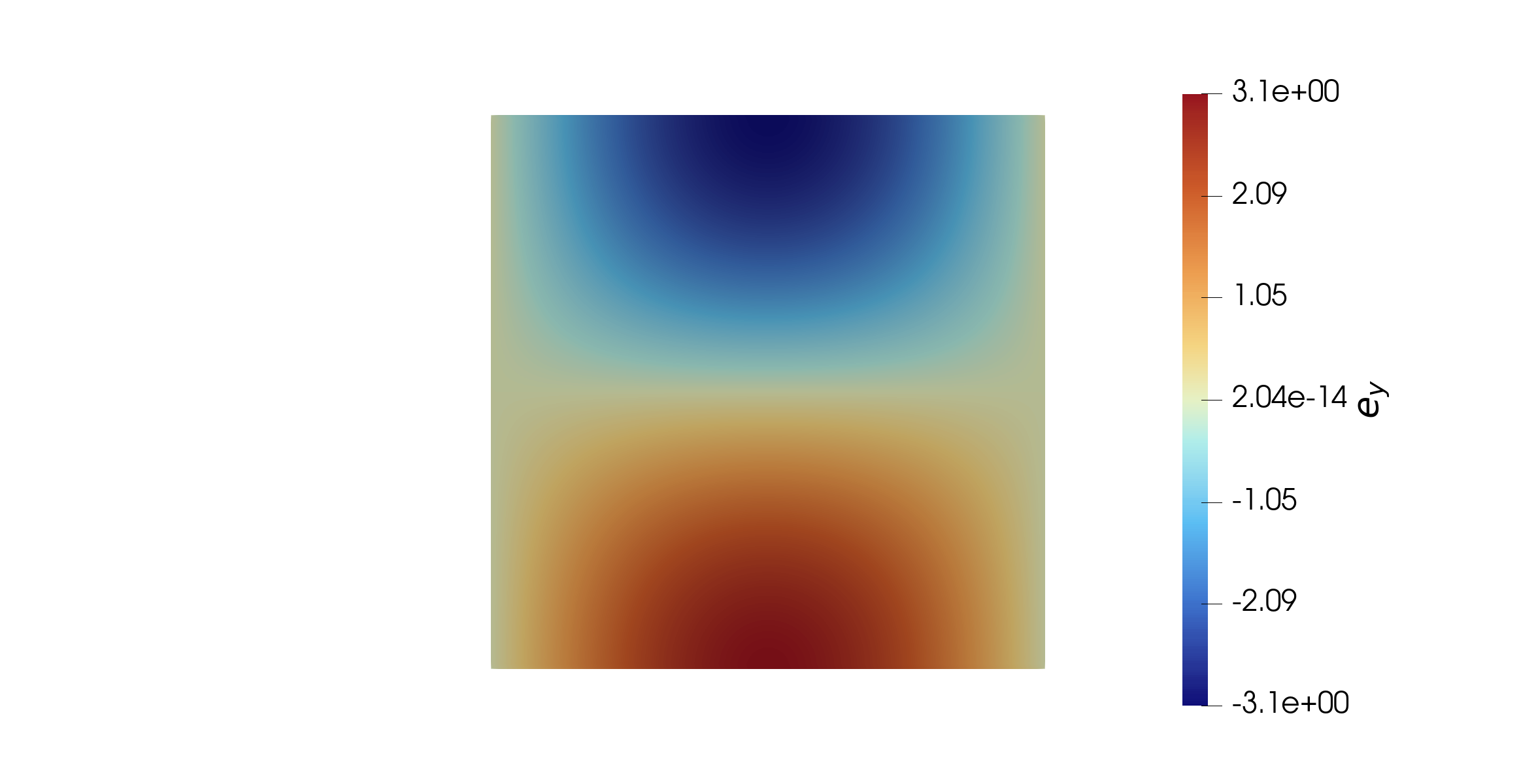}
        \caption{$e_y\in [-3.1,3.1]$.}
    \end{subfigure}
    \vspace{1em}
    \begin{subfigure}[b]{0.48\linewidth}
        \centering
        \admissfig{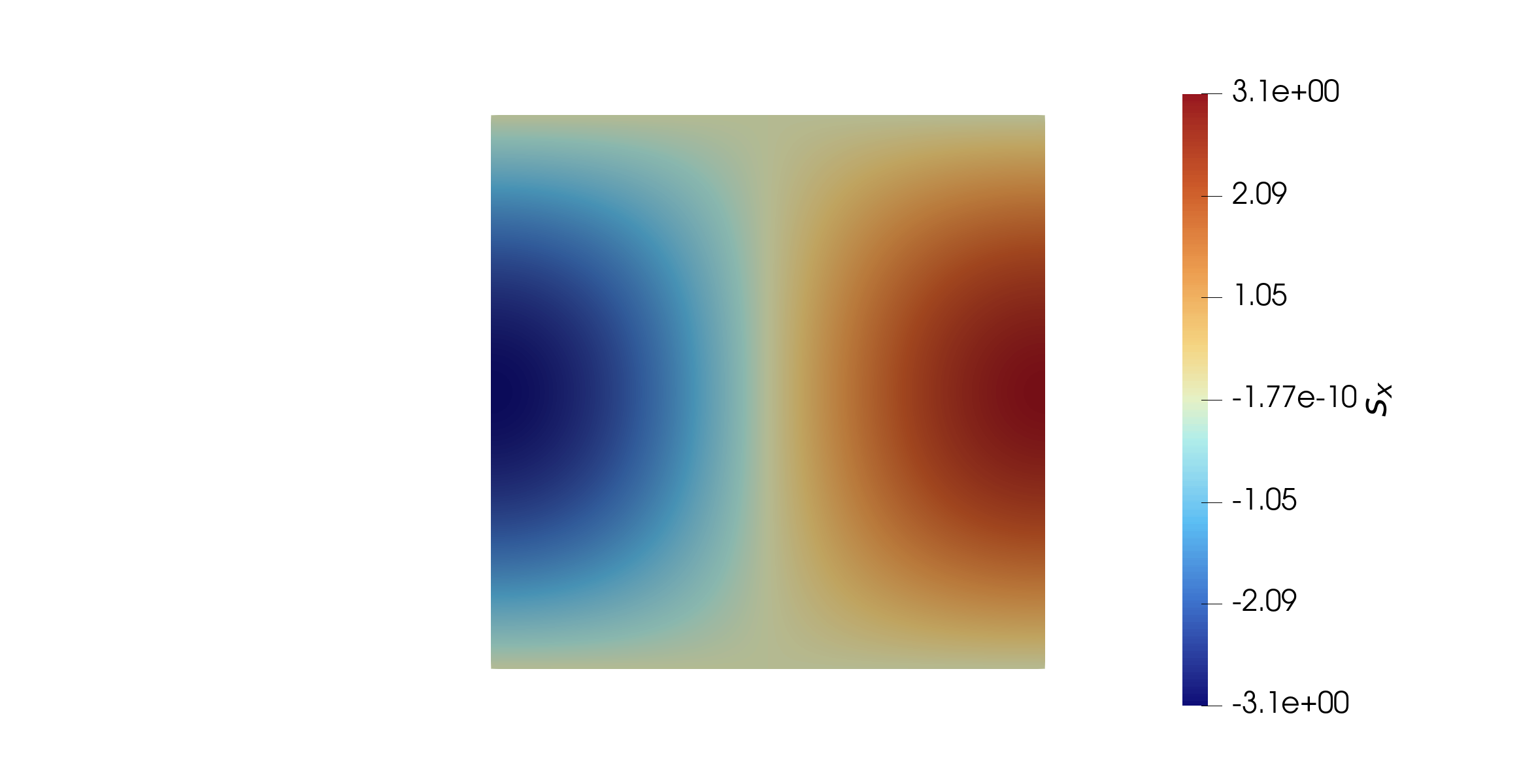}
        \caption{$s_x\in [-3.1,3.1]$.}
    \end{subfigure}
    \hfill
    \begin{subfigure}[b]{0.48\linewidth}
        \centering
        \admissfig{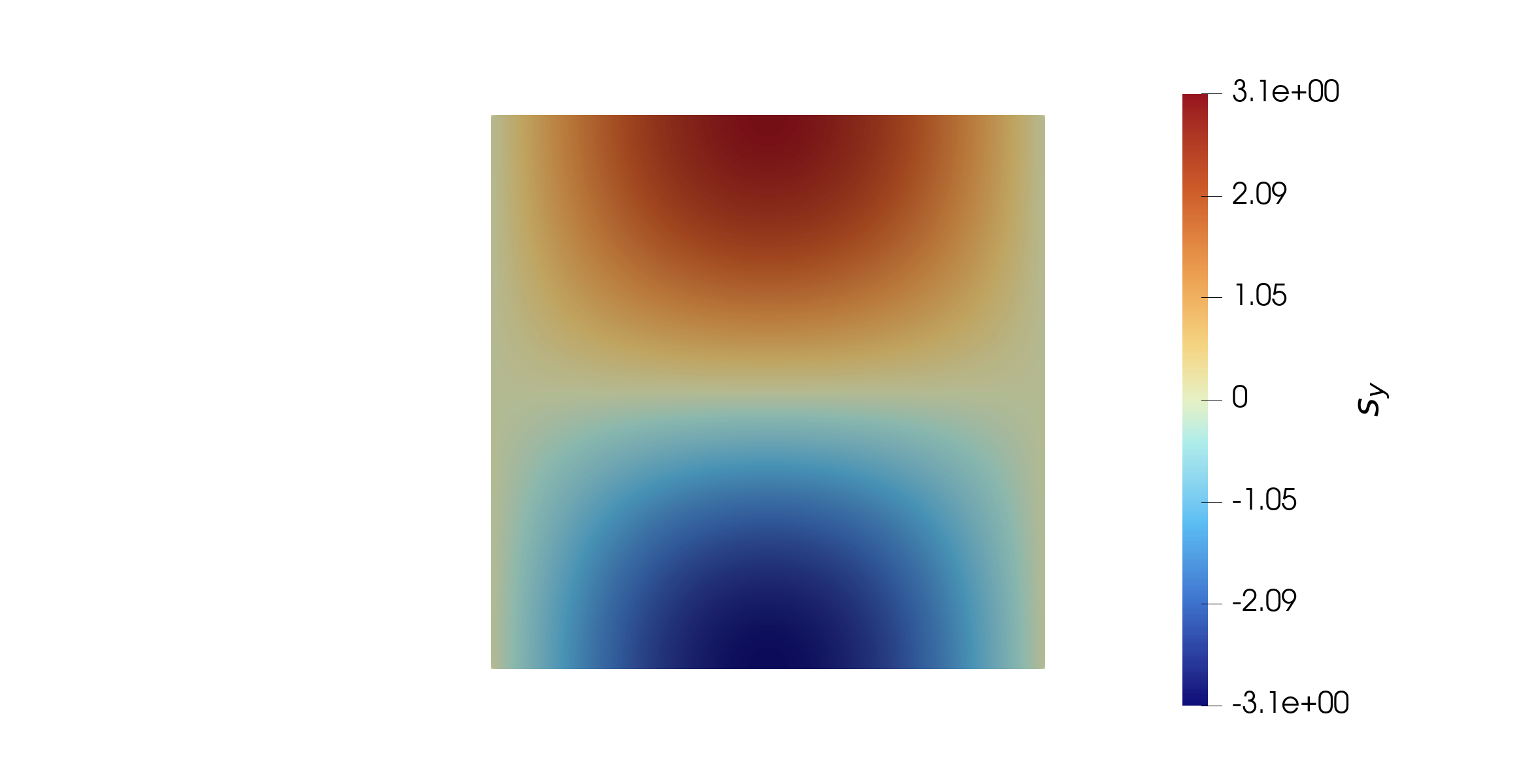}
        \caption{$s_y\in [-3.1,3.1]$.}
    \end{subfigure}
    \caption{Manufactured fields $e$ and $s$.}
    \label{fig:manufactured}
\end{figure}

Figure~\ref{fig:perturbed_tau05} shows the perturbed ``data-driven'' fields $\tilde{e}$ and $\tilde{s}$ for $\hat\tau=0.1592$. The high-frequency oscillations $r_e$ superimposed on $\tilde{e}$ are clearly visible, while $\tilde{s}$ is dominated by the localized Gaussian bump, producing a markedly different, more localized structure than the smooth manufactured fields of Figure~\ref{fig:manufactured}.

\begin{figure}[t]
    \centering
    \begin{subfigure}[b]{0.48\linewidth}
        \centering
        \admissfig{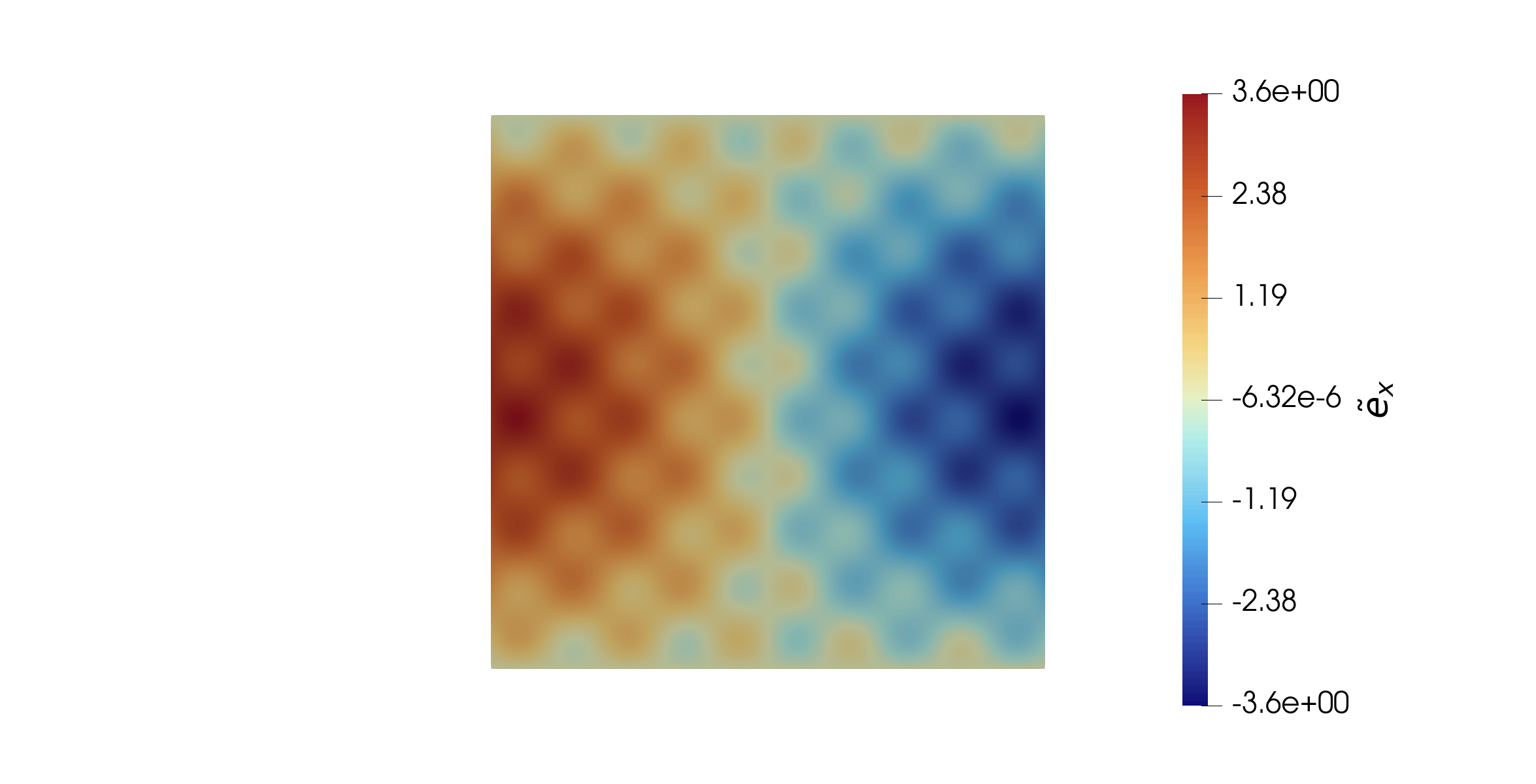}
        \caption{$\tilde{e}_x\in [-3.6,3.6]$.}
    \end{subfigure}
    \hfill
    \begin{subfigure}[b]{0.48\linewidth}
        \centering
        \admissfig{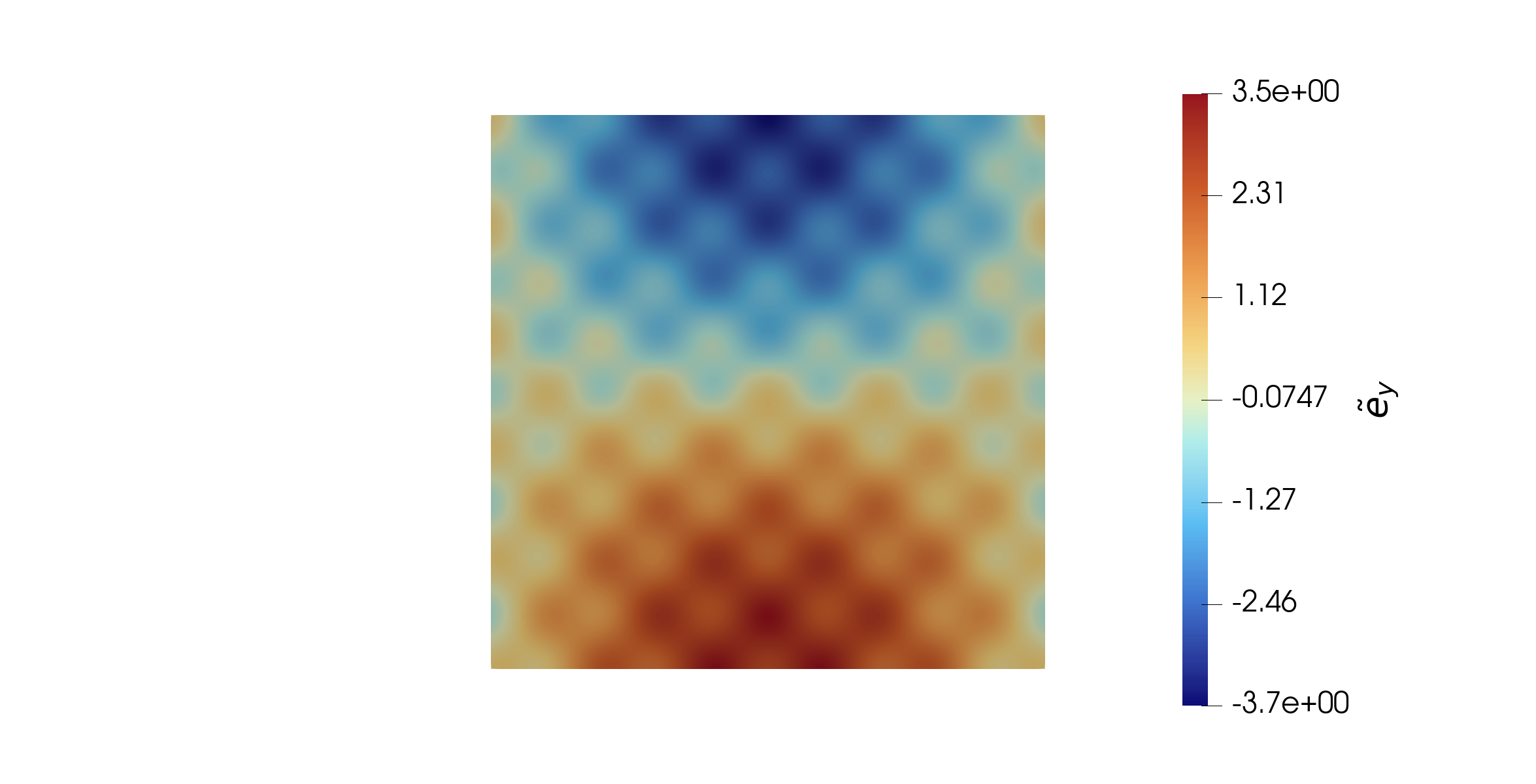}
        \caption{$\tilde{e}_y\in [-3.7,3.5]$.}
    \end{subfigure}
    \vspace{1em}
    \begin{subfigure}[b]{0.48\linewidth}
        \centering
        \admissfig{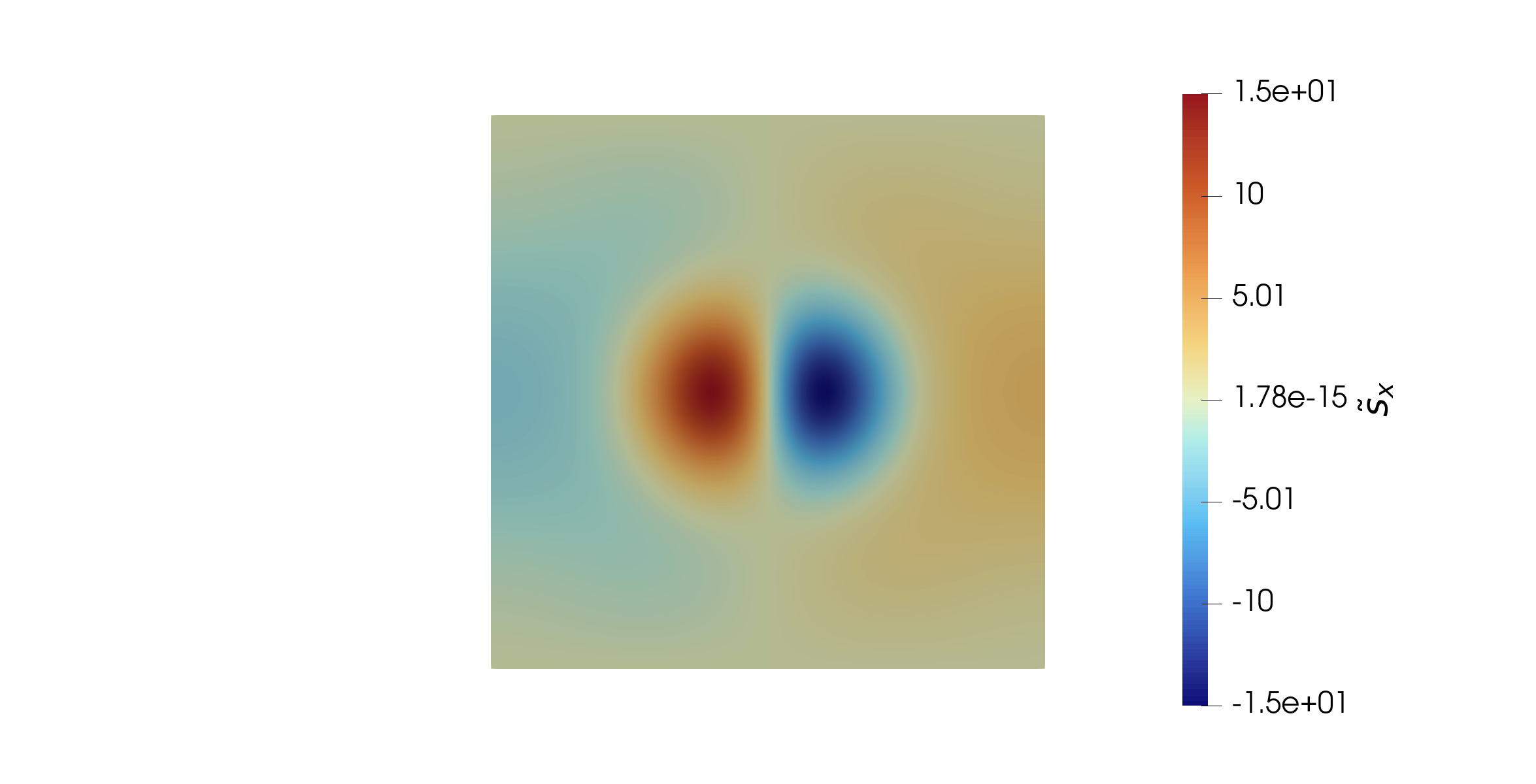}
        \caption{$\tilde{s}_x\in [-15.0,15.0]$.}
        \label{fig:violation-R12}
    \end{subfigure}
    \hfill
    \begin{subfigure}[b]{0.48\linewidth}
        \centering
        \admissfig{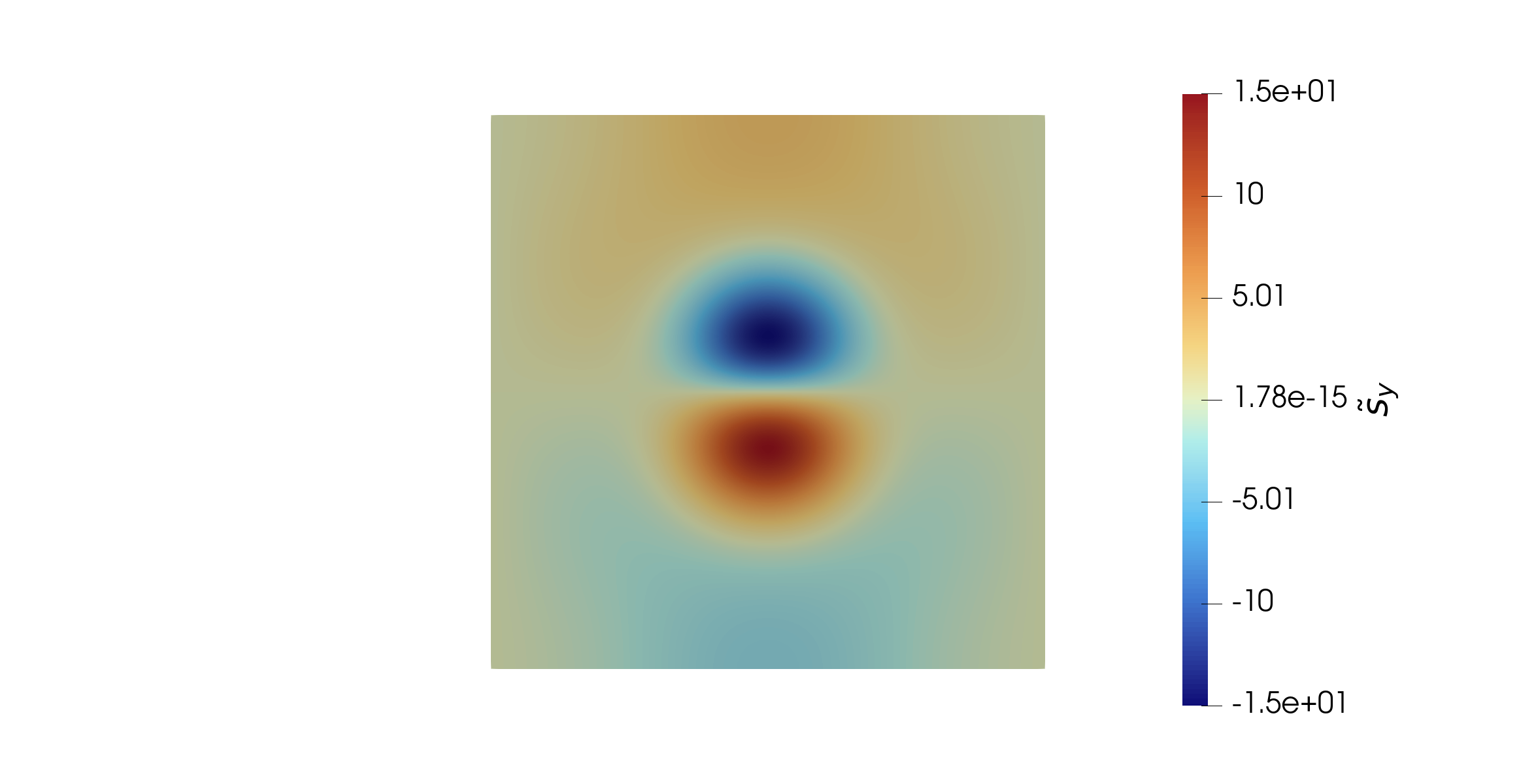}
        \caption{$\tilde{s}_y\in [-15.0,15.0]$.}
    \end{subfigure}
    \caption{Perturbed fields $\tilde e$ and $\tilde s$ for $\hat\tau = 0.1592$.}
    \label{fig:perturbed_tau05}
\end{figure}

Figure~\ref{fig:penalty_effect} shows the effect of the dimensionless penalty parameter $\hat R$ on thermodynamic admissibility for $\hat\tau=0.1592$. The raw data (A) exhibit a large connected violation region, which is progressively suppressed as $\hat R$ increases (B)--(D), completely vanishing by $\hat R=50$, confirming that the penalized formulation recovers admissible states from inconsistent material data.

\begin{figure}[t]
    \centering
    \begin{subfigure}[b]{0.48\linewidth}
        \centering
        \admissfig{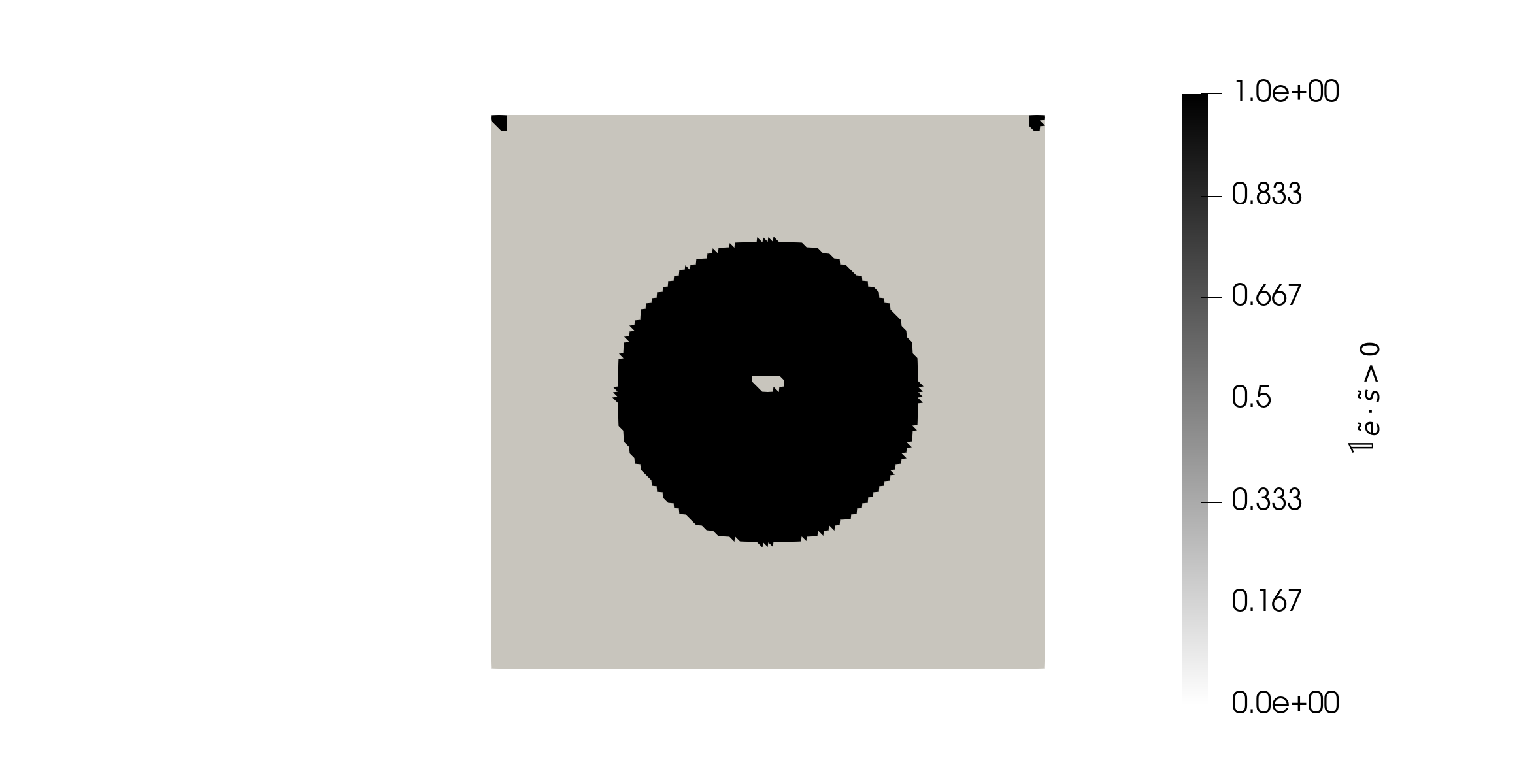}
        \caption{Raw data, $\hat\tau = 0.1592$, $\tilde{e}\cdot\tilde{s}$}
    \end{subfigure}
    \hfill
    \begin{subfigure}[b]{0.48\linewidth}
        \centering
        \admissfig{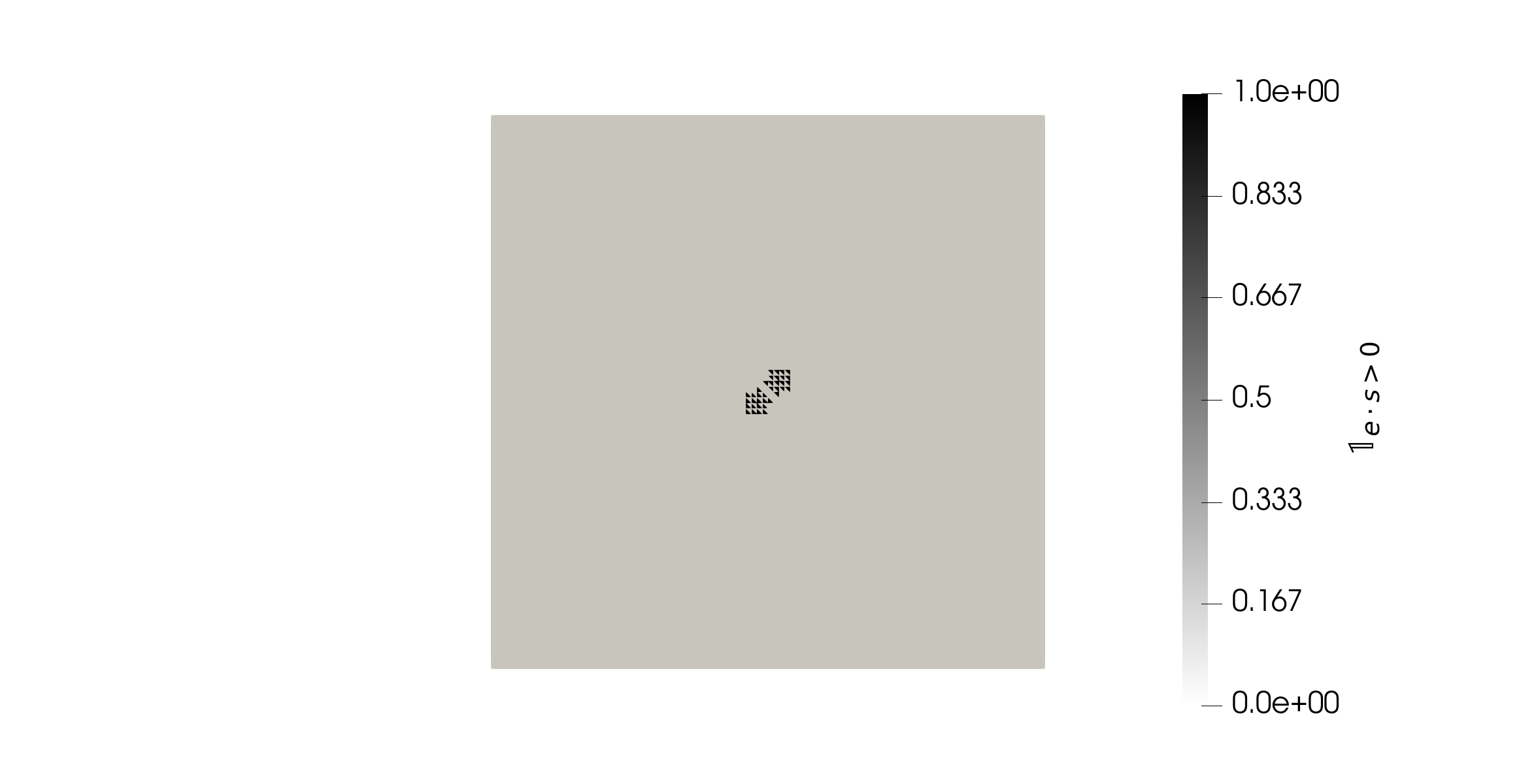}
        \caption{$\hat\tau=0.1592$, $\hat R=3.125$,  $e\cdot s$.}
    \end{subfigure}
    \vspace{1em}
    \begin{subfigure}[b]{0.48\linewidth}
        \centering
        \admissfig{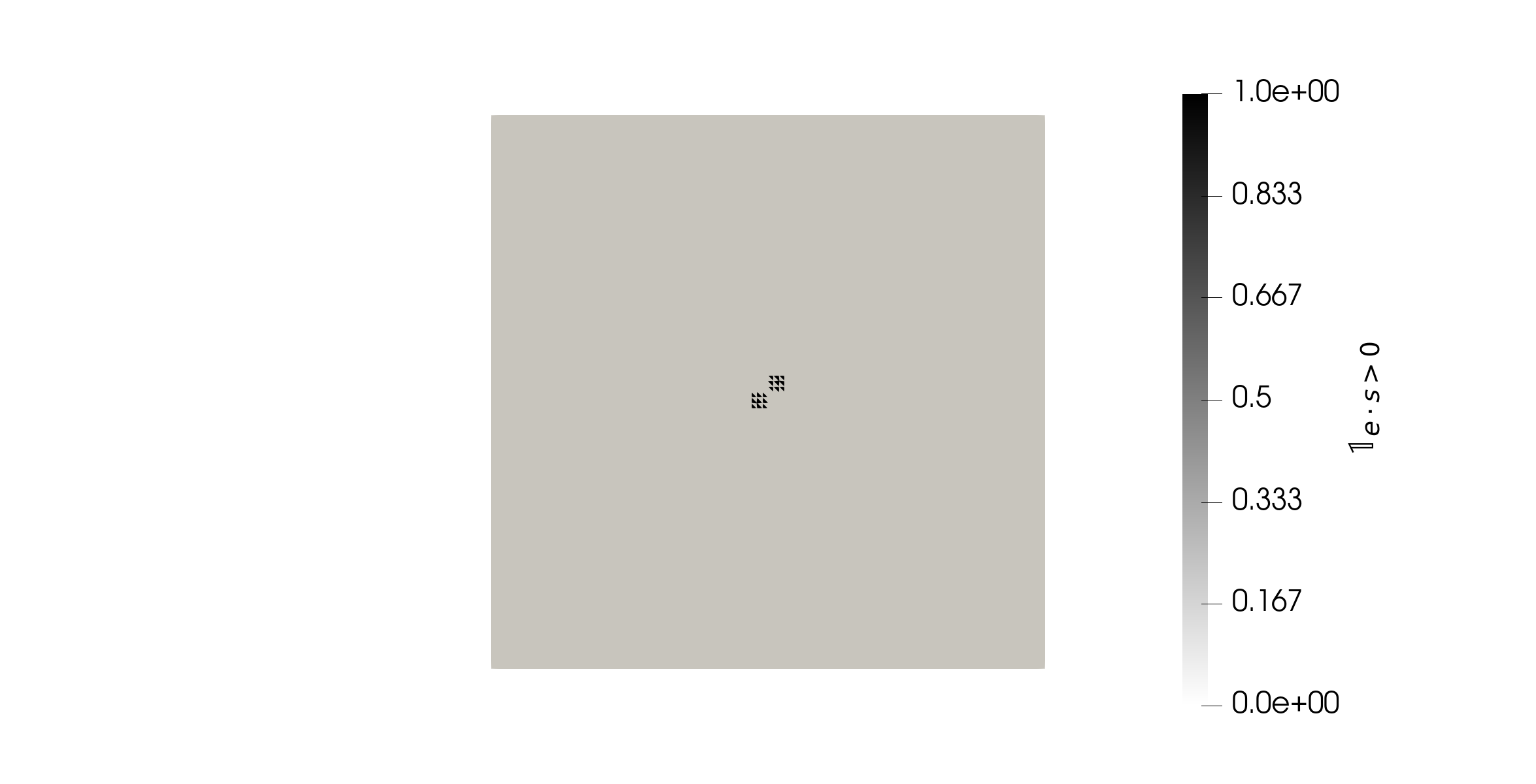}
        \caption{$\hat\tau=0.1592$, $\hat R=12.5$, $e\cdot s$.}
    \end{subfigure}
    \hfill
    \begin{subfigure}[b]{0.48\linewidth}
        \centering
        \admissfig{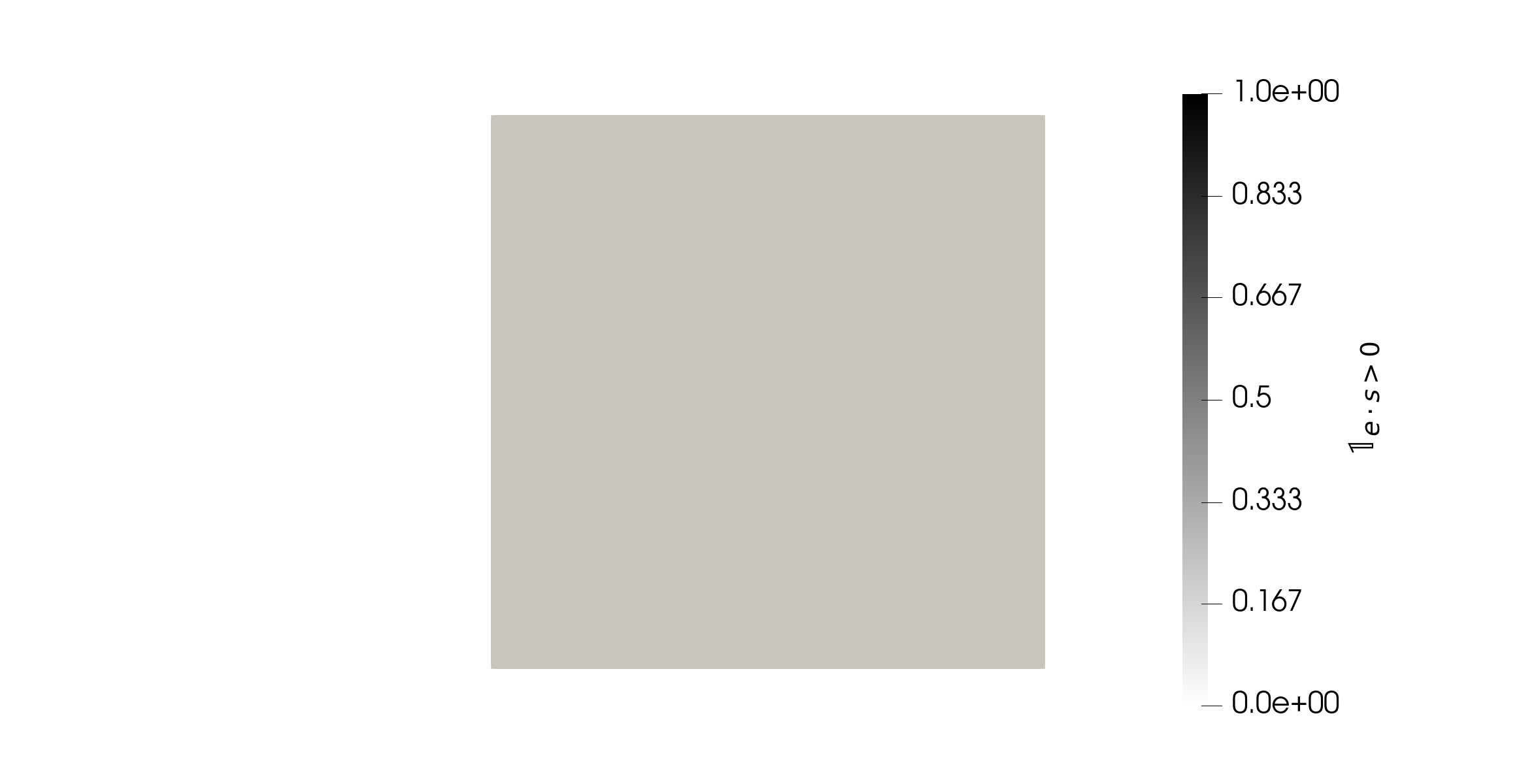}
        \caption{$\hat\tau=0.1592$, $\hat R=50$, $e\cdot s$.}
    \end{subfigure}
    \caption{Effect of the penalty parameter $\hat R$ on thermodynamic admissibility for $\hat\tau = 0.1592$.
    (A) shows the raw, inconsistent data $\tilde{e}\cdot\tilde{s}$, (B)--(D) show the solved $e\cdot s$ at increasing $\hat R$,
    with the violation region (black) shrinking as $\hat R$ grows. Grey cells
    satisfy $e\cdot s \le 0$; black cells violate it.}
    \label{fig:penalty_effect}
\end{figure}

Figure~\ref{fig:perturbed_tau5} shows the perturbed fields $\tilde{e}$ and $\tilde{s}$ for the larger perturbation amplitude $\hat\tau=1.592$. The high-frequency component of $\tilde{e}$ is now far more pronounced, while $\tilde{s}$ retains the same localized bump structure as in Figure~\ref{fig:perturbed_tau05}, but with substantially larger amplitude, reflecting the stronger admissibility violations reported in Table~\ref{tab:admissibility_violations}.

\begin{figure}[t]
    \centering
    \begin{subfigure}[b]{0.48\linewidth}
        \centering
        \admissfig{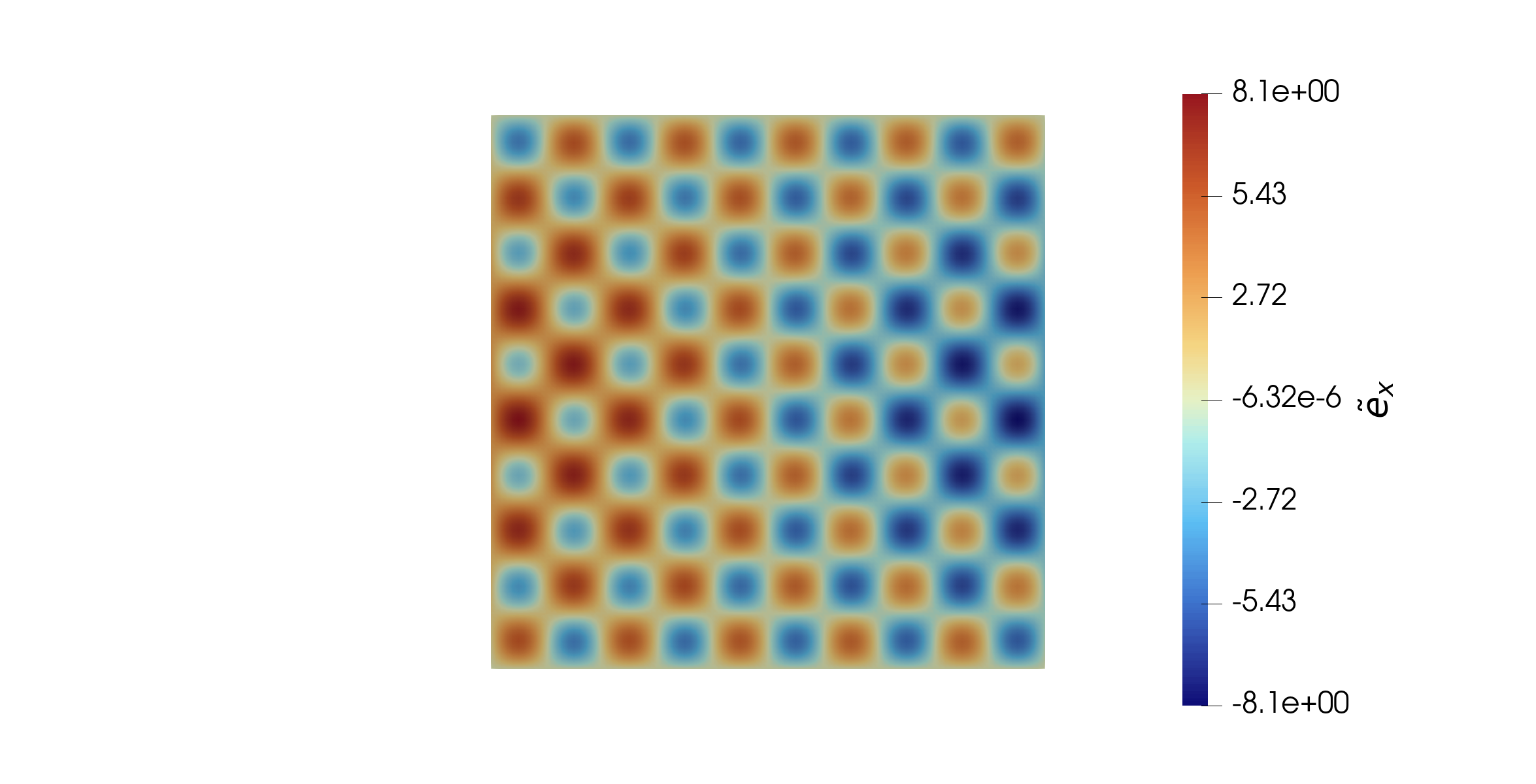}
        \caption{$\tilde{e}_x\in[-8.1,8.1]$.}
    \end{subfigure}
    \hfill
    \begin{subfigure}[b]{0.48\linewidth}
        \centering
        \admissfig{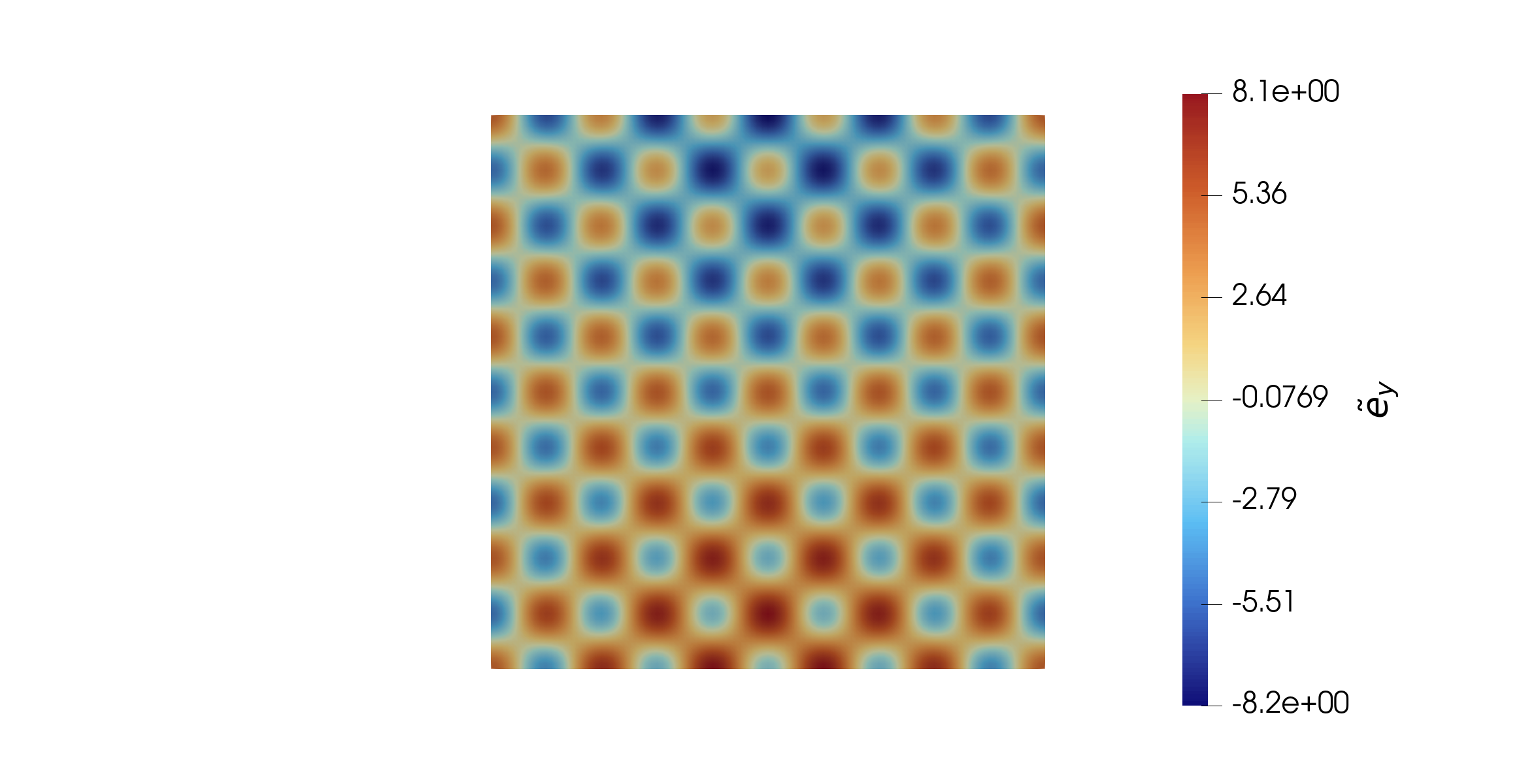}
        \caption{$\tilde{e}_y\in[-8.2,8.1]$.}
    \end{subfigure}
    \vspace{1em}
    \begin{subfigure}[b]{0.48\linewidth}
        \centering
        \admissfig{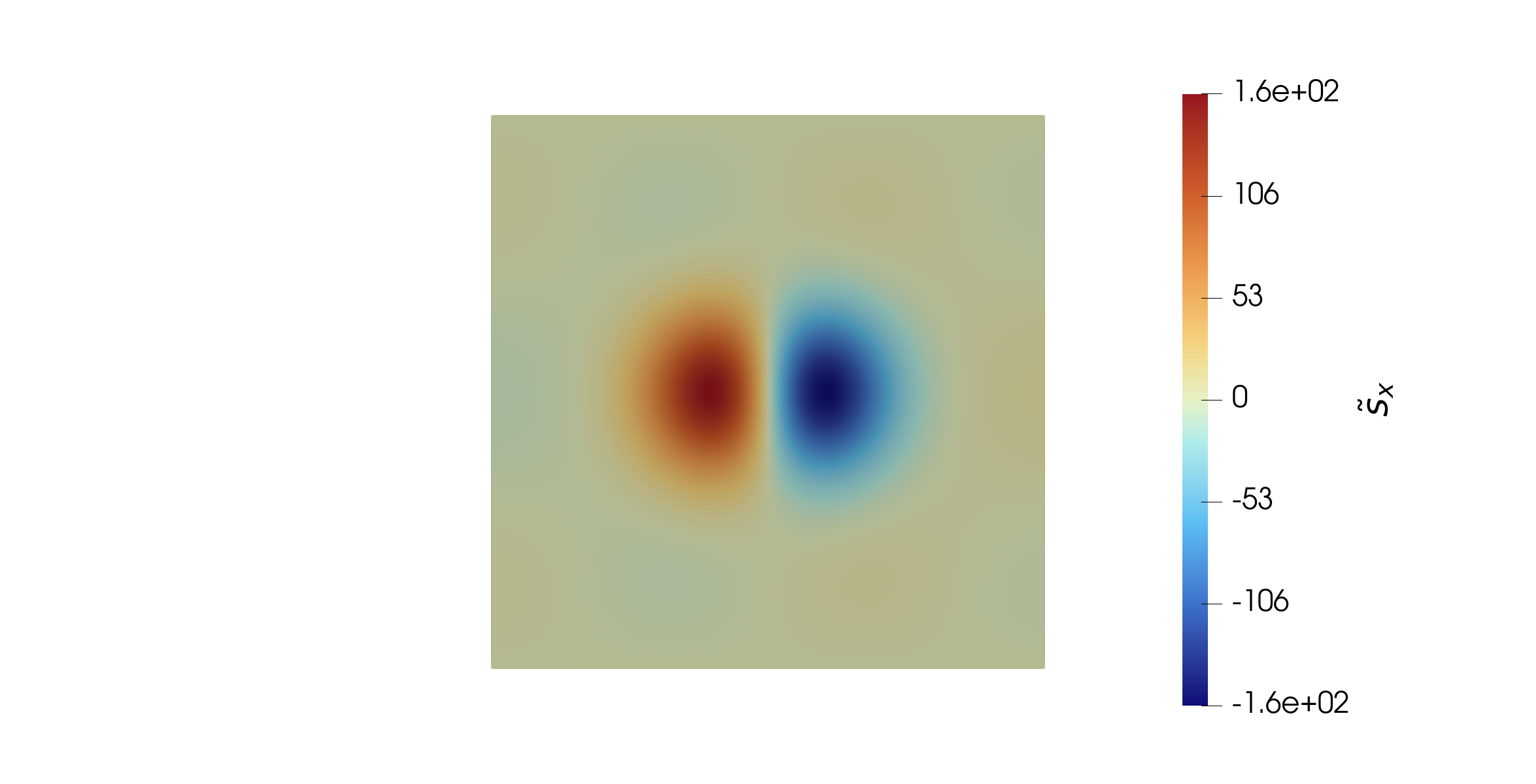}
        \caption{$\tilde{s}_x\in[-160.0,160.0]$.}
    \end{subfigure}
    \hfill
    \begin{subfigure}[b]{0.48\linewidth}
        \centering
        \admissfig{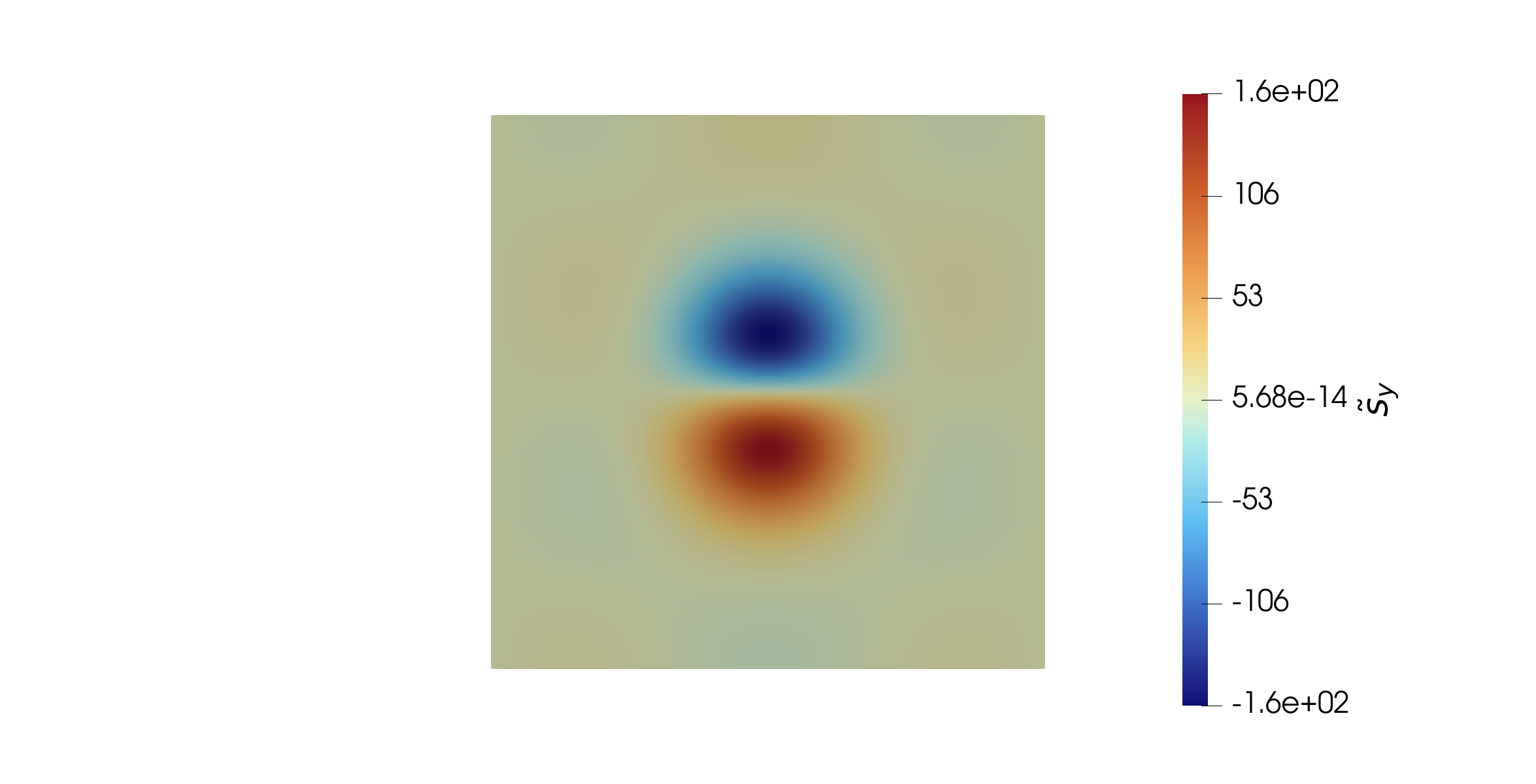}
        \caption{$\tilde{s}_y\in[-160.0,160.0]$.}
    \end{subfigure}
    \caption{Perturbed fields $\tilde e$ and $\tilde s$ for $\hat\tau = 1.592$.}
    \label{fig:perturbed_tau5}
\end{figure}

Figure~\ref{fig:computed_tau5} shows the computed fields $e$ and $s$ for $\hat\tau=1.592$ and $\hat R=15$. Away from the perturbation center the fields closely resemble the manufactured solution of Figure~\ref{fig:manufactured}, while a localized distortion emerges near the Gaussian bump, marking where the penalty actively corrects the thermodynamically inconsistent data.

\begin{figure}[t]
    \centering
    \begin{subfigure}[b]{0.48\linewidth}
        \centering
        \admissfig{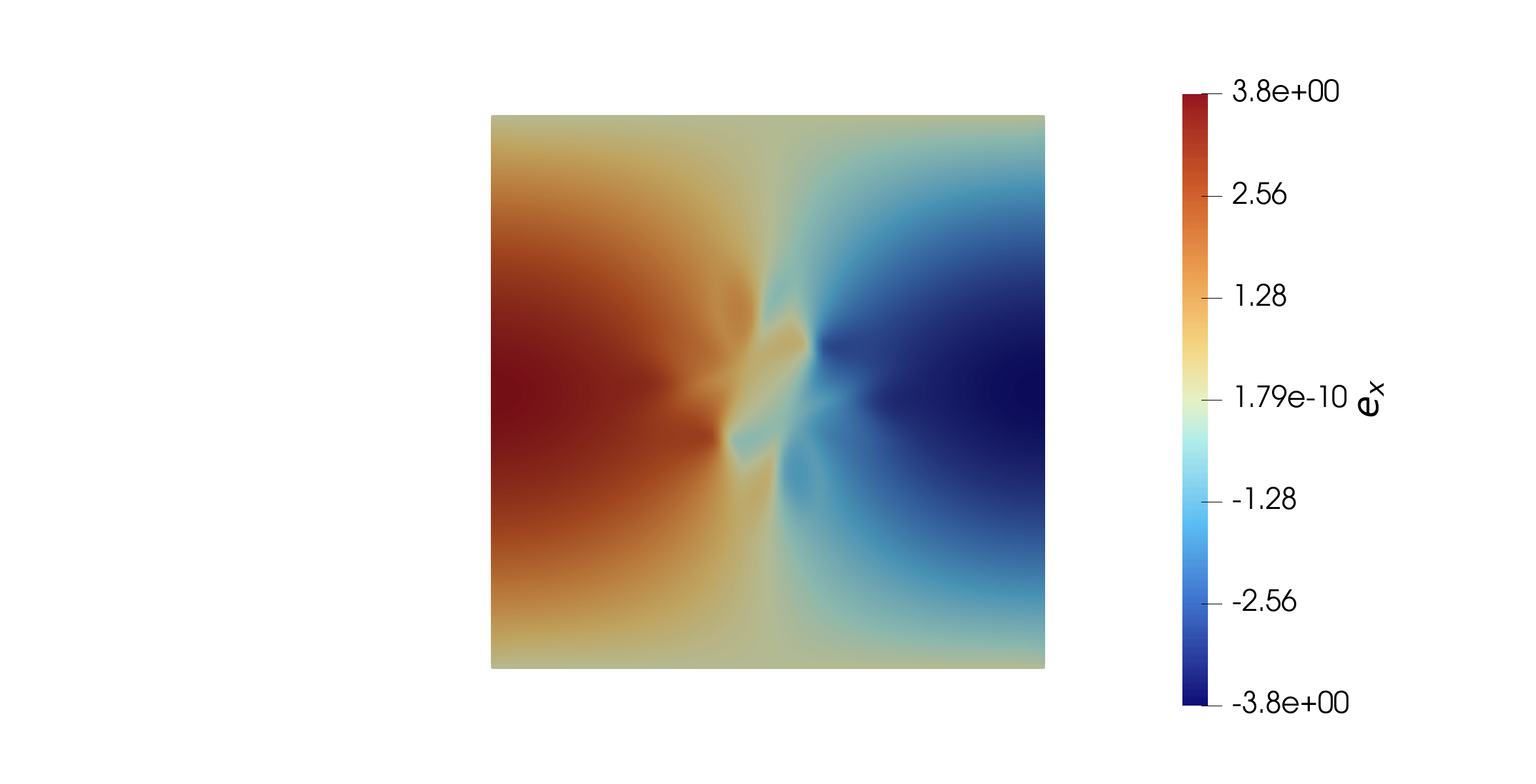}
        \caption{$e_x\in[-3.8,3.8]$.}
    \end{subfigure}
    \hfill
    \begin{subfigure}[b]{0.48\linewidth}
        \centering
        \admissfig{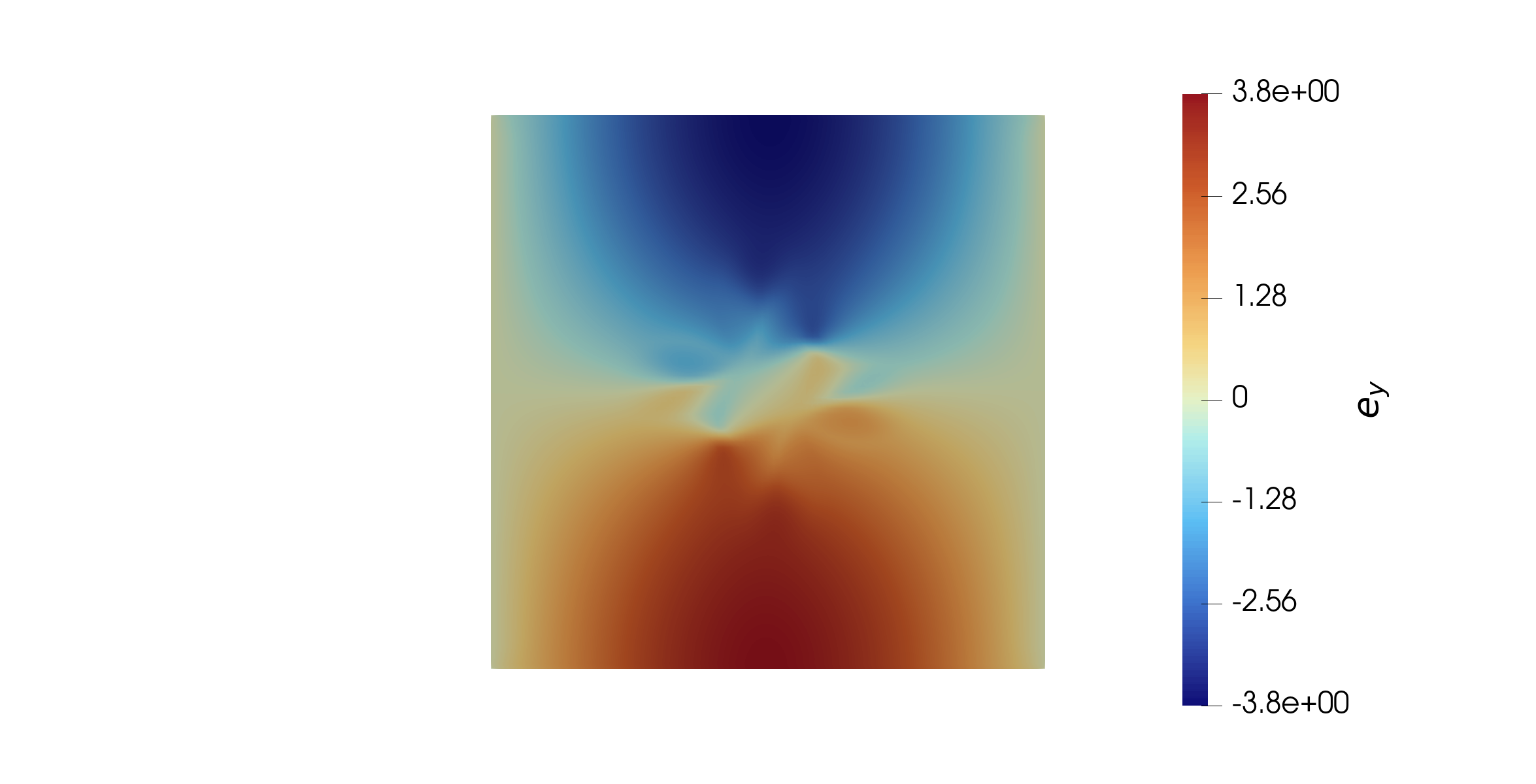}
        \caption{$e_y\in[-3.8,3.8]$.}
    \end{subfigure}

    \vspace{1em}

    \begin{subfigure}[b]{0.48\linewidth}
        \centering
        \admissfig{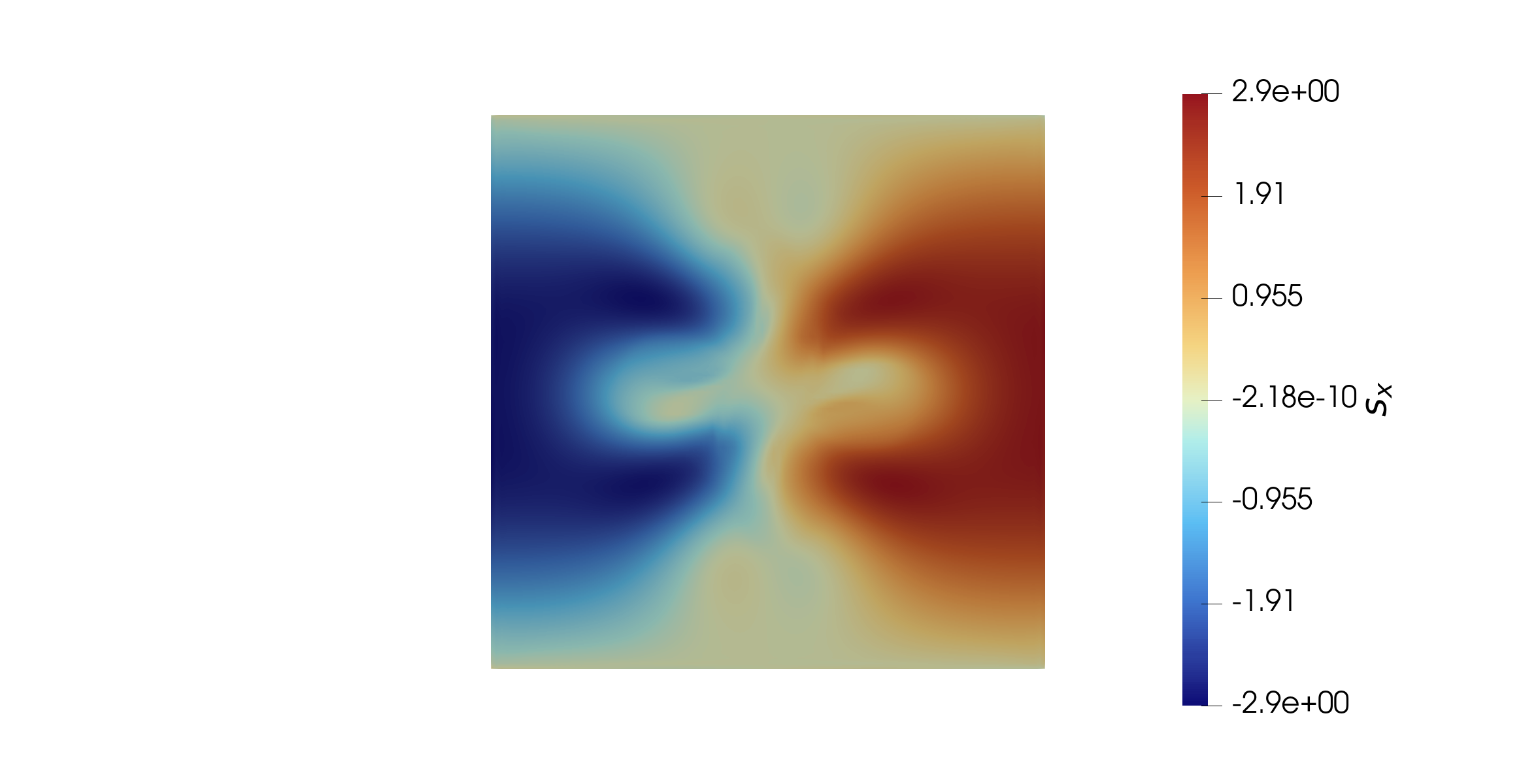}
        \caption{$s_x\in[-2.9,2.9]$.}
    \end{subfigure}
    \hfill
    \begin{subfigure}[b]{0.48\linewidth}
        \centering
        \admissfig{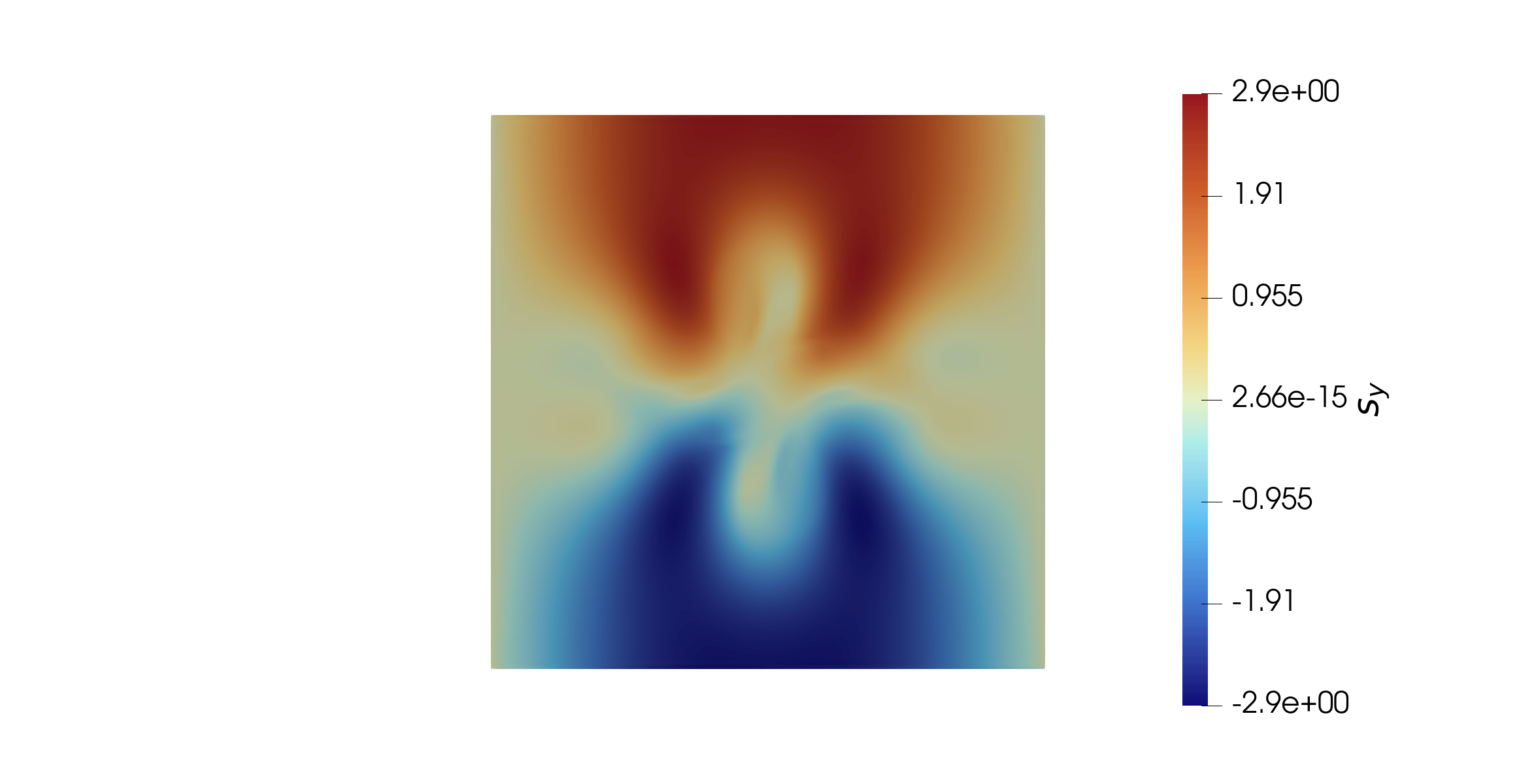}
        \caption{$s_y\in[-2.9,2.9]$.}
    \end{subfigure}
    \caption{Computed fields $e$ and $s$ for $\hat\tau = 1.592$ and $\hat R = 15$.}
    \label{fig:computed_tau5}
\end{figure}

Figure~\ref{fig:extreme_case} shows the extreme case $\hat\tau=1.592$: the raw data (A) violate admissibility over most of the domain, yet the penalized solution at $\hat R=15$ (B) confines violations to a handful of isolated cells, demonstrating that the formulation remains effective even under severe thermodynamic inconsistency.

\begin{figure}[t]
    \centering
    \begin{subfigure}[b]{0.49\linewidth}
        \centering
        \admissfig{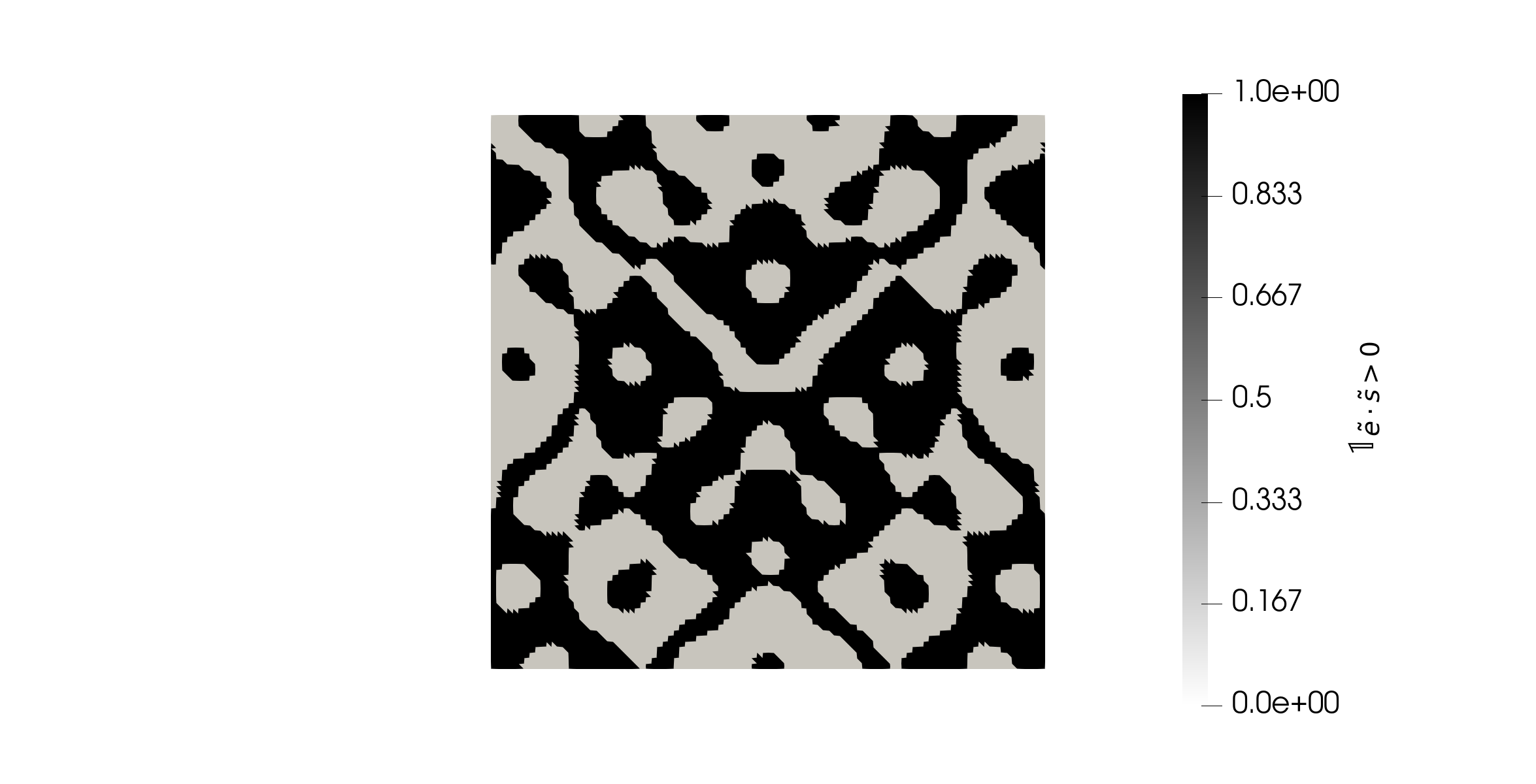}
        \caption{Raw data, $\hat\tau = 1.592$, $\tilde{e}\cdot\tilde{s}$.}
    \end{subfigure}
    \hfill
    \begin{subfigure}[b]{0.49\linewidth}
        \centering
        \admissfig{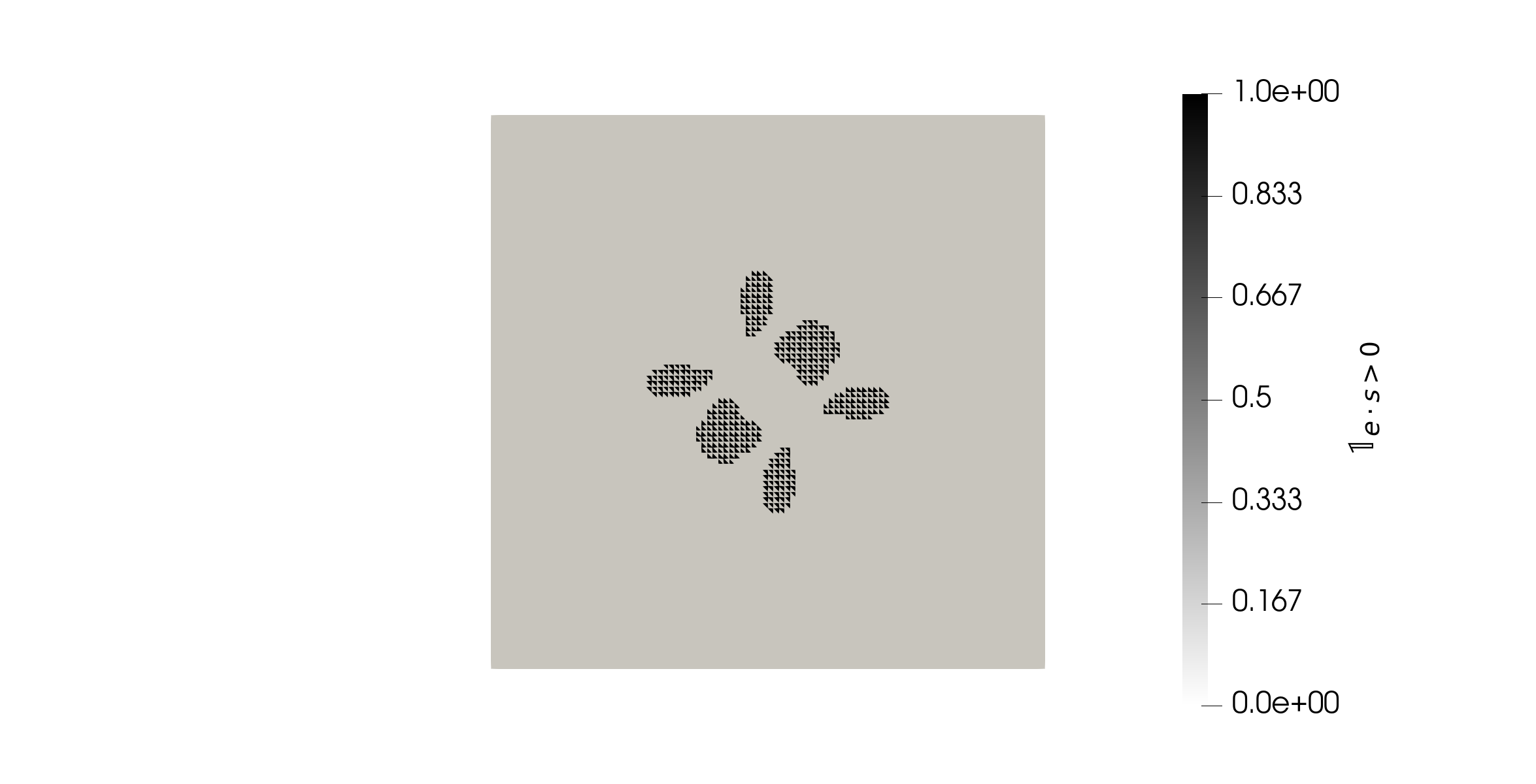}
        \caption{$\hat\tau=1.592$, $\hat R=15$, $e \cdot s$.}
    \end{subfigure}
    \caption{Extreme case showing the effectiveness of the penalty enforcement for $\hat\tau = 1.592$.
    (A) shows the raw, inconsistent data $\tilde{e}\cdot\tilde{s}$; (B) shows the solved $e\cdot s$ at $\hat R = 15$. Grey cells
    satisfy $e\cdot s \le 0$; black cells violate it.}
    \label{fig:extreme_case}
\end{figure}

\begin{remark}[Scope of the numerical verification]
The numerical example above verifies the penalty-based enforcement of the thermodynamic inequality on a manufactured solution. Numerical verification of the joint convergence results — in particular, of the balancing rules coupling the discretization, data-approximation, and penalty scales — is not addressed here and is left to future work.    
\end{remark}